\documentclass{article}
\usepackage{amssymb}
\usepackage{amsfonts}
\usepackage{amsmath}
\usepackage{geometry}

\newtheorem{theorem}{Theorem}

\newtheorem{corollary}[theorem]{Corollary}

\newtheorem{definition}[theorem]{Definition}
\newtheorem{example}[theorem]{Example}

\newtheorem{lemma}[theorem]{Lemma}

\newtheorem{proposition}[theorem]{Proposition}
\newtheorem{remark}[theorem]{Remark}

\newenvironment{proof}[1][Proof]{\noindent\textbf{#1.} }{\ \rule{0.5em}{0.5em}}
\input{tcilatex}

\begin{document}

\title{Differential\ Equations Driven by Gaussian Signals I}
\author{Peter Friz\thanks{%
Corresponding author. Department of Pure Mathematics and Mathematical
Statistics, University of Cambridge. Email: P.K.Friz@statslab.cam.ac.uk. } \
\ \thanks{%
Partially supported by a Leverhulme Research Fellowship. } \and Nicolas
Victoir}
\maketitle

\begin{abstract}
We consider multi-dimensional Gaussian processes and give a new condition on
the covariance, simple and sharp, for the existence of L\'{e}vy area(s).
Gaussian rough paths are constructed with a variety of weak and strong
approximation results. Together with a new RKHS embedding, we obtain a
powerful\ - yet conceptually simple - framework in which to analysize
differential equations driven by Gaussian signals in the rough paths sense.
\end{abstract}

\section{Introduction}

Let $X$ be a real-valued centered Gaussian process on $\left[ 0,1\right] $
with continuous sample paths and (continuous) covariance $R=R\left(
s,t\right) =\mathbb{E}\left( X_{s}X_{t}\right) $. From Kolmogorov's
criterion, it is clear that H\"{o}lder regularity of $R$ will imply H\"{o}%
lder continuity of sample paths. One can also deduce $p$-variation of sample
paths from $R$. Indeed, the condition%
\begin{equation}
\sup_{D=\left\{ t_{i}\right\} }\sum_{i}\left\vert \mathbb{E[}\left(
X_{t_{i+1}}-X_{t_{i}}\right) ^{2}]\right\vert ^{\rho }<\infty .
\label{CondJain}
\end{equation}%
implies that $X$ has sample paths of finite $p$-variation for $p>2\rho $,
see \cite{jain-monrad-1983} or the survey \cite{dudley-norvaisa-99}. Note
that (\ref{CondJain}) can be written in terms of $R$ and expresses some sort
of "on diagonal $\rho $-variation" regularity of $R$.

The results of this paper put forward the notion of \textit{genuine} $\rho $%
-variation regularity of $R$ as a function on $\left[ 0,1\right] ^{2}$ as
novel and, perhaps, fundamental quantity related to Gaussian processes.
Similar to (\ref{CondJain}), finite $\rho $-variation of $R$, in symbols $%
R\in C^{\rho \text{$-var$}}\left( \left[ 0,1\right] ^{2},\mathbb{R}\right) $%
, can be expressed in terms of the associated Gaussian process and amounts
to say that%
\begin{equation*}
\sup_{D}\sum_{i,j}\left\vert \mathbb{E}\left[ \left(
X_{t_{i+1}}-X_{t_{i}}\right) \left( X_{t_{j+1}}-X_{t_{j}}\right) \right]
\right\vert ^{\rho }<\infty .
\end{equation*}%
The notion of (2D) $\rho $-variation of the covariance leads naturally to%
\begin{equation}
\mathcal{H}\hookrightarrow C^{\rho \text{-var}}\left( \left[ 0,1\right] ,%
\mathbb{R}\right) ,  \label{Hembedding}
\end{equation}%
an embedding of the \textit{Cameron-Martin} or \textit{reproducing kernel
Hilbert space (RKHS)} $\mathcal{H}$ into the space of continuous path with
finite $\rho $-variation. Good examples to have in mind are\ standard
Brownian motion with $\rho =1$ and fractional Brownian motion with Hurst
parameter $H\in (0,1/2]$ for which $\rho =1/\left( 2H\right) $. We then
consider a $d$-dimensional, continuous, centered Gaussian process with
independent components,%
\begin{equation*}
X=\left( X^{1},\ldots ,X^{d}\right) ,
\end{equation*}%
with respective covariances $R_{1},\ldots ,R_{d}\in C^{\rho \text{$-var$}}$
and ask under what conditions there exists an a.s. well-defined lift to a
geometric rough path $\mathbf{X}$ in the sense of T. Lyons. (This amounts,
first and foremost, to define L\'{e}vy's area and higher iterated integrals
of $X,$ and to establish subtle regularity properties.) The answer to this
question is the sufficient (and essentially necessary) condition%
\begin{equation*}
\rho \in \lbrack 1,2)
\end{equation*}%
under which there exists a lift of $X$ to a \textit{Gaussian }geometric $p$%
-rough path $\mathbf{X}$ (short: \textit{Gaussian rough path}) for any $%
p>2\rho $. For fractional Brownian motion this requires $H>1/4$ which is
optimal \cite{coutin-victoir-2005} and our condition is seen to be sharp%
\footnote{%
From \cite{jain-monrad-1983} and \cite{musielak-orlicz-1959} we expect that
logarithmic refinements of this condition are possible but we shall not
pursue this here.}. Recall that geometric $p$-rough paths are (limits of)
paths together with their first $\left[ p\right] $-iterated integrals.
Assuming $\rho <2$ one can (and should) choose $p<4$; when $X$ has
sufficiently smooth sample paths, $\mathbf{X}_{\cdot }\equiv S_{3}\left(
X\right) $ is then simply given by its coordinates in the three
"tensor-levels", $\mathbb{R}^{d},\mathbb{R}^{d}\otimes \mathbb{R}^{d}$and $%
\mathbb{R}^{d}\otimes \mathbb{R}^{d}\otimes \mathbb{R}^{d}$, obtained by
iterated integration%
\begin{equation*}
\mathbf{X}_{\cdot }^{i}=\int_{0}^{\cdot }dX_{r}^{i},\,\,\,\mathbf{X}_{\cdot
}^{i,j}=\int_{0}^{\cdot }\int_{0}^{s}dX_{r}^{i}dX_{s}^{j},\,\,\,\mathbf{X}%
_{\cdot }^{i,j,k}=\int_{0}^{\cdot
}\int_{0}^{t}\int_{0}^{s}dX_{r}^{i}dX_{s}^{j}dX_{t}^{k},\,\,\,
\end{equation*}%
with indices $i,j,k\in \left\{ 1,...,d\right\} $. Our condition $\rho <2$ is
then easy to explain. Assuming $X_{0}=0$ and $i\neq j$, which is enough to
deal with the second tensor level, we have 
\begin{eqnarray*}
\mathbb{E}\left( \left\vert \mathbf{X}_{t}^{i,j}\right\vert ^{2}\right) 
&=&\int_{\left[ 0,t\right] ^{2}}R_{i}\left( u,v\right) \frac{\partial ^{2}}{%
\partial u\partial v}R_{j}\left( u,v\right) dudv \\
&\equiv &\int_{\left[ 0,t\right] ^{2}}R_{i}\left( u,v\right) dR_{j}\left(
u,v\right) .
\end{eqnarray*}%
The integral which appears on the right hand side above is a $2$-dimensional
(short:\ 2D) Young integral. It remains meaningful provided $R_{i},R_{j}$
have finite $\rho _{i}$ resp. $\rho _{j}$-variation with $\rho
_{i}^{-1}+\rho _{j}^{-1}>1$. In particular, if $R_{i},R_{j}$ have both
finite $\rho $-variation this condition reads $\rho <2$. is required. The $%
\rho $-variation condition on the covariance encodes some decorrelation of
the increments and this is the (partial)\ nature of the so-called $\left(
h,p\right) $-long time memory condition that appears in \cite{lyons-qian-02}
resp. Coutin-Qian's condition \cite{coutin-qian-02} which is seen to be more
restrictive than our $\rho $-variation condition.

Let us briefly state our main continuity result for Gaussian rough paths,
taken from section \ref{GaussianRoughPaths}.

\begin{theorem}
Let $X=\left( X^{1},\ldots ,X^{d}\right) ,Y=\left( Y^{1},\ldots
,Y^{d}\right) $ be two continuous, centered jointly Gaussian processes
defined on $\left[ 0,1\right] $ such that $\left( X^{i},Y^{i}\right) $ is
independent of $\left( X^{j},Y^{j}\right) $ when $i\neq j$. Let $\rho \in
\lbrack 1,2)$ and assume the covariance of $\left( X,Y\right) $ is of finite 
$\rho $-variation,%
\begin{equation*}
\left\vert R_{\left( X,Y\right) }\right\vert _{\rho \text{-var};\left[ 0,1%
\right] ^{2}}\leq K<\infty .
\end{equation*}%
Let $p>2\rho $ and $\mathbf{X},\mathbf{Y}$ denote the natural lift of $X,Y$
to a Gaussian rough path. Then there exist positive constants $\theta
=\theta \left( p,\rho \right) $ and $C=C\left( p,\rho ,K\right) $ such that
for all $q\in \lbrack 1,\infty )$,%
\begin{equation*}
\left\vert d_{p\text{-var}}\left( \mathbf{X},\mathbf{Y}\right) \right\vert
_{L^{q}}\leq C\sqrt{q}\left\vert R_{X-Y}\right\vert _{\infty ;\left[ 0,1%
\right] ^{2}}^{\theta }.
\end{equation*}
\end{theorem}

The natural lift to a Gaussian rough path is easily explained along the
above estimates: take a continuous, centered $d$-dimensional process $Z$
with independent components and finite $\rho \in \lbrack 1,2)$-covarianc and
consider its piecewise linear approximations $Z^{n}$. Applying the above
estimate to $\mathbf{X}=S_{3}\left( Z^{n}\right) ,\mathbf{Y}=S_{3}\left(
Z^{m}\right) $, identifies $S_{3}\left( Z^{n}\right) $ as Cauchy sequence
and we call the limit natural lift of $Z$. In conjunction with the \textit{%
universal limit theorem \cite{lyons-qian-02}}, i.e. the continuous
dependence of solutions to (rough) differential equations of the driving
signal $\mathbf{X}$ w.r.t. $d_{p\text{-var}}$, the above theorem contains a
collection of powerful limit theorems which cover, for instance, piecewise
linear and mollifier approximations to Stratonovich SDEs as special case. As
further consequence, weak convergence results are obtained. For instance,
differential equations driven by fractional Brownian Motion with Hurst
parameter $H\rightarrow 1/2$ converge to the corresponding Stratonovich SDEs.

We further note that a large deviation principle holds in the present
generality; thanks to the Cameron-Martin embedding (\ref{Hembedding}) this
follows immediately from the author's previous work \cite%
{friz-victoir-05-GaussLDP}. In the companion paper\ \cite{friz-victoir-2007}
it is shown that, in the same generality, approximations based on the $L^{2}$%
- or \textit{Karhunen-Loeve expansion}%
\begin{equation}
X^{i}\left( t,\omega \right) =\sum_{k\in \mathbb{N}}Z_{k}^{i}\left( \omega
\right) h^{i,k}\left( t\right)   \label{KHexpansion}
\end{equation}%
converge in rough path topology to our natural lift $\mathbf{X}$. (As
corollary, the support of $\mathbf{X}$ is identified as closure of $%
S_{3}\left( \mathcal{H}\right) $ in suitable rough path topology.) The
embedding (\ref{Hembedding}) is absolutely crucial for these purposes: given 
$\rho <2$ it tells us that elements in $\mathcal{H}$ (and in particular,
Karhunen-Loeve approximations which are finite sums of form (\ref%
{KHexpansion})) admit canoncially defined second and third iterated
integrals.

The lift of certain Gaussian processes including fractional Brownian Motion
with Hurst parameter $H>1/4$ is due to Coutin-Qian, \cite{coutin-qian-02}. A
large deviation principle for the lift of fractional Brownian Motion was
obtained in \cite{millet-sanz-2006}, for the Coutin-Qian class in \cite%
{friz-victoir-05-GaussLDP}. Support statements for lifted fractional
Brownian Motion for $H>\frac{1}{3}$ are proved in \cite{friz-victoir-04-Note}%
, \cite{feyel-pradelle-2006}; a Karhunen-Loeve type approximations for
fractional Brownian Motion is studied in \cite{millet-2005}.

The interest in our results goes beyond the unification and optimal
extension of the above-cited articles. It identifies a general framework of
differential equations driven by Gaussian signal surprisingly well-suited
for further (Gaussian) analysis: the embedding (\ref{Hembedding}) combined
with basic facts of Young integrals shows that, at least for $\rho <3/2$,
translations in $\mathcal{H}$-directions are well enough controlled to
exploit the isoperimetric inequality for abstract Wiener spaces;
applications towards optimal regularity/integrability statements for
stochastic area will be discussed in \cite{friz-oberhauser-2007}. Relatedly,
solutions to (rough) differential equations driven by Gaussian signals are $%
\mathcal{H}$-differentiable which allows to establish density results using
Malliavin calculus, to be discussed in \cite{cass-friz-victoir-2007}. 

\subsection{Notations\label{secNotations}}

Let $\left( E,d\right) $ be a metric space and $x\in C\left( \left[ 0,1%
\right] ,E\right) $. It then makes sense to speak of $\alpha $-H\"{o}lder-
and $p$-variation "norms" defined as%
\begin{equation*}
\left\Vert x\right\Vert _{\alpha -H\ddot{o}l}=\sup_{0\leq s<t\leq 1}\frac{%
d\left( x_{s},x_{t}\right) }{\left\vert t-s\right\vert ^{\alpha }}%
,\,\left\Vert x\right\Vert _{p-var}=\sup_{D=\left( t_{i}\right) }\left(
\sum_{i}d\left( x_{t_{i}},x_{t_{i+1}}\right) ^{p}\right) ^{1/p}\,\,.
\end{equation*}%
It also makes sense to speak of a $d_{\infty }$-distance of two such paths,%
\begin{equation*}
d_{\infty }\left( x,y\right) =\sup_{0\leq t\leq 1}d\left( x_{t},y_{t}\right)
.
\end{equation*}%
Given a positive integer $N$ the truncated tensor algebra of degree $N$ is
given by the direct sum 
\begin{equation*}
T^{N}\left( \mathbb{R}^{d}\right) =\mathbb{R}\oplus \mathbb{R}^{d}\oplus
...\oplus \left( \mathbb{R}^{d}\right) ^{\otimes N}.
\end{equation*}%
With tensor product $\otimes $, vector addition and usual scalar
multiplication, $T^{N}\left( \mathbb{R}^{d}\right) =\left( T^{N}\left( 
\mathbb{R}^{d}\right) ,\otimes ,+,.\right) $ is an algebra. Functions such
as $\exp $,$\ln :$ $T^{N}\left( \mathbb{R}^{d}\right) \rightarrow
T^{N}\left( \mathbb{R}^{d}\right) $ are defined immediately by their
power-series. Let $\pi _{i}$ denote the canonical projection from $%
T^{N}\left( \mathbb{R}^{d}\right) $ onto $\left( \mathbb{R}^{d}\right)
^{\otimes i}$. Let $p\in \lbrack 1,2)$ and $x\in $ $C^{p\text{-var}}\left( %
\left[ 0,1\right] ,\mathbb{R}^{d}\right) $, the space of continuous $\mathbb{%
R}^{d}$-valued paths of bounded $q$-variation. We define $\mathbf{x}\equiv
S_{N}(x):[0,1]\rightarrow T^{N}\left( \mathbb{R}^{d}\right) $ via iterated
(Young) integration, 
\begin{equation*}
\mathbf{x}_{t}\equiv
S_{N}(x)_{t}=1+\sum_{i=1}^{N}\int_{0<s_{1}<...<s_{i}<t}dx_{s_{1}}\otimes
...\otimes dx_{s_{i}}
\end{equation*}%
noting that $\mathbf{x}_{0}=1+0+...+0=\exp \left( 0\right) \equiv e$, the
neutral element for $\otimes $, and that $\mathbf{x}_{t}$ really takes
values in 
\begin{equation*}
G^{N}\left( \mathbb{R}^{d}\right) =\left\{ g\in T^{N}\left( \mathbb{R}%
^{d}\right) :\exists x\in C^{1\text{-var}}\left( \left[ 0,1\right] ,\mathbb{R%
}^{d}\right) :\text{ }g=S_{N}(x)_{1}\text{ }\right\} ,
\end{equation*}%
a submanifold of $T^{N}\left( \mathbb{R}^{d}\right) $ and, in fact, a Lie
group with product $\otimes $, called the free step-$N$ nilpotent group with 
$d$ generators. Because $\pi _{1}\left[ \mathbf{x}_{t}\right] =x_{t}-x_{0}$
we say that $\mathbf{x}=S_{N}(x)$ is the canonical lift of $x$. There is a
canonical notion of increments,$\,\mathbf{x}_{s,t}:=\mathbf{x}%
_{s}^{-1}\otimes \mathbf{x}_{t}.$The dilation operator $\delta :\mathbb{R}%
\times G^{N}\left( \mathbb{R}^{d}\right) \rightarrow G^{N}\left( \mathbb{R}%
^{d}\right) $ is defined by 
\begin{equation*}
\pi _{i}\left( \delta _{\lambda }(g)\right) =\lambda ^{i}\pi
_{i}(g),\,\,\,i=0,...,N
\end{equation*}%
and a continuous norm on $G^{N}\left( \mathbb{R}^{d}\right) $, homogenous
with respect to $\delta $, the \textit{Carnot-Caratheodory norm}, is given 
\begin{equation*}
\left\Vert g\right\Vert =\inf \left\{ \text{length}(x):x\in C^{1\text{-var}%
}\left( \left[ 0,1\right] ,\mathbb{R}^{d}\right) ,S_{N}(x)_{1}=g\right\} .
\end{equation*}%
By equivalence of continuous, homogenous norms there exists a constant $%
K_{N} $ such that%
\begin{equation*}
\frac{1}{K_{N}}\max_{i=1,...,N}|\pi _{i}(g)|^{1/i}\leq \left\Vert
g\right\Vert \leq K_{N}\max_{i=1,...,N}|\pi _{i}(g)|^{1/i}.
\end{equation*}%
The norm $\left\Vert \cdot \right\Vert $ induces a (left-invariant) metric
on $G^{N}\left( \mathbb{R}^{d}\right) $ known as \textit{Carnot-Caratheodory
metric}, $d(g,h):=\left\Vert g^{-1}\otimes h\right\Vert .$ Now let $x,y\in $ 
$C_{0}\left( [0,1],G^{N}\left( \mathbb{R}^{d}\right) \right) $, the space of
continuous $G^{N}\left( \mathbb{R}^{d}\right) $-valued paths started at the
neutral element $\exp \left( 0\right) =e$. We define $\alpha $-H\"{o}lder-
and $p$-variation distances 
\begin{eqnarray*}
d_{\alpha -H\ddot{o}l}\left( \mathbf{x,y}\right) &=&\sup_{0\leq s<t\leq 1}%
\frac{d\left( \mathbf{x}_{s,t},\mathbf{y}_{s,t}\right) }{\left\vert
t-s\right\vert ^{\alpha }},\,\,\,\,\, \\
d_{p-var}\left( \mathbf{x,y}\right) &=&\sup_{D=\left( t_{i}\right) }\left(
\sum_{i}d\left( \mathbf{x}_{t_{i},t_{i+1}},\mathbf{y}_{t_{i},t_{i+1}}\right)
^{p}\right) ^{1/p}
\end{eqnarray*}%
and also the "$0$-H\"{o}lder" distance, locally $1/N$-H\"{o}lder equivalent
to $d_{\infty }\left( \mathbf{x,y}\right) $,%
\begin{equation*}
d_{0}\left( \mathbf{x,y}\right) =d_{0-H\ddot{o}l}\left( \mathbf{x,y}\right)
=\sup_{0\leq s<t\leq 1}d\left( \mathbf{x}_{s,t},\mathbf{y}_{s,t}\right) .
\end{equation*}%
Note that $d_{\alpha -H\ddot{o}l}\left( \mathbf{x,}0\right) =\left\Vert 
\mathbf{x}\right\Vert _{\alpha -H\ddot{o}l},\,d_{p-var}\left( \mathbf{x,}%
0\right) =\left\Vert \mathbf{x}\right\Vert _{p-var}$ where $0$ denotes the
constant path $\exp \left( 0\right) $, or in fact, any constant path. The
following path spaces will be useful to us

\begin{enumerate}
\item[(i)] $C_{0}^{p-var}\left( \left[ 0,1\right] ,G^{N}\left( \mathbb{R}%
^{d}\right) \right) $: the set of continuous functions $\mathbf{x}$ from $%
\left[ 0,1\right] $ into $G^{N}\left( \mathbb{R}^{d}\right) $ such that $%
\left\Vert \mathbf{x}\right\Vert _{p-var}<\infty $ and $\mathbf{x}_{0}=\exp
\left( 0\right) $.

\item[(ii)] $C_{0}^{0,p-var}\left( \left[ 0,1\right] ,G^{N}\left( \mathbb{R}%
^{d}\right) \right) $: the $d_{p-var}$-closure of 
\begin{equation*}
\left\{ S_{N}\left( x\right) ,x:\left[ 0,1\right] \rightarrow \mathbb{R}^{d}%
\text{ smooth}\right\} .
\end{equation*}

\item[(iii)] $C_{0}^{1/p-H\ddot{o}l}\left( \left[ 0,1\right] ,G^{N}\left( 
\mathbb{R}^{d}\right) \right) $: the set of continuous functions $\mathbf{x}$
from $\left[ 0,1\right] $ into $G^{n}\left( \mathbb{R}^{d}\right) $ such
that $d_{1/p-H\ddot{o}l}\left( 0,\mathbf{x}\right) <\infty $ and $\mathbf{x}%
_{0}=\exp \left( 0\right) .$

\item[(iv)] $C_{0}^{0,1/p-H\ddot{o}l}\left( \left[ 0,1\right] ,G^{N}\left( 
\mathbb{R}^{d}\right) \right) $: the $d_{1/p-H\ddot{o}l}$-closure of 
\begin{equation*}
\left\{ S_{n}\left( x\right) ,x:\left[ 0,1\right] \rightarrow \mathbb{R}^{d}%
\text{ smooth}\right\} .
\end{equation*}
\end{enumerate}

Recall that a geometric $p$-rough path is an element of $C_{0}^{0,p-var}%
\left( \left[ 0,1\right] ,G^{[p]}\left( \mathbb{R}^{d}\right) \right) ,$ and
a weak geometric rough path is an element of $C_{0}^{p-var}\left( \left[ 0,1%
\right] ,G^{[p]}\left( \mathbb{R}^{d}\right) \right) .$ For a detail study
of these spaces and their properties the reader is referred to \cite%
{friz-victoir-04-Note}.

\section{2D\ Young Integral}

\subsection{On 2D $\protect\rho $-variation}

For a function $f$ from $\left[ 0,1\right] ^{2}$ into a Banach space $\left( 
\mathcal{B},\left\vert .\right\vert \right) $ we will use the notation%
\begin{equation*}
f\left( 
\begin{array}{c}
s \\ 
t%
\end{array}%
,%
\begin{array}{c}
u \\ 
v%
\end{array}%
\right) :=f\left( s,u\right) +f\left( t,v\right) -f\left( s,v\right)
-f\left( t,u\right) \text{.}
\end{equation*}

If $f$ is the 2D distribution function of a signed measure on $\left[ 0,1%
\right] ^{2}$ this is precisely the measure of the rectangle $(s,t]\times
(u,v]$. If $f\left( s,t\right) =\mathbb{E}\left( X_{s}X_{t}\right) \in 
\mathbb{R}$ for some real-valued stochastic process $X,$ then%
\begin{equation*}
f\left( 
\begin{array}{c}
s \\ 
t%
\end{array}%
,%
\begin{array}{c}
u \\ 
v%
\end{array}%
\right) =\mathbb{E}\left( X_{s,t}X_{u,v}\right) .
\end{equation*}%
A similar formula holds when $f\left( s,t\right) =\mathbb{E}\left(
X_{s}\otimes X_{t}\right) \in \mathbb{R}^{d}\otimes \mathbb{R}^{d}$ (which
we equip with its canonical Euclidean structure).

\begin{definition}
Let $f:\left[ 0,1\right] ^{2}\rightarrow $ $\left( \mathcal{B},\left\vert
.\right\vert \right) $. We say that $f$ has finite $\rho $-variation if $%
\left\vert f\right\vert _{\rho -var,\left[ 0,1\right] ^{2}}<\infty ,$ where%
\begin{equation*}
\left\vert f\right\vert _{\rho -var,\left[ s,t\right] \times \left[ u,v%
\right] }=\sup_{\substack{ D=\left( t_{i}\right) \text{ subdivision of }%
\left[ s,t\right]  \\ D^{\prime }=\left( t_{j}^{\prime }\right) \text{
subdivision of }\left[ u,v\right] }}\left( \sum_{i,j}\left\vert f\left( 
\begin{array}{c}
t_{i} \\ 
t_{i+1}%
\end{array}%
,%
\begin{array}{c}
t_{j}^{\prime } \\ 
t_{j+1}^{\prime }%
\end{array}%
\right) \right\vert ^{\rho }\right) ^{1/\rho }.
\end{equation*}
\end{definition}

\begin{definition}
A $2D$ control is a map $\omega $ from $\left( s\leq t,u\leq v\right) $ such
that for all $r\leq s\leq t,$ $u\leq v$,%
\begin{eqnarray*}
\omega \left( \left[ r,s\right] \times \left[ u,v\right] \right) +\omega
\left( \left[ s,t\right] \times \left[ u,v\right] \right) &\leq &\omega
\left( \left[ r,t\right] \times \left[ u,v\right] \right) , \\
\omega \left( \left[ u,v\right] \times \left[ r,s\right] \right) +\omega
\left( \left[ u,v\right] \times \left[ s,t\right] \right) &\leq &\omega
\left( \left[ u,v\right] \times \left[ r,t\right] \right) ,
\end{eqnarray*}%
and such that $\lim_{s\rightarrow t}\omega \left( \left[ s,t\right] \times %
\left[ 0,1\right] \right) =\lim_{s\rightarrow t}\omega \left( \left[ 0,1%
\right] \times \left[ s,t\right] \right) =0.$ Moreover, we will say that the 
$2D$ control $\omega $ is H\"{o}lder-dominated if there exists a constant $C$
such that for all $0\leq s\leq t\leq 1$%
\begin{equation*}
\omega \left( \left[ s,t\right] ^{2}\right) \leq C\left\vert t-s\right\vert
\end{equation*}
\end{definition}

\begin{lemma}
Let $f$ be a $\left( \mathcal{B},\left\vert .\right\vert \right) $-valued
continuous function on $\left[ 0,1\right] ^{2}$.\ Then \newline
(i) If $f$ is of finite $\rho $-variation for some $\rho \geq 1,$%
\begin{equation*}
\left[ s,t\right] \times \left[ u,v\right] \mapsto \left\vert f\right\vert
_{\rho \text{$-var$;}\left[ s,t\right] \times \left[ u,v\right] }^{\rho }
\end{equation*}%
is a 2D\ control.\newline
(i) $f$ is of finite $\rho $-variation on $\left[ 0,1\right] ^{2}$ if and
only if there exists a $2D$ control $\omega $ such that for all $\left[ s,t%
\right] \times \left[ u,v\right] \subset \left[ 0,1\right] ^{2},$%
\begin{equation*}
\left\vert f\left( 
\begin{array}{c}
s \\ 
t%
\end{array}%
,%
\begin{array}{c}
u \\ 
v%
\end{array}%
\right) \right\vert ^{\rho }\leq \omega \left( \left[ s,t\right] \times %
\left[ u,v\right] \right)
\end{equation*}%
and we say that "$\omega $ controls the $\rho $-variation of $f$."
\end{lemma}

\begin{proof}
Straight-forward.
\end{proof}

\begin{remark}
If $f:\left[ 0,T\right] ^{2}\rightarrow $ $\left( \mathcal{B},\left\vert
.\right\vert \right) $ is symmetric (i.e. $f\left( x,y\right) =f\left(
y,x\right) $ for all $x,y$) and of finite $\rho $-variation then $\left[ s,t%
\right] \times \left[ u,v\right] \mapsto \left\vert f\right\vert _{\rho 
\text{$-var$;}\left[ s,t\right] \times \left[ u,v\right] }^{\rho }$ is
symmetric. In fact, one can always work with symmetric controls, it suffices
to replace a given $\omega $ with $\left[ s,t\right] \times \left[ u,v\right]
\mapsto \omega \left( \left[ s,t\right] \times \left[ u,v\right] \right)
+\omega \left( \left[ u,v\right] \times \left[ s,t\right] \right) $.
\end{remark}

\begin{lemma}
\label{1subdivision}A continuous function $f:\left[ 0,1\right]
^{2}\rightarrow $ $\left( \mathcal{B},\left\vert .\right\vert \right) $ is
of finite $\rho $-variation if and only if%
\begin{equation*}
\sup_{D=\left( t_{i}\right) \text{ subdivision of }\left[ 0,1\right] }\left(
\sum_{i,j}\left\vert f\left( 
\begin{array}{c}
t_{i} \\ 
t_{i+1}%
\end{array}%
,%
\begin{array}{c}
t_{j} \\ 
t_{j+1}%
\end{array}%
\right) \right\vert ^{\rho }\right) ^{1/\rho }<\infty .
\end{equation*}%
Moreover, the $\rho $-variation of $f$ is controlled by 
\begin{equation*}
\omega \left( \left[ s,t\right] \times \left[ u,v\right] \right) :=3^{\rho
-1}\sup_{D=\left( t_{i}\right) \text{ subdivision of }\left[ 0,1\right]
}\sum _{\substack{ i,j  \\ \left[ t_{i},t_{i+1}\right] \subset \left[ s,t%
\right]  \\ \left[ t_{j},t_{j+1}\right] \subset \left[ u,v\right] }}%
\left\vert f\left( 
\begin{array}{c}
t_{i} \\ 
t_{i+1}%
\end{array}%
,%
\begin{array}{c}
t_{j} \\ 
t_{j+1}%
\end{array}%
\right) \right\vert ^{\rho }.
\end{equation*}
\end{lemma}

\begin{proof}
Assuming that $\omega \left( \left[ 0,1\right] ^{2}\right) $ is finite, it
is easy to check that $\omega $ is a $2D$ control. Then, for any given $%
\left[ s,t\right] $ and $\left[ u,v\right] $ which do not intersect or such
that $\left[ s,t\right] =\left[ u,v\right] ,$%
\begin{equation*}
\left\vert f\left( 
\begin{array}{c}
s \\ 
t%
\end{array}%
,%
\begin{array}{c}
u \\ 
v%
\end{array}%
\right) \right\vert ^{\rho }\leq \omega \left( \left[ s,t\right] \times %
\left[ u,v\right] \right) .
\end{equation*}%
Take now $s\leq u\leq t\leq v,$ then, 
\begin{eqnarray*}
f\left( 
\begin{array}{c}
s \\ 
t%
\end{array}%
,%
\begin{array}{c}
u \\ 
v%
\end{array}%
\right) &=&f\left( 
\begin{array}{c}
s \\ 
u%
\end{array}%
,%
\begin{array}{c}
u \\ 
v%
\end{array}%
\right) +f\left( 
\begin{array}{c}
u \\ 
t%
\end{array}%
,%
\begin{array}{c}
u \\ 
v%
\end{array}%
\right) \\
&=&f\left( 
\begin{array}{c}
s \\ 
u%
\end{array}%
,%
\begin{array}{c}
u \\ 
v%
\end{array}%
\right) +f\left( 
\begin{array}{c}
u \\ 
t%
\end{array}%
,%
\begin{array}{c}
u \\ 
t%
\end{array}%
\right) +f\left( 
\begin{array}{c}
s \\ 
u%
\end{array}%
,%
\begin{array}{c}
t \\ 
v%
\end{array}%
\right) .
\end{eqnarray*}%
Hence, 
\begin{eqnarray*}
\left\vert f\left( 
\begin{array}{c}
s \\ 
t%
\end{array}%
,%
\begin{array}{c}
u \\ 
v%
\end{array}%
\right) \right\vert ^{\rho } &\leq &3^{\rho -1}\left( \omega \left( \left[
s,u\right] \times \left[ u,v\right] \right) +\omega \left( \left[ u,t\right]
^{2}\right) +\omega \left( \left[ s,u\right] \times \left[ t,v\right]
\right) \right) \\
&\leq &3^{\rho -1}\omega \left( \left[ s,t\right] \times \left[ u,v\right]
\right) .
\end{eqnarray*}%
The other cases are dealt similarly, and we find at the end that for all $%
s\leq t,$ $u\leq v,$%
\begin{equation*}
\left\vert f\left( 
\begin{array}{c}
s \\ 
t%
\end{array}%
,%
\begin{array}{c}
u \\ 
v%
\end{array}%
\right) \right\vert ^{\rho }\leq 3^{\rho -1}\left( \omega \left[ s,t\right]
\times \left[ u,v\right] \right) .
\end{equation*}%
That concludes the proof.
\end{proof}

\begin{example}
\label{2Dvariation_of_FunctionTensorSquared}Given two functions $g,h\in
C^{\rho \text{$-var$}}\left( \left[ 0,T\right] ,\mathcal{B}\right) $ we can
define 
\begin{equation*}
\left( g\otimes h\right) \left( s,t\right) :=g\left( s\right) \otimes
h\left( t\right) \in \mathcal{B}\otimes \mathcal{B}
\end{equation*}%
and $g\otimes h$ has finite $2D$ $\rho $-variation. More precisely, 
\begin{equation*}
\left\vert \left( g\otimes h\right) \left( 
\begin{array}{c}
s \\ 
t%
\end{array}%
,%
\begin{array}{c}
u \\ 
v%
\end{array}%
\right) \right\vert ^{\rho }\leq \left\vert g\right\vert _{\rho \text{$-var$;%
}\left[ s,t\right] }^{\rho }\left\vert h\right\vert _{\rho \text{$-var$;}%
\left[ u,v\right] }^{\rho }=:\omega \left( \left[ s,t\right] \times \left[
u,v\right] \right)
\end{equation*}%
and since $\omega \ $is indeed a $2D$ control function (as product of two $%
1D $ control functions!) we see that%
\begin{equation*}
\left\vert g\otimes h\right\vert _{\rho \text{$-var$;}\left[ s,t\right]
\times \left[ u,v\right] }\leq \left\vert g\right\vert _{\rho \text{$-var$;}%
\left[ s,t\right] }\left\vert h\right\vert _{\rho \text{$-var$;}\left[ u,v%
\right] }\text{.}
\end{equation*}
\end{example}

\begin{remark}
\label{Rmk2DcontrolFct}If $\omega =\omega \left( \left[ s,t\right] \times %
\left[ u,v\right] \right) $ is a $2D$ control function, then%
\begin{equation*}
\left( s,t\right) \mapsto \omega \left( \left[ s,t\right] ^{2}\right)
\end{equation*}%
is a $1D$ control function i.e. $\omega \left( \left[ s,t\right] ^{2}\right)
+\omega \left( \left[ t,u\right] ^{2}\right) \leq \omega \left( \left[ s,u%
\right] ^{2}\right) $, and $s,t\rightarrow \omega \left( \left[ s,t\right]
^{2}\right) $ is continuous and zero on the diagonal.
\end{remark}

A function $f:\left[ 0,T\right] ^{2}\rightarrow $ $\left( \mathcal{B}%
,\left\vert .\right\vert \right) $ of finite $q$-variation can also be
considered as path $t\mapsto f\left( t,\cdot \right) $ with values in the
Banach space $C^{q\text{$-var$}}\left( \left[ 0,T\right] ,\mathcal{B}\right) 
$ with $q$-variation (semi-)norm. It is instructive to observe that $%
t\mapsto f\left( t,\cdot \right) $ has finite $q$-variation if and only if $%
f $ has finite 2D $q$-variation. Let us now prove a (simplified) 2D version
of a result of Musielak Semandi \cite{musielak-semadeni-1961} where they
show that (in 1D) the family $C^{p\text{$-var$}}$ depends on $p$
"semi-continuously from above".

\begin{lemma}
\label{LemmaRhoPrimeToRho}For all $s,t,u,v\in \left[ 0,1\right] $ and $%
\left\vert R\right\vert _{\rho ^{\prime }\text{$-var$;}\left[ s,t\right]
\times \left[ u,v\right] }\rightarrow \left\vert R\right\vert _{\rho \text{$%
-var$;}\left[ s,t\right] \times \left[ u,v\right] }$ when $\rho ^{\prime
}\searrow \rho $.
\end{lemma}

\begin{proof}
Define $\omega \left( s,t\right) ^{1/\rho }=\liminf_{\rho ^{\prime }\searrow
\rho }\left\vert x\right\vert _{\rho ^{\prime }-var,\left[ s,t\right] }.$ As 
$\rho ^{\prime }\mapsto \left\vert x\right\vert _{\rho ^{\prime }-var,\left[
s,t\right] }$ is decreasing, this limit exists in $\left[ 0,\infty \right] ,$
and as $\left\vert x\right\vert _{\rho ^{\prime }-var,\left[ s,t\right]
}\leq \left\vert x\right\vert _{\rho -var,\left[ s,t\right] }<\infty ,$ it
actually exists in $[0,\infty )$ and $\omega \left( s,t\right) \leq
\left\vert x\right\vert _{\rho -var,\left[ s,t\right] }^{\rho }$. For all $%
s,t\in \left[ 0,1\right] ,$ for all $\rho ^{\prime }$%
\begin{equation*}
\left\vert x_{s,t}\right\vert \leq \left\vert x\right\vert _{\rho ^{\prime
}-var,\left[ s,t\right] }.
\end{equation*}%
Taking the limit, we obtain $\left\vert x_{s,t}\right\vert \leq \omega
\left( s,t\right) ^{1/\rho }$ and we now show that that $\omega $ is
super-additive. Take $s\leq t\leq u,$%
\begin{eqnarray*}
\omega \left( s,t\right) +\omega \left( t,u\right) &=&\lim_{\rho ^{\prime
}\rightarrow \rho }\left\vert x\right\vert _{\rho ^{\prime }-var,\left[ s,t%
\right] }^{\rho ^{\prime }}+\lim_{\rho ^{\prime }\rightarrow \rho
}\left\vert x\right\vert _{\rho ^{\prime }-var,\left[ t,u\right] }^{\rho
^{\prime }} \\
&=&\lim_{\rho ^{\prime }\rightarrow \rho }\left( \left\vert x\right\vert
_{\rho ^{\prime }-var,\left[ s,t\right] }^{\rho ^{\prime }}+\left\vert
x\right\vert _{\rho ^{\prime }-var,\left[ t,u\right] }^{\rho ^{\prime
}}\right) \\
&\leq &\lim_{\rho ^{\prime }\rightarrow \rho }\left\vert x\right\vert _{\rho
^{\prime }-var,\left[ s,u\right] }^{\rho ^{\prime }} \\
&\leq &\omega \left( s,u\right) .
\end{eqnarray*}%
It follows that $\left\vert x\right\vert _{\rho \text{$-var$;}\left[ s,t%
\right] }^{\rho }\leq \omega \left( s,t\right) $. But this implies that for
all $s<t$ in $\left[ 0,1\right] $,%
\begin{equation*}
\lim_{\rho ^{\prime }\rightarrow \rho }\left\vert x\right\vert _{\rho
^{\prime }-var,\left[ s,t\right] }=\left\vert x\right\vert _{\rho ^{\prime
}-var,\left[ s,t\right] }.
\end{equation*}
\end{proof}

\subsection{The Integral}

Young integrals extend naturally to higher dimensions, see \cite{young-36}, 
\cite{towghi-2002}. We focus on dimension $2$ and $\mathcal{B}=\mathbb{R},$%
which is what we need in the sequel.

\begin{definition}
Let $f:\left[ s,t\right] ^{2}\rightarrow \mathbb{R},$ $g:\left[ u,v\right]
^{2}\rightarrow \mathbb{R}$ be continuous. Let $D=\left( t_{i}\right) $ be a
dissection of $\left[ s,t\right] ,\,D^{\prime }=\left( t_{j}^{\prime
}\right) $ be a dissection of $\left[ u,v\right] $. If the 2D
Riemann--Stieltjes sum%
\begin{equation*}
\sum_{i,j}f\left( t_{i},t_{j}^{\prime }\right) g\left( 
\begin{array}{c}
t_{i} \\ 
t_{i+1}%
\end{array}%
,%
\begin{array}{c}
t_{j}^{\prime } \\ 
t_{j+1}^{\prime }%
\end{array}%
\right)
\end{equation*}%
converges when $\max \left\{ mesh\left( D\right) ,mesh\left( D^{\prime
}\right) \right\} \rightarrow 0$ we call the limit 2D\ Young-integral and
write $\int_{\left[ s,t\right] \times \left[ u,v\right] }fdg$ or simply $%
\int fdg$ if no confusion arises.
\end{definition}

We leave it to the reader to check that if $g$ is of bounded variation (i.e.
finite $1$-variation) it induces a signed Radon measure, say $\lambda _{g}$,
and%
\begin{equation*}
\int fdg=\int fd\lambda _{g}.
\end{equation*}

\begin{example}
If $g\left( s,t\right) =\int_{0}^{s}\int_{0}^{t}r\left( x,y\right) dxy$
(think of a 2D distribution function!) then%
\begin{equation*}
\int_{\left[ s,t\right] \times \left[ u,v\right] }fdg=\int_{\left[ s,t\right]
\times \left[ u,v\right] }f\left( x,y\right) r\left( x,y\right) dxdy.
\end{equation*}
\end{example}

The following theorem was proved in \cite{towghi-2002}, see also Young's
original paper \cite{young-36} for a (weaker)\ result in the same direction;
we include a proof for the reader's convenience.

\begin{theorem}
Let $f:\left[ 0,T\right] ^{2}\rightarrow \mathbb{R},$ $g:\left[ 0,T\right]
^{2}\rightarrow \mathbb{R}$ two continuous functions of finite $q$-variation
(respectively of finite $p$-variation), with $q^{-1}+p^{-1}>1,$ controlled
by $\omega $. Then the 2D\ Young-integral $\int_{\left[ 0,T\right] ^{2}}fdg$
exists and if $f\left( s,.\right) =f\left( .,u\right) =0$%
\begin{equation*}
\left\vert \int_{\left[ s,t\right] \times \left[ u,v\right] }fdg\right\vert
\leq C_{p,q}\left\vert f\right\vert _{q-var,\left[ s,t\right] \times \left[
u,v\right] }.\left\vert g\right\vert _{p-var,\left[ s,t\right] \times \left[
u,v\right] }.
\end{equation*}
\end{theorem}

\begin{proof}
Let $\omega _{f},\omega _{g}$ be controls dominating the $q$-variation of $f$
and $p$-variation of $g$, and let $\omega =\omega _{f}^{1/q}\omega
_{g}^{1/p}.$ Observe that by H\"{o}lder inequality, $\omega $ itself is a
control For a fixed $x,x^{\prime }\in \left[ s,t\right] ,$ define the
functions $f_{x},g_{x}$ by 
\begin{eqnarray*}
f_{x,x^{\prime }}\left( y\right) &=&f\left( x,y\right) -f\left( x^{\prime
},y\right) ,\text{ \ }y\in \left[ u,v\right] , \\
g_{x,x^{\prime }}\left( y\right) &=&f\left( x,y\right) -f\left( x^{\prime
},y\right) ,\text{\ }y\in \left[ u,v\right] .
\end{eqnarray*}%
Observe that $y\rightarrow f_{x,x^{\prime }}\left( y\right) $ (resp. $%
y\rightarrow g_{x,x^{\prime }}\left( y\right) $) is of finite $q$-variation
(resp. $p$-variation) controlled by $\left( y,y^{\prime }\right) \rightarrow
\omega _{f}\left( \left[ x,x^{\prime }\right] \times \left[ y,y^{\prime }%
\right] \right) $ (resp. by $\left( y,y^{\prime }\right) \rightarrow \omega
_{g}\left( \left[ x,x^{\prime }\right] \times \left[ y,y^{\prime }\right]
\right) $). That implies in particular by Young $1D$ estimates that 
\begin{equation*}
\left\vert \int_{u}^{v}f_{x_{1},x_{2}}\left( y\right) dg_{x_{3},x_{4}}\left(
y\right) \right\vert \leq C_{q,p}\omega _{f}\left( \left[ x_{1},x_{2}\right]
\times \left[ u,v\right] \right) ^{1/q}\omega _{g}\left( \left[ x_{3},x_{4}%
\right] \times \left[ u,v\right] \right) ^{1/p}
\end{equation*}

For subdivision $D=\left( t_{i}\right) $ of $\left[ s,t\right] ,$ let $%
I_{u,v}^{D}=\sum_{i}\int_{u}^{v}f_{s,x_{i}}\left( y\right)
dg_{x_{i},x_{i+1}}\left( y\right) .$ Now, let $D\backslash \left\{ i\right\} 
$ the subdivision $D$ with the point $t_{i}$ removed. It is easy to see that 
\begin{eqnarray*}
\left\vert I_{u,v}^{D}-I_{u,v}^{D\backslash \left\{ i\right\} }\right\vert
&=&\left\vert \int_{u}^{v}f_{x_{i-1},x_{i}}\left( y\right)
dg_{x_{i},x_{i+1}}\left( y\right) \right\vert \\
&\leq &C_{q,p}\omega _{f}\left( \left[ x_{i-1},x_{i}\right] \times \left[ u,v%
\right] \right) ^{1/q}\omega _{g}\left( \left[ x_{i},x_{i+1}\right] \times %
\left[ u,v\right] \right) ^{1/p} \\
&\leq &C_{q,p}\omega _{f}\left( \left[ x_{i-1},x_{i+1}\right] \times \left[
u,v\right] \right) ^{1/q}\omega _{g}\left( \left[ x_{i-1},x_{i+1}\right]
\times \left[ u,v\right] \right) ^{1/p} \\
&=&C_{q,p}\omega \left( \left[ x_{i-1},x_{i+1}\right] \times \left[ u,v%
\right] \right)
\end{eqnarray*}%
Choosing the point $i$ such that $\omega \left( \left[ x_{i-1},x_{i+1}\right]
\times \left[ u,v\right] \right) \leq \frac{2}{r-1}\omega \left( \left[ s,t%
\right] \times \left[ u,v\right] \right) ,$ where $r$ is the number of
points in the subdivision $D.$ Working from this point as in the proof of
Young $1D$ estimate, we therefore obtain%
\begin{eqnarray*}
\left\vert I_{u,v}^{D}\right\vert &\leq &C_{q,p}^{2}\omega \left( \left[ s,t%
\right] \times \left[ u,v\right] \right) \\
&=&C_{q,p}^{2}\omega _{f}\left( \left[ s,t\right] \times \left[ u,v\right]
\right) ^{1/q}\omega _{g}\left( \left[ s,t\right] \times \left[ u,v\right]
\right) ^{1/p}.
\end{eqnarray*}%
We finish as in Young 1D proof, \cite{young-36}.
\end{proof}

\section{One Dimensional Gaussian Processes and the $\protect\rho $%
-variation of their Covariance}

\subsection{Examples}

\subsubsection{Brownian Motion}

Standard Brownian motion $B$ on $\left[ 0,1\right] $ has covariance $%
R_{BM}\left( s,t\right) =\min \left( s,t\right) $. By Lemma \ref%
{1subdivision}, or directly from the definition, $R$ has finite $\rho $%
-variation with $\rho =1$, controlled by%
\begin{eqnarray*}
\omega \left( \left[ s,t\right] \times \left[ u,v\right] \right)
&=&\left\vert \left( s,t\right) \cap \left( u,v\right) \right\vert \\
&=&\int_{\left[ s,t\right] \times \left[ u,v\right] }\delta _{x=y}\left(
dxdy\right) ,
\end{eqnarray*}%
where $\delta $ is the Dirac mass. Since $\omega \left( \left[ s,t\right]
^{2}\right) =\left\vert t-s\right\vert ,$ it is H\"{o}lder dominated.

\subsubsection{(Gaussian) Martingales}

We know that a continuous Gaussian martingale $M$ has a deterministic
bracket so that%
\begin{equation*}
M\left( t\right) \overset{D}{=}B_{\left\langle M\right\rangle _{t}}.
\end{equation*}%
In particular,%
\begin{equation*}
R\left( s,t\right) =\min \left\{ \left\langle M\right\rangle
_{s},\left\langle M\right\rangle _{t}\right\} =\left\langle M\right\rangle
_{\min \left( s,t\right) }.
\end{equation*}

But the notion of $\rho $-variation is invariant under time-change and it
follows that $R$ has finite $1$-variation since $R_{BM}$ has finite $1$%
-variation. One should notice that $L^{2}$-martingales (without assuming a
Gaussian structure) have orthogonal increments i.e.%
\begin{equation*}
\mathbb{E}\left( X_{s,t}X_{u,v}\right) =0\text{ if }s<t<u<v
\end{equation*}%
and this alone will take care of the *usually difficult to handle)
off-diagonal part in the variation of the covariance $\left( s,t\right)
\mapsto \mathbb{E}\left( X_{s}X_{t}\right) $.

\subsubsection{Bridges, Ornstein-Uhlenbeck Process}

Gaussian Bridge processes are immediate generalizations of the Brownian
Bridge. Given a real-valued centered Gaussian process $X$ on $\left[ 0,1%
\right] $ with continuous covariance $R$ of finite $\rho $-variation the
corresponding Bridge is defined as%
\begin{equation*}
X_{B}\left( t\right) =X\left( t\right) -tX\left( 1\right)
\end{equation*}%
with covariance $R_{B}$. It is a simple exercise left to the reader to see
that $R_{B}$ has finite $\rho $-variation. Moreover, if $R$ has its $\rho $%
-variation over $\left[ s,t\right] ^{2}$ dominated by a H\"{o}lder control,
then $R_{B}$ has also $\rho $-variation dominated by a H\"{o}lder control.

The usual Ornstein-Uhlenbeck (stationary or started at a fixed point) also
has finite $1$-variation, H\"{o}lder dominated on $\left[ s,t\right] ^{2}$.
This is seen directly from the explicitly known covariance function and also
left to the reader.

\subsubsection{Fractional Brownian Motion}

Finding the precise $\rho $-variation for the covariance of the fractional
Brownian motion is more involved. For Hurst parameter $H>1/2$, fractional
Brownian Motion has H\"{o}lder sample paths with exponent greater than $1/2$
which is, for the purpose of this paper, a trivial case.

\begin{proposition}
\label{PropFBMrhoVariationReg}Let $B^{H}$ be fractional Brownian motion of
Hurst parameters $H\in (0,1/2].$ Then, its covariance is of finite $1/\left(
2H\right) $-variation. Moreover, its $\rho $-variation over $\left[ s,t%
\right] ^{2}$ is bounded by $C_{H}\left\vert t-s\right\vert .$
\end{proposition}

\begin{proof}
Let $D=\left\{ t_{i}\right\} $ be a dissection of $\left[ s,t\right] $, and
let us look at 
\begin{equation*}
\sum_{i,j}\left\vert \mathbb{E}\left(
B_{t_{i},t_{i+1}}^{H}B_{t_{j,}t_{j+1}}^{H}\right) \right\vert ^{\frac{1}{2H}%
}.
\end{equation*}%
For a fixed $i$ and $i\neq j,$ as $H\leq \frac{1}{2},$ $\mathbb{E}\left(
B_{t_{i},t_{i+1}}^{H}B_{t_{j,}t_{j+1}}^{H}\right) $ is negative, hence, 
\begin{eqnarray*}
&&\sum_{j}\left\vert \mathbb{E}\left(
B_{t_{i},t_{i+1}}^{H}B_{t_{j,}t_{j+1}}^{H}\right) \right\vert ^{\frac{1}{2H}}
\\
&\leq &\sum_{j\neq i}\left\vert \mathbb{E}\left(
B_{t_{i},t_{i+1}}^{H}B_{t_{j,}t_{j+1}}^{H}\right) \right\vert ^{\frac{1}{2H}%
}+\mathbb{E}\left( \left\vert B_{t_{i},t_{i+1}}^{H}\right\vert ^{2}\right) ^{%
\frac{1}{2H}} \\
&\leq &\left\vert \mathbb{E}\left( \sum_{j\neq
i}B_{t_{i},t_{i+1}}^{H}B_{t_{j,}t_{j+1}}^{H}\right) \right\vert ^{\frac{1}{2H%
}}+\mathbb{E}\left( \left\vert B_{t_{i},t_{i+1}}^{H}\right\vert ^{2}\right)
^{\frac{1}{2H}} \\
&\leq &\left( 2^{\frac{1}{2H}-1}\left\vert \mathbb{E}\left(
\sum_{j}B_{t_{i},t_{i+1}}^{H}B_{t_{j,}t_{j+1}}^{H}\right) \right\vert ^{%
\frac{1}{2H}}+2^{\frac{1}{2H}-1}\mathbb{E}\left( \left\vert
B_{t_{i},t_{i+1}}^{H}\right\vert ^{2}\right) ^{\frac{1}{2H}}\right) \\
&&+\mathbb{E}\left( \left\vert B_{t_{i},t_{i+1}}^{H}\right\vert ^{2}\right)
^{\frac{1}{2H}} \\
&\leq &C_{H}\left\vert \mathbb{E}\left(
B_{t_{i},t_{i+1}}^{H}B_{s,t}^{H}\right) \right\vert ^{\frac{1}{2H}}+C_{H}%
\mathbb{E}\left( \left\vert B_{t_{i},t_{i+1}}^{H}\right\vert ^{2}\right) ^{%
\frac{1}{2H}}.
\end{eqnarray*}%
Hence, 
\begin{eqnarray*}
\sum_{i,j}\left\vert \mathbb{E}\left(
B_{t_{i},t_{i+1}}^{H}B_{t_{j,}t_{j+1}}^{H}\right) \right\vert ^{\frac{1}{2H}%
} &\leq &C_{H}\sum_{i}\mathbb{E}\left( \left\vert
B_{t_{i},t_{i+1}}^{H}\right\vert ^{2}\right) ^{\frac{1}{2H}} \\
&&+C_{H}\sum_{i}\left\vert \mathbb{E}\left(
B_{t_{i},t_{i+1}}^{H}B_{s,t}^{H}\right) \right\vert ^{\frac{1}{2H}}.
\end{eqnarray*}%
The first term is equal to $C_{H}\left\vert t-s\right\vert ,$ so we just
need to prove that\footnote{$h(\cdot )=E\left( B_{\cdot
}^{H}B_{s,t}^{H}\right) $ defines a Cameron-Martin path and estimate (\ref%
{ForPfOfRhoVarForFbm}) says that $\left\vert h\right\vert _{1/(2H)\text{-var;%
}\left[ s,t\right] }\leq C\left\vert t-s\right\vert ^{2H}$. It is
instructive to compare this with the Section on Cameron Martin spaces.} 
\begin{equation}
\sum_{i}\left\vert \mathbb{E}\left( B_{t_{i},t_{i+1}}^{H}B_{s,t}^{H}\right)
\right\vert ^{\frac{1}{2H}}\leq C_{H}\left\vert t-s\right\vert .
\label{ForPfOfRhoVarForFbm}
\end{equation}%
To achieve this, it will be enough to prove that for $\left[ u,v\right]
\subset \left[ s,t\right] ,$%
\begin{equation*}
\left\vert \mathbb{E}\left( B_{u,v}^{H}B_{s,t}^{H}\right) \right\vert \leq
C_{H}\left\vert v-u\right\vert ^{2H}.
\end{equation*}%
First recall that as $2H<1,$ if $0<x<y,$ then $\left( x+y\right)
^{2H}-x^{2H}\leq y^{2H}.$ Hence, using this inequality and the triangle
inequality,%
\begin{eqnarray*}
\left\vert \mathbb{E}\left( B_{u,v}^{H}B_{s,t}^{H}\right) \right\vert
&=&c_{H}\left\vert \left( t-v\right) ^{2H}+\left( u-s\right) ^{2H}-\left(
v-s\right) ^{2H}-\left( t-u\right) ^{2H}\right\vert \\
&\leq &c_{H}\left( \left( t-u\right) ^{2H}-\left( t-v\right) ^{2H}\right)
+c_{H}\left( \left( v-s\right) ^{2H}-\left( u-s\right) ^{2H}\right) \\
&\leq &2c_{H}\left( v-u\right) ^{2H}.
\end{eqnarray*}
\end{proof}

\subsubsection{Coutin-Qian condition on the covariance\label{CQcond}}

Coutin and Qian \cite{coutin-qian-02} constructed a rough paths over a class
of Gaussian process. We prove here that when we look at the $\rho $%
-variation of their covariance, they are not very different than fractional
Brownian motion\footnote{%
As remarked in more detail in the introduction, a slight generalization of
this condition appears in \cite{coutin-qian-02} and is applicable to certain
non-Gaussian processes.}.

\begin{definition}
A real-valued Gaussian process $X$ on $\left[ 0,1\right] $ satisfies the
Coutin-Qian conditions if for some $H$%
\begin{eqnarray}
\mathbb{E}\left( \left\vert X_{s,t}\right\vert ^{2}\right) &\leq
&c_{H}\left\vert t-s\right\vert ^{2H},\text{ for all }s<t,  \label{ass1} \\
\left\vert \mathbb{E}\left( X_{s,s+h}X_{t,t+h}\right) \right\vert &\leq
&c_{H}\left\vert t-s\right\vert ^{2H-2}h^{2},\text{ for all }s,t,h\text{
with }h<t-s.  \label{ass2}
\end{eqnarray}
\end{definition}

\begin{lemma}
\label{QCboundedbyfBM}Let $X$ be a Gaussian process on $\left[ 0,1\right] $
that satisfies the Coutin-Qian conditions for some $H>0,$ and let $\omega
_{H}$ the control of the $\frac{1}{2H}$-variation of the covariance of the
fractional Brownian Motion with Hurst parameter $H$. Then, for $s\leq t$ and 
$u\leq v,$ 
\begin{equation*}
\left\vert \mathbb{E}\left( X_{s,t}X_{u,v}\right) \right\vert \leq
C_{H}\omega _{H}\left( \left[ s,t\right] \times \left[ u,v\right] \right)
^{2H}.
\end{equation*}%
In particular, the covariance of $X$ has finite $\frac{1}{2H}$-variation.
\end{lemma}

\begin{proof}
Working as in lemma \ref{1subdivision}, at the price of a factor $3^{\frac{1%
}{2H}-1},$ we can restrict ourselves to the cases $s=u\leq t=v,$ and $s\leq
t\leq u\leq v.$ The first case it given by assumption (\ref{ass1}), so let
us focus on the second one. Assume first we can write $t-s=nh,$ $v-u=mh.$
and that $u-t>h.$ Then,%
\begin{equation*}
\mathbb{E}\left( X_{s,t}X_{u,v}\right) =\sum_{k=0}^{n-1}\sum_{l=0}^{m-1}%
\mathbb{E}\left( X_{s+kh,s+\left( k+1\right) h}X_{t+lh,t+\left( l+1\right)
h}\right) .
\end{equation*}%
Using the triangle inequality and our assumption,%
\begin{eqnarray*}
\left\vert \mathbb{E}\left( X_{s,t}X_{u,v}\right) \right\vert
&=&\sum_{k=0}^{n-1}\sum_{l=0}^{m-1}\left\vert \mathbb{E}\left(
X_{s+kh,s+\left( k+1\right) h}X_{u+lh,u+\left( l+1\right) h}\right)
\right\vert \\
&\leq &C_{H}\sum_{k=0}^{n-1}\sum_{l=0}^{m-1}\left\vert \left( u+lh\right)
-\left( s+kh\right) \right\vert ^{2H-2}h^{2} \\
&\leq &C_{H}\sum_{k=0}^{n-1}\sum_{l=0}^{m-1}\int_{u+\left( l-1\right)
h}^{u+lh}\int_{s+kh}^{s+\left( k+1\right) h}\left\vert y-x\right\vert
^{2H-2}dxdy \\
&\leq &C_{H}\int_{u-h}^{v-h}\int_{s}^{t}\left\vert y-x\right\vert ^{2H-2}dxdy
\\
&\leq &C_{H}\left\vert \mathbb{E}\left( B_{u-h,v-h}^{H}B_{s,t}^{H}\right)
\right\vert
\end{eqnarray*}%
Letting $h$ tends to $0,$ by continuity, we easily see that 
\begin{equation*}
\left\vert \mathbb{E}\left( X_{s,t}X_{u,v}\right) \right\vert \leq
C_{H}\left\vert \mathbb{E}\left( B_{u,v}^{H}B_{s,t}^{H}\right) \right\vert ,
\end{equation*}%
which implies our statement for $s\leq t\leq u\leq v.$ That concludes the
proof.
\end{proof}

\subsection{Cameron Martin space}

We consider a real-valued centered Gaussian process $X$ on $\left[ 0,1\right]
$ with continuous sample paths and covariance $R$. The associated
Cameron-Martin space $\mathcal{H}\subset C\left( \left[ 0,1\right] \right) $
consists of paths%
\begin{equation*}
t\mapsto h_{t}=\mathbb{E}\left( ZX_{t}\right)
\end{equation*}%
where $Z$ is a $\left\{ \sigma \left( X_{t}\right) ,t\in \left[ 0,1\right]
\right\} $-measurable, Gaussian random variable. If $h^{\prime }={E}\left(
Z^{\prime }X_{\cdot }\right) $ denotes another element in $\mathcal{H}$, the
inner product on$\mathcal{\ H}$ is defined as 
\begin{equation*}
\left\langle h,h^{\prime }\right\rangle _{\mathcal{H}}=\mathbb{E}\left(
ZZ^{\prime }\right) .
\end{equation*}

Regularity of Cameron-Martin paths is not only a natural question in its own
right but will prove crucial in our later sections on support theorem and
large deviations.

\begin{proposition}
\label{CM_pVar_embedding}If $R$ is of finite $\rho $-variation, then $%
\mathcal{H}\subset C^{\rho \text{$-var$}}$. More, precisely, for all $h\in 
\mathcal{H}$%
\begin{equation*}
\left\vert h\right\vert _{\rho \text{$-var$;}\left[ s,t\right] }\leq \sqrt{%
\left\langle h,h\right\rangle _{\mathcal{H}}}\sqrt{R_{\rho \text{$-var$;}%
\left[ s,t\right] ^{2}}}.
\end{equation*}
\end{proposition}

\begin{proof}
Let $h=\mathbb{E}\left( ZX_{.}\right) ,$ and $\left( t_{j}\right) $ a
subdivision of $\left[ s,t\right] .$ We write $\left\vert x\right\vert
_{l^{r}}=\left( \sum_{i}x_{i}^{r}\right) ^{1/r}$ for $r\geq 1.$ Let $\rho
^{\prime }$ be the conjugate of $\rho :$%
\begin{eqnarray*}
&&\left( \sum_{j}\left\vert h_{t_{j},t_{j+1}}\right\vert ^{\rho }\right)
^{1/\rho } \\
&=&\sup_{\beta ,\left\vert \beta \right\vert _{l^{\rho ^{\prime }}}\leq
1}\sum_{j}\beta _{j}h_{t_{j},t_{j+1}}=\sup_{\beta ,\left\vert \beta
\right\vert _{l^{\rho ^{\prime }}}\leq 1}\mathbb{E}\left( Z\sum_{j}\beta
_{j}X_{t_{j},t_{j+1}}\right) \\
&\leq &\sqrt{\mathbb{E}\left( Z^{2}\right) }\sup_{\beta ,\left\vert \beta
\right\vert _{l^{\rho ^{\prime }}}\leq 1}\sqrt{\sum_{j,k}\beta _{j}\beta _{k}%
\mathbb{E}\left( X_{t_{j},t_{j+1}}X_{t_{k},t_{k+1}}\right) } \\
&\leq &\sqrt{\left\langle h,h\right\rangle _{\mathcal{H}}}\sup_{\beta
,\left\vert \beta \right\vert _{l^{\rho ^{\prime }}}\leq 1}\sqrt{\left(
\sum_{j,k}\left\vert \beta _{j}\right\vert ^{\rho ^{\prime }}\left\vert
\beta _{k}\right\vert ^{\rho ^{\prime }}\right) ^{\frac{1}{\rho ^{\prime }}%
}\left( \sum_{j,k}\left\vert \mathbb{E}\left(
X_{t_{j},t_{j+1}}X_{t_{k},t_{k+1}}\right) \right\vert ^{\rho }\right) ^{%
\frac{1}{\rho }}} \\
&\leq &\sqrt{\left\langle h,h\right\rangle _{\mathcal{H}}}\left(
\sum_{j,k}\left\vert \mathbb{E}\left(
X_{t_{j},t_{j+1}}X_{t_{k},t_{k+1}}\right) \right\vert ^{\rho }\right)
^{1/\left( 2\rho \right) } \\
&\leq &\sqrt{\left\langle h,h\right\rangle _{\mathcal{H}}}\sqrt{R_{\rho -var,%
\left[ s,t\right] ^{2}}}.
\end{eqnarray*}%
Optimizing over all subdivision $\left( t_{j}\right) $ of $\left[ s,t\right]
,$ we obtain our result.
\end{proof}

\begin{remark}
Observe that for Brownian motion ($\rho =1$), this is a sharp result.
\end{remark}

\subsection{Piecewise-linear Approximations}

Let $X$ be centered real-valued continuous Gaussian process on $\left[ 0,1%
\right] $ with covariance $R$ assumed to be of finite $\rho $-variation,
dominated by some 2D control function $\omega $. Let $D=\left\{
t_{i}\right\} $ be a dissection of $\left[ 0,1\right] $ and let $X^{D}$
denote the piecewise-linear approximation to $X$ i.e. $X_{t}^{D}=X_{t}$ for $%
t\in D$ and $X^{D}$ is linear between two successive points of $D$. If $%
\left( s,t\right) \times \left( u,v\right) \subset \left(
t_{i},t_{i+1}\right) \times \left( t_{j},t_{j+1}\right) $ then the
covariance of $X^{D}$, denoted by $R^{D}$, is given by

\begin{eqnarray}
R^{D}\left( 
\begin{array}{c}
s \\ 
t%
\end{array}%
,%
\begin{array}{c}
u \\ 
v%
\end{array}%
\right) &=&\mathbb{E}\left( \int_{s}^{t}\dot{X}_{r}^{D}dr\int_{u}^{v}\dot{X}%
_{r}^{D}dr\right)  \label{RDonDatom} \\
&=&\frac{t-s}{t_{i+1}-t_{i}}\times \frac{v-u}{t_{j+1}-t_{j}}R\left( 
\begin{array}{c}
t_{i} \\ 
t_{i+1}%
\end{array}%
,%
\begin{array}{c}
t_{j} \\ 
t_{j+1}%
\end{array}%
\right) .  \notag
\end{eqnarray}%
The aim of this section is to show that the $\rho $-variation of $R^{D}$ is
fully comparable with the $\rho $-variation of $R$. As usual, given $s\in %
\left[ 0,1\right] ,$ we write $s_{D}$ the greatest element of $D$ such that $%
s_{D}\leq s,$ and $s^{D}$ the smallest element of $D$ such that $s<s^{D}.$

\begin{lemma}
(i) For all $u_{1},v_{1}$,$u_{2},v_{2}\in D,$%
\begin{equation*}
\left\vert R^{D}\right\vert _{\rho -var,\left[ u_{1},v_{1}\right] \times %
\left[ u_{2},v_{2}\right] }\leq 9^{1-\frac{1}{\rho }}\left\vert R\right\vert
_{\rho -var,\left[ u_{1},v_{1}\right] \times \left[ u_{2},v_{2}\right] },
\end{equation*}%
\newline
(ii) For all $s,t\in \left[ 0,1\right] $, with $s_{D}\leq s,t\leq s^{D},$
for all $u,v\in D,$ 
\begin{equation*}
\left\vert R^{D}\right\vert _{\rho -var,\left[ s,t\right] \times \left[ u,v%
\right] }\leq 9^{1-\frac{1}{\rho }}\left\vert \frac{t-s}{s^{D}-s_{D}}%
\right\vert \mathbb{E}\left( \left\vert X_{s_{D},s^{D}}\right\vert
^{2}\right) ^{1/2}\left\vert R\right\vert _{\rho -var,\left[ u,v\right]
^{2}}^{1/2},
\end{equation*}%
\newline
(iii) For all $s_{1},t_{1},s_{2},t_{2}\in \left[ 0,1\right] $, with $%
s_{1,D}\leq s_{1},t_{1}\leq s_{1}^{D},$ $s_{2,D}\leq s_{2},t_{2}\leq
s_{2}^{D},$%
\begin{equation*}
\left\vert R^{D}\right\vert _{\rho -var,\left[ s_{1},t_{1}\right] \times %
\left[ s_{2},t_{2}\right] }\leq \left\vert \frac{t_{1}-s_{1}}{%
s_{1}^{D}-s_{1,D}}\right\vert \left\vert \frac{t_{2}-s_{2}}{s_{2}^{D}-s_{2,D}%
}\right\vert \left\vert \mathbb{E}\left(
X_{s_{1,D},s^{1,D}}X_{s_{2,D},s^{2,D}}\right) \right\vert .
\end{equation*}
\end{lemma}

\begin{proof}
For the first point, let $\left( s_{i}\right) $ be a subdivision of $\left[
u_{1},v_{1}\right] $ and $\left( t_{j}\right) $ be a subdivision of $\left[
u_{2},v_{2}\right] .$ By definition of $X^{D},$%
\begin{eqnarray*}
X_{s_{i},s_{i+1}}^{D} &=&\frac{s_{i}^{D}-s_{i}}{s_{i}^{D}-s_{i,D}}%
X_{s_{i,D},s_{i}^{D}}+X_{s_{i}^{D},s_{i+1,D}}+\frac{s_{i+1}-s_{i+1,D}}{%
s_{i+1}^{D}-s_{i+1,D}}X_{s_{i+1,D},s_{i+1}^{D}}, \\
X_{t_{j},t_{j+1}}^{D} &=&\frac{t_{j}^{D}-t_{j}}{t_{j}^{D}-t_{j,D}}%
X_{t_{j,D},t_{j}^{D}}+X_{t_{j}^{D},t_{j+1,D}}+\frac{t_{j+1}-t_{j+1,D}}{%
t_{j+1}^{D}-t_{j+1,D}}X_{t_{j+1,D},t_{j+1}^{D}}.
\end{eqnarray*}%
Therefore, using H\"{o}lder inequality and the fact that $\left\vert \frac{%
s_{i}^{D}-s_{i}}{s_{i}^{D}-s_{i,D}}\right\vert ^{\rho }+\left\vert \frac{%
s_{i}-s_{i,D}}{s_{i}^{D}-s_{i,D}}\right\vert ^{\rho }$ and $\left\vert \frac{%
t_{j}^{D}-t_{j}}{t_{j}^{D}-t_{j,D}}\right\vert ^{\rho }+\left\vert \frac{%
t_{j}-t_{j,D}}{t_{j}^{D}-t_{j,D}}\right\vert ^{\rho }$ are less than $1$, we
see that $9^{1-\rho }\sum_{i,j}\left\vert \mathbb{E}\left(
X_{s_{i},s_{i+1}}^{D}X_{t_{j},t_{j+1}}^{D}\right) \right\vert ^{\rho }$ is
less than or equal to 
\begin{eqnarray*}
&&\sum_{i,j}\left\vert \mathbb{E}\left(
X_{s_{i,D},s_{i}^{D}}X_{t_{j,D},t_{j}^{D}}\right) \right\vert ^{\rho
}+\sum_{i,j}\left\vert \mathbb{E}\left(
X_{s_{i,D},s_{i}^{D}}X_{t_{j}^{D},t_{j+1,D}}\right) \right\vert ^{\rho } \\
&&+\sum_{i,j}\left\vert \mathbb{E}\left(
X_{s_{i}^{D},s_{i+1,D}}X_{t_{j,D},t_{j}^{D}}\right) \right\vert ^{\rho
}+\sum_{i,j}\left\vert \mathbb{E}\left(
X_{s_{i}^{D},s_{i+1,D}}X_{t_{j}^{D},t_{j+1,D}}\right) \right\vert ^{\rho }.
\end{eqnarray*}%
itself bound $\left\vert R\right\vert _{\rho -var,\left[ u_{1},v_{1}\right]
\times \left[ u_{2},v_{2}\right] }^{\rho }.$

The second estimate is a bit more subtle. Take $s,t\in \left[ 0,1\right] $,
with $s_{D}\leq s,t\leq s^{D},$ $u,v\in D,$ $\left( s_{i}\right) $ and $%
\left( t_{j}\right) $ subdivisions of $\left[ s,t\right] $ and $\left[ u,v%
\right] .$ Then, if $h_{t}^{i,D}=\mathbb{E}\left(
X_{s_{i},s_{i+1}}^{D}X_{t}^{D}\right) ,$ we know from Proposition \ref%
{CM_pVar_embedding} that 
\begin{eqnarray*}
\left\vert h^{i,D}\right\vert _{\rho -var,\left[ u,v\right] } &\leq
&\left\vert R^{D}\right\vert _{\rho -var,\left[ u,v\right] }^{1/2}\mathbb{E}%
\left( \left\vert X_{s,t}^{D}\right\vert ^{2}\right) ^{1/2} \\
&\leq &9^{\rho -1}\frac{s_{i+1}-s_{i}}{s^{D}-s_{D}}\left\vert R\right\vert
_{\rho -var,\left[ u,v\right] ^{2}}^{1/2}\mathbb{E}\left( \left\vert
X_{s_{D},s^{D}}\right\vert ^{2}\right) ^{1/2}
\end{eqnarray*}%
Hence, for a fixed $i,$ 
\begin{eqnarray*}
\sum_{j}\left\vert \mathbb{E}\left(
X_{s_{i}s_{i+1}}^{D}X_{t_{j},t_{j+1}}^{D}\right) \right\vert ^{\rho } &\leq
&\left\vert h^{i,D}\right\vert _{\rho \text{$-var$;}\left[ u,v\right]
}^{\rho } \\
&\leq &\left( 9^{\rho -1}\frac{s_{i+1}-s_{i}}{s^{D}-s_{D}}\left\vert
R\right\vert _{\rho \text{$-var$}\left[ u,v\right] ^{2}}^{1/2}\mathbb{E}%
\left( \left\vert X_{s_{D},s^{D}}\right\vert ^{2}\right) ^{1/2}\right)
^{\rho }.
\end{eqnarray*}%
Summing over $i$ and taking the supremum over all subdivision ends the proof
of the second estimate. We leave the easy proof of the third estimate to the
reader.
\end{proof}

\begin{corollary}
\bigskip \label{PropUnifRhoVarforRD}Let $X$ be continuous centered
real-valued continuous Gaussian process on $\left[ 0,1\right] $ with
covariance $R$ assumed to be of finite $\rho $-variation. Then,\newline
(i) for $s,t\in D,$ the $\rho $-variation of $R^{D}$, the covariance of $%
X^{D}$, is bounded by $9^{\rho -1}\left\vert R\right\vert _{\rho \text{$-var$%
;}\left[ s,t\right] ^{2}}.$\newline
(ii) for all $s,t,u,v\in \left[ 0,1\right] $ and $\rho ^{\prime }>\rho $ the 
$\rho ^{\prime }$-variation of $R^{D}$ over $\left[ s,t\right] \times \left[
u,v\right] $ converges to $\left\vert R\right\vert _{\rho ^{\prime }\text{$%
-var$;}\left[ s,t\right] \times \left[ u,v\right] }$ when $\left\vert
D\right\vert \rightarrow 0.$\newline
(iii) if $\left\vert R\right\vert _{\rho \text{$-var$;}\left[ s,t\right]
^{2}}\leq C_{\ref{PropUnifRhoVarforRD}}^{R}\left\vert t-s\right\vert
^{1/\rho },$ then, $\left\vert R^{D}\right\vert _{\rho \text{$-var$;}\left[
s,t\right] ^{2}}\leq 9C_{\ref{PropUnifRhoVarforRD}}^{R}\left\vert
t-s\right\vert ^{1/\rho }.$\newline
The same estimates apply to the covariance of $\left( X,X^{D}\right) .$
\end{corollary}

\begin{proof}
The first statement is an easy corollary of the previous lemma. For the
second we note that, by interpolation, $R^{D}\rightarrow R$ in $\rho
^{\prime }$-variation for any $\rho ^{\prime }>\rho $ so that%
\begin{equation*}
\left\vert R^{D}\right\vert _{\rho ^{\prime }\text{$-var$;}\left[ s,t\right]
\times \left[ u,v\right] }\rightarrow \left\vert R\right\vert _{\rho
^{\prime }\text{$-var$;}\left[ s,t\right] \times \left[ u,v\right] }\text{
with }\left\vert D\right\vert \rightarrow 0.
\end{equation*}%
(Note that we do not have $R^{D}\rightarrow R$ in $\rho $-variation in
general but see remark below.) For the third one, without loss of
generalities, we assume that $C_{\ref{PropUnifRhoVarforRD}}^{R}=1.$ Then, by
subadditivity of the $\rho $-variation at the power $\rho ,$%
\begin{eqnarray*}
\left\vert R^{D}\right\vert _{\rho -var,\left[ s,t\right] ^{2}}^{\rho }
&\leq &\left\vert R^{D}\right\vert _{\rho -var,\left[ s,s^{D}\right]
^{2}}^{\rho }+\left\vert R^{D}\right\vert _{\rho -var,\left[ s,s^{D}\right]
\times \left[ s^{D},t_{D}\right] }^{\rho } \\
&&+\left\vert R^{D}\right\vert _{\rho -var,\left[ s,s^{D}\right] \times %
\left[ t_{D},t\right] }^{\rho }+\left\vert R^{D}\right\vert _{\rho -var,%
\left[ s^{D},t_{D}\right] \times \left[ s,s^{D}\right] }^{\rho } \\
&&+\left\vert R^{D}\right\vert _{\rho -var,\left[ s^{D},t_{D}\right] \times %
\left[ s^{D},t_{D}\right] }^{\rho }+\left\vert R^{D}\right\vert _{\rho -var,%
\left[ s^{D},t_{D}\right] \times \left[ t_{D},t\right] }^{\rho } \\
&&+\left\vert R^{D}\right\vert _{\rho -var,\left[ t_{D},t\right] \times %
\left[ s,s^{D}\right] }^{\rho }+\left\vert R^{D}\right\vert _{\rho -var,%
\left[ t_{D},t\right] \times \left[ s^{D},t_{D}\right] }^{\rho } \\
&&+\left\vert R^{D}\right\vert _{\rho -var,\left[ t_{D},t\right] ^{2}}^{\rho
}.
\end{eqnarray*}%
We bound each term using the previous lemma, and using on top estimates of
the type%
\begin{eqnarray*}
\left\vert \frac{t-s}{s^{D}-s_{D}}\right\vert \mathbb{E}\left( \left\vert
X_{s_{D},s^{D}}\right\vert ^{2}\right) ^{1/2} &\leq &\left\vert \frac{t-s}{%
s^{D}-s_{D}}\right\vert \left\vert s^{D}-s_{D}\right\vert ^{1/\left( 2\rho
\right) } \\
&=&\left\vert \frac{t-s}{s^{D}-s_{D}}\right\vert ^{1-1/\rho }\left\vert
t-s\right\vert ^{1/\left( 2\rho \right) } \\
&\leq &\left\vert t-s\right\vert ^{1/\rho }.
\end{eqnarray*}%
We leave the extension to the estimates on the covariation of $\left(
X,X^{D}\right) $ to the reader.
\end{proof}

\section{Multidimensional Gaussian Processes}

As remarked in the introduction, any $\mathbb{R}^{d}$-valued centered
Gaussian process $X=\left( X^{1},...,X^{d}\right) $ with continuous sample
paths gives rise to an abstract Wiener space $\left( E,\mathcal{H},\mathbb{P}%
\right) $ with $E=C\left( \left[ 0,1\right] ,\mathbb{R}^{d}\right) $ and $%
\mathcal{H}\subset C\left( \left[ 0,1\right] ,\mathbb{R}^{d}\right) $. If $%
\mathcal{H}^{i}$ denotes the Cameron-Martin space associated to the one
dimensional Gaussian process $X^{i}$ and all $\left\{
X^{i}:i=1,...,d\right\} $ are independent then $\mathcal{H\cong \oplus }%
_{i=1}^{d}\mathcal{H}^{i}$.

\subsection{Wiener Chaos}

Given an abstract Wiener space, there is a decomposition of $L^{2}\left( 
\mathbb{P}\right) $ known as Wiener-It\^{o} chaos decomposition, see \cite%
{ledoux-1996}, \cite{nualart-1995} or \cite{revuz-yor-1999} for the case of
Wiener measure. Our interest in\ Wiener chaos comes from the following
simple fact.

\begin{proposition}
\label{PropGroupIncrInWIC}Assume the $\mathbb{R}^{d}$-valued continuous
centered Gaussian process $X=\left( X^{1},...,X^{d}\right) $ has sample
paths of finite variation and let $S_{N}\left( X\right) \equiv \mathbf{X}$
denote its natural lift to a process with values in $G^{N}\left( \mathbb{R}%
^{d}\right) \subset T^{N}\left( \mathbb{R}^{d}\right) $.Then, for $n=1,...N$
and any $s,t\in \left[ 0,1\right] $ the random variable $\pi _{n}\left( 
\mathbf{X}_{s,t}\right) $ is an element in the $n^{\text{th}}$ (in general,
not homogenous) Wiener chaos\footnote{%
Stricltly speaking, the $\left( \mathbb{R}^{d}\right) ^{\otimes n}$-valued
chaos.}.
\end{proposition}

\begin{proof}
$\pi _{n}\left( \mathbf{X}\right) $ is given by $n$ iterated integrals which
can be written out in terms of (a.s. convergent) Riemann-Stieltjes sums.
Each such Riemann-Stieljes sum is a polynomial of degree at most $n$ and of
variables of form $X_{s,t}$. It now suffices to remark that the $n^{\text{th}%
}$ Wiener chaos contains all such polynomials and is closed under
convergence in probability.
\end{proof}

As a consequence of the hypercontractivity property of the
Ornstein-Uhlenbeck semigroup, $L^{p}$- and $L^{q}$-norms are equivalent on
the $n^{\text{th}}$ Wiener chaos. Usually this is stated for the \textit{%
homogenous} chaos, \cite{ledoux-1996} \cite{nualart-1995}, but the extension
to the $n^{\text{th}}$ (non-homogenous) chaos is not difficult, at least if
we do not worry too much about optimal constants.

\begin{lemma}
\label{LpLqEquivalence}Let $n\in \mathbb{N}$ and $Z$ be a random variable in
the $n^{th}$ Wiener chaos. Assume $1<p<q<\infty $. Then 
\begin{equation*}
\left\vert Z\right\vert _{L^{p}}\leq \left\vert Z\right\vert _{L^{q}}\leq
\left\vert Z\right\vert _{L^{p}}\left( n+1\right) \left( q-1\right)
^{n/2}\max \left( 1,\left( p-1\right) ^{-n}\right) .
\end{equation*}%
In particular, for $q>2$,%
\begin{equation*}
\left\vert Z\right\vert _{L^{2}}\leq \left\vert Z\right\vert _{L^{q}}\leq
\left\vert Z\right\vert _{L^{2}}\left( n+1\right) \left( q-1\right) ^{n/2}
\end{equation*}
\end{lemma}

\begin{proof}
Only the second inequality requires proof. We need two well-known facts,
both found in \cite{nualart-1995} for instance. First, if $Z_{k}$ is a
random variable in\ the $k^{th}$ homogeneous Wiener chaos then 
\begin{equation}
\left\vert Z_{k}\right\vert _{L^{q}}\leq \left( \frac{q-1}{p-1}\right) ^{%
\frac{k}{2}}\left\vert Z_{k}\right\vert _{L^{p}}.  \label{HyperContrEstimate}
\end{equation}%
Secondly, the $L^{2}$-projection on the $k^{th}$ homogeneous chaos, denoted
by $J_{k}$, is a bounded operator from $L^{p}\rightarrow L^{p}$ for any $%
1<p<\infty ;$ more precisely\footnote{%
In fact, this is a simple consequence of (\ref{HyperContrEstimate}) when $%
p>2 $ and combined with a duality argument for $1<p<2$.}%
\begin{equation*}
\left\vert J_{k}Z\right\vert _{L^{p}}\leq \left\{ 
\begin{array}{c}
\left( p-1\right) ^{k/2}\left\vert Z\right\vert _{L^{p}}\,\,\,\text{if }%
p\geq 2 \\ 
\left( p-1\right) ^{-k/2}\left\vert Z\right\vert _{L^{p}}\,\,\,\text{if }p<2.%
\end{array}%
\right.
\end{equation*}%
From $Z=\sum_{k=0}^{n}J_{k}Z,$ we have $\left\vert Z\right\vert _{L^{q}}\leq
\sum_{k=0}^{n}\left\vert J_{k}Z\right\vert _{L^{q}}$ and hence 
\begin{eqnarray*}
\left\vert Z\right\vert _{L^{q}} &\leq &\sum_{k=0}^{n}\left( \frac{q-1}{p-1}%
\right) ^{\frac{k}{2}}\left\vert J_{k}Z\right\vert _{L^{p}} \\
&\leq &\left\vert Z\right\vert _{L^{p}}\sum_{k=0}^{n}\left\{ 
\begin{array}{c}
\left( q-1\right) ^{k/2}\,\,\,\,\,\,\,\,\,\,\,\,\,\,\,\,\,\text{if }p\geq 2
\\ 
\left( q-1\right) ^{k/2}\left( p-1\right) ^{-k}\,\,\,\text{if }p<2.%
\end{array}%
\right. \\
&\leq &\left\vert Z\right\vert _{L^{p}}\left( n+1\right) \left( q-1\right)
^{n/2}\max \left( 1,\left( p-1\right) ^{-n}\right) \text{.}
\end{eqnarray*}
\end{proof}

Here is a immediate, yet useful, application. Assume $Z,W$ are in the $%
n^{th} $ Wiener chaos. Then there exists $C=C\left( n\right) $ 
\begin{equation}
\left\vert WZ\right\vert _{L^{2}}\leq C\left\vert W\right\vert
_{L^{2}}\left\vert Z\right\vert _{L^{2}}.  \label{L2splitup}
\end{equation}%
(There is nothing special about $L^{2}$ here, but this is how we usually use
it.) We now discuss more involved corollaries.

\begin{corollary}
\label{ChaosAndGroup}Let $g$ be a random element of $G^{N}\left( \mathbb{R}%
^{d}\right) $ such that for all $1\leq n\leq N$ the projection $\pi
_{n}\left( g\right) $ is an element of the $n^{th}$ Wiener chaos. Let $%
\delta $ be a positive real. Then, the following statements 1-6 are
equivalent.\newline
(i) There exists a constant $C_{1}>0$ such that for all $n=1,...,N$ there
exists $q=q\left( n\right) \in \left( 1,\infty \right) :$ $\left\vert \pi
_{n}\left( g\right) \right\vert _{L^{q}}\leq C_{1}\delta ^{n};$\newline
(ii) There exists a constant $C_{2}>0$ such that for all $n=1,...,N$ and for
all $q\in \lbrack 1,\infty ):$ $\left\vert \pi _{n}\left( g\right)
\right\vert _{L^{q}}\leq C_{2}q^{\frac{n}{2}}\delta ^{n};$\newline
(iii) There exists a constant $C_{3}>0$ such that for all $n=1,...,N$ there
exists $q=q\left( n\right) \in \left( 1,\infty \right) :$ $\left\vert \pi
_{n}\left( \log \left( g\right) \right) \right\vert _{L^{q}}\leq C_{3}\delta
^{n};$\newline
(iv) There exists a constant $C_{4}>0$ such that for all $n=1,...,N$ and for
all $q\in \lbrack 1,\infty ):$ $\left\vert \pi _{n}\left( \log \left(
g\right) \right) \right\vert _{L^{q}}\leq C_{4}q^{\frac{n}{2}}\delta ^{n};$%
\newline
(v) There exists a constant $C_{5}>0$ and there exists $q\in \left( N,\infty
\right) :$ $\mathbb{E}\left( \left\Vert g\right\Vert ^{q}\right) ^{1/q}\leq
C_{5}\delta ;$\newline
(vi) There exists a constant $C_{6}>0$ such that for all $q\in ,$ $\mathbb{E}%
\left( \left\Vert g\right\Vert ^{q}\right) ^{1/q}\leq C_{6}q^{\frac{1}{2}%
}\delta .$\newline
When switching from $i^{th}$ to the $j^{th}$ statement, the constant $C_{j}$
depends only on $C_{i},N$ and $d$.
\end{corollary}

\begin{remark}
The restrictions on $q$ in statements 1,3,5 comes from Lemma \ref%
{LpLqEquivalence} where equivalence of $L^{p}$- and $L^{q}$-norms (on the $%
n^{\text{th}}$ Wiener chaos) is shown only for $p,q>1$. In fact, this
equivalence holds true for all $0<p<q<\infty $ (and hence statements 1,3,5
can be formulated with $q\in (0,\infty )$). This follows from the work of C.
Borell \cite{borell-1978, borell-1984, borell-1984-casewestern} and is easy
to see if one accepts a results of Schreiber \cite{schreiber-1969} that
convergence in probability and in $L^{q}$ are equivalent on the $n^{\text{th}%
}$ Wiener chaos. Indeed, first note that for any $p>0$, $L^{p}$-convergence
implies convergence in probability and hence in $L^{q}$ so that the identity
map from $L^{p}\rightarrow L^{q}$ is continuous. Assume it is not bounded.
Then there exists a sequence of random variables $\left( Z^{n}\right) $ such
that $\left\vert Z^{n}\right\vert _{L^{q}}>n\left\vert Z^{n}\right\vert
_{L^{p}}$. But $W^{n}:=Z^{n}/\left\vert Z^{n}\right\vert _{L^{q}}$ satisfies 
$1/n>\left\vert W^{n}\right\vert _{L^{p}}$and hence converges to $0$ in $%
L^{p}$ which contradicts $\left\vert W^{n}\right\vert _{L^{q}}\equiv 1$.
\end{remark}

\begin{proof}
Clearly, (vi)$\Longrightarrow $(v), (iv)$\Longrightarrow $(iii), (ii)$%
\Longrightarrow $(i), and Lemma \ref{LpLqEquivalence} shows (iii)$%
\Longrightarrow $(iv), and (i)$\Longrightarrow $(ii). it is therefore enough
to prove (ii)$\Longrightarrow $(vi), (v)$\Longrightarrow $(i), and (ii)$%
\Longrightarrow $(iv).\newline
\underline{(ii)$\Longrightarrow $(vi):} By equivalence of homogeneous norm,
there exists a constant $C>0$ such that,%
\begin{equation*}
\left\Vert g\right\Vert \leq C\max_{n=1,...,N}\left\vert \pi _{n}\left(
g\right) \right\vert ^{1/n},
\end{equation*}%
so that, 
\begin{eqnarray*}
\mathbb{E}\left( \left\Vert g\right\Vert ^{q}\right) ^{1/q} &\leq &C\mathbb{E%
}\left( \max_{n=1,...,N}\left\vert \pi _{n}\left( g\right) \right\vert
^{q/n}\right) ^{1/q}\leq C\left( \sum_{n=1}^{N}\mathbb{E}\left( \left\vert
\pi _{n}\left( g\right) \right\vert ^{q/n}\right) \right) ^{1/q} \\
&\leq &C\left( \sum_{n=1}^{N}\mathbb{E}\left( \left\vert \pi _{n}\left(
g\right) \right\vert ^{q/n}\right) \right) ^{1/q}\leq C\left(
\sum_{n=1}^{N}C_{2}^{q/n}\left( q^{\frac{n}{2}}\delta ^{n}\right)
^{q/n}\right) ^{1/q} \\
&\leq &C_{6}q^{\frac{1}{2}}\delta .
\end{eqnarray*}%
\underline{(v)$\Longrightarrow $(i):} By equivalence of homogeneous norm,
there exists a constant $c>0$ such that,%
\begin{equation*}
\left\vert \pi _{n}\left( g\right) \right\vert ^{1/n}\leq c\left\Vert
g\right\Vert .
\end{equation*}%
Hence,%
\begin{equation*}
\mathbb{E}\left( \left\vert \pi _{n}\left( g\right) \right\vert
^{q_{0}/n}\right) ^{n/q_{0}}\leq c^{n}\mathbb{E}\left( \left\Vert
g\right\Vert ^{q_{0}}\right) ^{n/q_{0}}\leq \left\vert cC\right\vert
^{N}\delta ^{n}.
\end{equation*}%
\underline{(ii)$\Longrightarrow $(iv):} \ An easy consequence of (\ref%
{L2splitup}) and the formulas%
\begin{eqnarray*}
\pi _{n}\left( g\right) &=&\sum_{\substack{ k_{1},\ldots ,k_{l}  \\ %
\sum_{i}k_{i}=n}}a_{k_{1},\ldots ,k_{l}}\bigotimes_{i}\pi _{i}\left( \ln
g\right) , \\
\pi _{n}\left( \ln g\right) &=&\sum_{\substack{ k_{1},\ldots ,k_{l}  \\ %
\sum_{i}k_{i}=n}}b_{k_{1},\ldots ,k_{l}}\bigotimes_{i}\pi _{i}\left(
g\right) ,
\end{eqnarray*}%
where the real coefficients $a_{k_{1},\ldots ,k_{l}}$ and $b_{k_{1},\ldots
,k_{l}}$ can be explicitly computed from the power series definition of $\ln 
$ and $\exp .$
\end{proof}

\begin{proposition}
\label{PropWienerItoAndGRR}Let $\mathbf{X}$ be a continuous $G^{N}\left( 
\mathbb{R}^{d}\right) $-valued stochastic process. Assume that for all $s<t$
in $\left[ 0,1\right] $ and $n=1,...,N$, the projection $\pi _{n}\left( 
\mathbf{X}_{s,t}\right) $ is an element in the $n^{th\text{ }}$Wiener chaos
and that, for some constant $C$ and 1D control function $\omega $,%
\begin{equation}
\left\vert \pi _{n}\left( \ln \mathbf{X}_{s,t}\right) \right\vert
_{L^{2}}\leq C\omega \left( s,t\right) ^{\frac{n}{2\rho }}.
\label{AssumptionPInX}
\end{equation}%
Then exists a constant $C^{\prime }=C^{\prime }\left( \rho ,N\right) $ such
that for all $q\in \lbrack 1,\infty )$ 
\begin{equation}
\left\vert d\left( \mathbf{X}_{s},\mathbf{X}_{t}\right) \right\vert
_{L^{q}}\leq C^{\prime }\sqrt{q}\omega \left( s,t\right) ^{\frac{1}{2\rho }};
\label{Lqdxsxt}
\end{equation}

\begin{enumerate}
\item[(i)] If $p>2\rho $ then $\left\Vert \mathbf{X}\right\Vert _{p\text{$%
-var$;}\left[ 0,1\right] }$ has a Gauss tail i.e. there exists $\eta =\eta
\left( p,\rho ,N,K\right) >0$, with $\omega \left( \left[ 0,1\right]
^{2}\right) \leq K$, such that%
\begin{equation*}
\mathbb{E}\left( e^{\eta \left\Vert \mathbf{X}\right\Vert _{p\text{$-var$;}%
\left[ 0,1\right] }^{2}}\right) <\infty .
\end{equation*}%
In particular, $\mathbf{X}$ has a.s. sample paths of finite $p$-variation.

\item[(ii)] If $\omega \left( s,t\right) \leq K\left\vert t-s\right\vert ,$
then $\left\Vert \mathbf{X}\right\Vert _{p\text{$-var$}}$ above may be
replaced by $\left\Vert \mathbf{X}\right\Vert _{1/p-H\ddot{o}l}$ and $%
\mathbf{X}$ has a.s. sample paths of $1/p$-H\"{o}lder regularity.
\end{enumerate}
\end{proposition}

\begin{proof}
Equation (\ref{Lqdxsxt}) is a clear consequence of the corollary \ref%
{ChaosAndGroup}. The rest follows from Corollary \ref{GRRBesovEmbLq} in
Appendix II.
\end{proof}

The same argument, but using corollary \ref{BesovDistanceLq} in Appendix II
leads to:

\begin{proposition}
\label{PropWienerItoAndGRRcontinuity}Let $\mathbf{X,Y}$ be two continuous $%
G^{N}\left( \mathbb{R}^{d}\right) $-valued stochastic processes. Assume that
for all $s<t\ $in $\left[ 0,1\right] $ and $n=1,...,N$ the projection $\pi
_{n}\left( \mathbf{X}_{s,t}^{-1}\otimes \mathbf{Y}_{s,t}\right) $ is an
element in the $n^{th\text{ }}$Wiener chaos and that, for some $C>0$, $%
\varepsilon \in \lbrack 0,1)$ and 1D control function $\omega $,%
\begin{eqnarray}
\left\vert \pi _{n}\left( \ln \mathbf{X}_{s,t}\right) \right\vert
_{L^{2}},\left\vert \pi _{n}\left( \ln \mathbf{Y}_{s,t}\right) \right\vert
_{L^{2}} &\leq &C\omega \left( s,t\right) ^{\frac{n}{2\rho }},
\label{AssXom} \\
\left\vert \pi _{n}\left( \ln \left( \mathbf{X}_{s,t}^{-1}\otimes \mathbf{Y}%
_{s,t}\right) \right) \right\vert _{L^{2}} &\leq &C\varepsilon \omega \left(
s,t\right) ^{\frac{n}{2\rho }}  \label{AssXYom}
\end{eqnarray}%
Then for all $q\in \lbrack 1,\infty )$ there exists a constant $C^{\prime
}=C^{\prime }\left( \rho ,N,C\right) >0$ such that 
\begin{equation}
\left\Vert d\left( \mathbf{X}_{s,t},\mathbf{Y}_{s,t}\right) \right\Vert
_{L^{q}}\leq C^{\prime }\varepsilon ^{\frac{1}{N}}\sqrt{q}\omega \left(
s,t\right) ^{\frac{1}{2\rho }};  \label{Lqdxstyst}
\end{equation}

\begin{enumerate}
\item[(i)] If $p>2\rho $ and then there exist positive constants $\theta
=\theta (p,\rho ,N)$ and $C^{\prime \prime }=C^{\prime \prime }\left( p,\rho
,N,C,K\right) $ with $\omega \left( \left[ 0,1\right] ^{2}\right) \leq K$
such that%
\begin{equation*}
\left\vert d_{p\text{$-var$;}\left[ 0,1\right] }\left( \mathbf{X,Y}\right)
\right\vert _{L^{q}}\leq C^{\prime \prime }\varepsilon ^{\theta }\sqrt{q}.
\end{equation*}

\item[(ii)] If $\omega \left( s,t\right) \leq K\left\vert t-s\right\vert $,
then $d_{p\text{$-var$;}\left[ 0,1\right] }\left( \mathbf{X,Y}\right) $
above may be replaced by $d_{1/p-H\ddot{o}l}\left( \mathbf{X,Y}\right) $.
\end{enumerate}
\end{proposition}

\subsection{Uniform Estimates For Lifts of Piecewise\ Linear Gaussian
Processes}

We recall that all Gaussian processes under consideration are defined on $%
\left[ 0,1\right] $, centered and with continuous sample paths. The aim of
this section is to construct the lift of $X=\left( X^{1},...,X^{d}\right) $
for $X^{1},...X^{d}$ independent, provided that the covariance function for
each $X^{i}$ has finite $\rho $-variation for some $\rho \in \lbrack 1,2).$

The proof of the following lemma is left to the reader.

\begin{lemma}
Let $\left( X_{1},\ldots ,X_{d}\right) $ be a $d$-dimensional Gaussian
process, with covariance $R$ of finite $\rho $-variation controlled by $%
\omega $. Then, for every fixed $\alpha =\left( \alpha _{1},\ldots ,\alpha
_{d}\right) \in \mathbb{R}^{d}$, the covariance of 
\begin{equation*}
\alpha _{1}X_{1}+\ldots +\alpha _{d}X_{d}
\end{equation*}%
has finite $\rho $-variation controlled by $\omega $ times a constant
depending on $\alpha .$
\end{lemma}

\begin{proposition}
\label{p-variationE(X2)}Let $\left( X,Y\right) $ be a $2$ dimensional
centered Gaussian process with covariance $R$ of finite $\rho $-variation
controlled by $\omega $. Then, for fixed $s<t\ $in $\left[ 0,1\right] $, the
function%
\begin{equation*}
\left( u,v\right) \in \left[ s,t\right] ^{2}\mapsto f\left( u,v\right) :=%
\mathbb{E}\left( X_{s,u}Y_{s,u}X_{s,v}Y_{s,v}\right)
\end{equation*}%
satisfies $f\left( s,\cdot \right) =f\left( \cdot ,s\right) =0$ and has
finite $\rho $-variation.\ More precisely, there exists a constant $C_{\ref%
{p-variationE(X2)}}=C_{\ref{p-variationE(X2)}}\left( \rho \right) $ such
that 
\begin{equation*}
\left\vert f\right\vert _{\rho \text{$-var$;}\left[ s,t\right] ^{2}}^{\rho
}\leq C_{\ref{p-variationE(X2)}}\omega \left( \left[ s,t\right] ^{2}\right)
^{2}.
\end{equation*}
\end{proposition}

\begin{proof}
We fix $u<u^{\prime },$ $v<v^{\prime },$ all in $\left[ s,t\right] .$ Using 
\begin{equation*}
X_{s,u^{\prime }}Y_{s,u^{\prime }}-X_{s,u}Y_{s,u}=X_{u,u^{\prime
}}Y_{s,u^{\prime }}+X_{s,u}Y_{u,u^{\prime }},
\end{equation*}%
we bound $\left\vert \mathbb{E}\left( \left( X_{s,u^{\prime }}Y_{s,u^{\prime
}}-X_{s,u}Y_{s,u}\right) \left( X_{s,v^{\prime }}Y_{s,v^{\prime
}}-X_{s,v}Y_{s,v}\right) \right) \right\vert $ by%
\begin{eqnarray*}
&&\left\vert \mathbb{E}\left( X_{u,u^{\prime }}Y_{s,u^{\prime
}}X_{v,v^{\prime }}Y_{s,v^{\prime }}\right) \right\vert +\left\vert \mathbb{E%
}\left( X_{s,u}Y_{u,u^{\prime }}X_{v,v^{\prime }}Y_{s,v^{\prime }}\right)
\right\vert \\
&&+\left\vert \mathbb{E}\left( X_{u,u^{\prime }}Y_{s,u^{\prime
}}X_{s,v}Y_{v,v^{\prime }}\right) \right\vert +\left\vert \mathbb{E}\left(
X_{s,u}Y_{u,u^{\prime }}X_{s,v}Y_{v,v^{\prime }}\right) \right\vert
\end{eqnarray*}%
To bound the second expression for example, we use a well-known identity for
the product of Gaussian random variables,%
\begin{eqnarray*}
\mathbb{E}\left( X_{s,u}Y_{u,u^{\prime }}X_{v,v^{\prime }}Y_{s,v^{\prime
}}\right) &=&\mathbb{E}\left( X_{s,u}Y_{u,u^{\prime }}\right) \mathbb{E}%
\left( X_{v,v^{\prime }}Y_{s,v^{\prime }}\right) \\
&&+\mathbb{E}\left( X_{s,u}X_{v,v^{\prime }}\right) \mathbb{E}\left(
Y_{u,u^{\prime }}Y_{s,v^{\prime }}\right) \\
&&+\mathbb{E}\left( X_{s,u}Y_{s,v^{\prime }}\right) \mathbb{E}\left(
X_{v,v^{\prime }}Y_{u,u^{\prime }}\right) ,
\end{eqnarray*}%
to obtain 
\begin{eqnarray*}
\frac{1}{C_{\rho }}\left\vert \mathbb{E}\left( X_{s,u}Y_{u,u^{\prime
}}X_{v,v^{\prime }}Y_{s,v^{\prime }}\right) \right\vert ^{\rho } &\leq
&\omega \left( \left[ s,u\right] \times \left[ u,u^{\prime }\right] \right)
\omega \left( \left[ v,v^{\prime }\right] \times \left[ s,v^{\prime }\right]
\right) \\
&&+\omega \left( \left[ s,u\right] \times \left[ v,v^{\prime }\right]
\right) \omega \left( \left[ u,u^{\prime }\right] \times \left[ s,v^{\prime }%
\right] \right) \\
&&+\omega \left( \left[ s,u\right] \times \left[ s,v^{\prime }\right]
\right) \omega \left( \left[ u,u^{\prime }\right] \times \left[ v,v^{\prime }%
\right] \right) \\
&\leq &\omega \left( \left[ s,t\right] \times \left[ u,u^{\prime }\right]
\right) \omega \left( \left[ v,v^{\prime }\right] \times \left[ s,t\right]
\right) \\
&&+\omega \left( \left[ s,t\right] \times \left[ v,v^{\prime }\right]
\right) \omega \left( \left[ u,u^{\prime }\right] \times \left[ s,t\right]
\right) \\
&&+\omega \left( \left[ s,t\right] \times \left[ s,t\right] \right) \omega
\left( \left[ u,u^{\prime }\right] \times \left[ v,v^{\prime }\right] \right)
\end{eqnarray*}%
Working similarly with all terms, we obtain that this last expression
controls the $\rho $-variation of $\left( u,v\right) \in \left[ s,t\right]
^{2}\rightarrow \mathbb{E}\left( X_{s,u}Y_{s,u}X_{s,v}Y_{s,v}\right) ,$ and
the bound on the $\rho $-variation on $\left[ s,t\right] ^{2}.$
\end{proof}

\begin{proposition}
Assume $X=\left( X^{1},\ldots ,X^{d}\right) $ is a centered continuous
Gaussian process with independent components with piecewise linear sample
paths. Let $\rho \in \lbrack 1,2)$ and assume that the covariance of $X$ is
of finite $\rho $-variation dominated by a 2D control $\omega $. \ Let $%
\mathbf{X}=S_{3}\left( X\right) $ denote the natural lift of $X$ to a $%
G^{3}\left( \mathbb{R}^{d}\right) $-valued process. There exists $C=C\left(
\rho \right) $ such that for all $s<t$ in $\left[ 0,1\right] $ and indices $%
i,j,k\in \left\{ 1,...d\right\} ,$%
\begin{eqnarray*}
\text{(i)}\,\ \mathbb{E}\left( \left\vert X_{s,t}^{i}\right\vert ^{2}\right)
&\leq &\omega \left( \left[ s,t\right] ^{2}\right) ^{1/\rho }\text{ for all }%
i\text{;} \\
\text{(ii) }\mathbb{E}\left( \left\vert \mathbf{X}_{s,t}^{i,j}\right\vert
^{2}\right) &\leq &C\omega \left( \left[ s,t\right] ^{2}\right) ^{2/\rho }%
\text{ for }i,j\text{ distinct;} \\
\text{(iii.1) }\mathbb{E}\left( \left\vert \mathbf{X}_{s,t}^{i,i,j}\right%
\vert ^{2}\right) &\leq &C\omega \left( \left[ s,t\right] ^{2}\right)
^{3/\rho }\text{ for }i,j\text{ distinct;} \\
\text{(iii.2) }\mathbb{E}\left( \left\vert \mathbf{X}_{s,t}^{i,j,k}\right%
\vert ^{2}\right) &\leq &C\omega \left( \left[ s,t\right] ^{2}\right)
^{3/\rho }\text{ for }i,j,k\text{ distinct.}
\end{eqnarray*}
\end{proposition}

\begin{proof}
(i) is obvious. For (ii) fix $i\neq j$ and $s<t,\,s^{\prime }<t^{\prime }$.\
Then, using independence of $X^{i}$ and $X^{j}$,%
\begin{eqnarray*}
\mathbb{E}\left( \mathbf{X}_{s,t}^{i,j}\mathbf{X}_{s^{\prime },t^{\prime
}}^{i,j}\right) &=&\mathbb{E}\left( \int_{s}^{t}\int_{s^{\prime
}}^{t^{\prime }}X_{s,u}^{i}X_{s^{\prime },v}^{i}dX_{u}^{j}dX_{v}^{j}\right)
\\
&=&\int_{s}^{t}\int_{s^{\prime }}^{t^{\prime }}\mathbb{E}\left(
X_{s,u}^{i}X_{s^{\prime },v}^{i}\right) d\mathbb{E}\left(
X_{u}^{j}X_{v}^{j}\right) \text{ } \\
&=&\int_{s}^{t}\int_{s^{\prime }}^{t}\left[ R_{i}\left( u,v\right)
-R_{i}\left( s,v\right) -R_{i}\left( u,s^{\prime }\right) +R_{i}\left(
s,s^{\prime }\right) \right] dR_{j}\left( u,v\right) \\
&\leq &C\omega \left( \left[ s,t\right] \times \left[ s^{\prime },t^{\prime }%
\right] \right) ^{2/\rho }\text{ by Young 2D estimate.}
\end{eqnarray*}%
(ii) follows trivially from setting $s=s^{\prime },t=t^{\prime }$ (the
general result will be used in the level (iii) estimates, see step 2 below).
We break up the level (iii)\ estimates in a few steps, assuming $i\neq j$
throughout.\newline
\underline{Step 1}: For fixed $s<t,s^{\prime }<t^{\prime },t^{\prime
}<u^{\prime }$ we claim that 
\begin{equation*}
\mathbb{E}\left( \mathbf{X}_{s,t}^{i,j}X_{s^{\prime },t^{\prime
}}^{i}X_{t^{\prime },u^{\prime }}^{j}\right) \leq C\omega \left( \left[ s,t%
\right] \times \left[ s^{\prime },t^{\prime }\right] \right) ^{1/\rho
}\omega \left( \left[ s,t\right] \times \left[ t^{\prime },u^{\prime }\right]
\right) ^{1/\rho }.
\end{equation*}%
Indeed, with $d\mathbb{E}\left( X_{t^{\prime },u^{\prime
}}^{j}X_{u}^{j}\right) \equiv \mathbb{E}\left( X_{t^{\prime },u^{\prime
}}^{j}\dot{X}_{u}^{j}\right) du$ we have 
\begin{equation*}
\mathbb{E}\left( \mathbf{X}_{s,t}^{i,j}X_{s^{\prime },t^{\prime
}}^{i}X_{t^{\prime },u^{\prime }}^{j}\right) =\mathbb{E}\left(
\int_{s}^{t}X_{s,u}^{i}X_{s^{\prime },t^{\prime }}^{i}X_{t^{\prime
},u^{\prime }}^{j}dX_{u}^{j}\right) =\int_{u=s}^{t}\mathbb{E}\left(
X_{s,u}^{i}X_{s^{\prime },t^{\prime }}^{i}\right) d\mathbb{E}\left(
X_{t^{\prime },u^{\prime }}^{j}X_{u}^{j}\right) .
\end{equation*}%
Since the 1D $\rho $-variation of $u\mapsto \mathbb{E}\left(
X_{s,u}^{i}X_{s^{\prime },t^{\prime }}^{i}\right) $ is controlled by $\left(
u,v\right) \mapsto \omega \left( \left[ u,v\right] \times \left[ s^{\prime
},t^{\prime }\right] \right) ,$ and similarly for $u\mapsto \mathbb{E}\left(
X_{t^{\prime },u^{\prime }}^{j}X_{u}^{j}\right) $, the (classical 1D) Young
estimate gives%
\begin{equation*}
\left\vert \int_{u=s}^{t}\mathbb{E}\left( X_{s,u}^{i}X_{s^{\prime
},t^{\prime }}^{i}\right) d\mathbb{E}\left( X_{t^{\prime },u^{\prime
}}^{j}X_{u}^{j}\right) \right\vert \leq C\omega \left( \left[ s,t\right]
\times \left[ s^{\prime },t^{\prime }\right] \right) ^{1/\rho }\omega \left( %
\left[ s,t\right] \times \left[ t^{\prime },u^{\prime }\right] \right)
^{1/\rho }.
\end{equation*}%
\underline{Step 2:} For fixed $s<t,$ we claim that the 2D\ map $\left(
u,v\right) \in \left[ s,t\right] ^{2}\mapsto \mathbb{E}\left( \mathbf{X}%
_{s,u}^{i,j}\mathbf{X}_{s,v}^{i,j}\right) $ has finite $\rho $-variation
controlled by%
\begin{equation*}
\left[ u_{1},u_{2}\right] \times \left[ v_{1},v_{2}\right] \mapsto C\omega
\left( \left[ s,t\right] ^{2}\right) \omega \left( \left[ u_{1},u_{2}\right]
\times \left[ v_{1},v_{2}\right] \right) .
\end{equation*}%
Then, using the level (ii) estimate and step 1, for $u_{1}<u_{2},$ $%
v_{1}<v_{2}$ all in $\left[ s,t\right] $,%
\begin{multline*}
\mathbb{E}\left( \left( \mathbf{X}_{s,u_{2}}^{i,j}-\mathbf{X}%
_{s,u_{1}}^{i,j}\right) \left( \mathbf{X}_{s,v_{2}}^{i,j}-\mathbf{X}%
_{s,v_{1}}^{i,j}\right) \right) \\
\left. 
\begin{array}{l}
=\mathbb{E}\left( \left( \mathbf{X}%
_{u_{1},u_{2}}^{i,j}+X_{s,u_{1}}^{i}X_{u_{1},u_{2}}^{j}\right) \left( 
\mathbf{X}_{v_{1},v_{2}}^{i,j}+X_{s,v_{1}}^{i}X_{v_{1},v_{2}}^{j}\right)
\right) \\ 
=\mathbb{E}\left( \mathbf{X}_{u_{1},u_{2}}^{i,j}\mathbf{X}%
_{v_{1},v_{2}}^{i,j}\right) \\ 
\text{ \ }+\mathbb{E}\left( \mathbf{X}%
_{u_{1},u_{2}}^{i,j}X_{s,v_{1}}^{i}X_{v_{1},v_{2}}^{j}\right) \\ 
\text{ \ }+\mathbb{E}\left( X_{s,u_{1}}^{i}X_{u_{1},u_{2}}^{j}\mathbf{X}%
_{v_{1},v_{2}}^{i,j}\right) \\ 
\text{ \ }+\mathbb{E}\left( X_{s,u_{1}}^{i}X_{s,v_{1}}^{i}\right) \mathbb{E}%
\left( X_{u_{1},u_{2}}^{j}X_{v_{1},v_{2}}^{j}\right) \\ 
\leq \omega \left( \left[ u_{1},u_{2}\right] \times \left[ v_{1},v_{2}\right]
\right) ^{2/\rho } \\ 
\text{ \ }+\omega \left( \left[ u_{1},u_{2}\right] \times \left[ s,v_{1}%
\right] \right) ^{1/\rho }\omega \left( \left[ u_{1},u_{2}\right] \times %
\left[ v_{1},v_{2}\right] \right) ^{1/\rho } \\ 
\text{ \ }+\omega \left( \left[ s,u_{1}\right] \times \left[ v_{1},v_{2}%
\right] \right) ^{1/\rho }\omega \left( \left[ u_{1},u_{2}\right] \times %
\left[ v_{1},v_{2}\right] \right) ^{1/\rho } \\ 
\text{ \ }+\omega \left( \left[ s,u_{1}\right] \times \left[ s,v_{1}\right]
\right) ^{1/\rho }\omega \left( \left[ u_{1},u_{2}\right] \times \left[
v_{1},v_{2}\right] \right) ^{1/\rho } \\ 
\leq 4\left\{ \omega \left( \left[ s,t\right] ^{2}\right) \omega \left( %
\left[ u_{1},u_{2}\right] \times \left[ v_{1},v_{2}\right] \right) \right\}
^{1/\rho }.%
\end{array}%
\right.
\end{multline*}%
(Here we used that $\omega $ can be taken symmetric.) \newline
\underline{Step 3:} We now prove the level (iii) estimates. For $i,j,k$
distinct, we have%
\begin{equation*}
\mathbb{E}\left( \left\vert \int_{s}^{t}\mathbf{X}_{s,u}^{i,j}dX_{u}^{k}%
\right\vert ^{2}\right) =\int \int_{\left[ s,t\right] ^{2}}\mathbb{E}\left( 
\mathbf{X}_{s,u}^{i,j}\mathbf{X}_{s,v}^{i,j}\right) dR_{k}\left( u,v\right) .
\end{equation*}%
By Young's 2D estimate, combined with $\rho $-variation regularity of the
integrand established in step 2, 
\begin{equation*}
\mathbb{E}\left( \left\vert \int_{s}^{t}\mathbf{X}_{s,u}^{i,j}dX_{u}^{k}%
\right\vert ^{2}\right) \leq C\omega \left( \left[ s,t\right] ^{2}\right)
^{3/\rho }.
\end{equation*}%
Secondly, the estimate (iii.2) follows from%
\begin{equation*}
\mathbb{E}\left( \left\vert \int_{s}^{t}\left( X_{s,u}^{i}\right)
^{2}dX_{u}^{k}\right\vert ^{2}\right) =\int \int_{\left[ s,t\right] ^{2}}%
\mathbb{E}\left( \left( X_{s,u}^{i}\right) ^{2}\left( X_{s,v}^{i}\right)
^{2}\right) dR_{k}\left( u,v\right)
\end{equation*}%
and Young's 2D estimate, combined with $\rho $-variation regularity of the
integrand which follows as a special case of Proposition \ref%
{p-variationE(X2)} (the full generality will be used in the next section).
\end{proof}

\begin{corollary}
With $\mathbf{X},\rho ,\omega $ as in the last proposition, there exists $%
C=C\left( \rho ,d\right) $ such that for all $s<t$ in $\left[ 0,1\right] $
and $n=1,2,3,$%
\begin{equation*}
\mathbb{E}\left( \left\vert \pi _{n}\left( \ln \mathbf{X}_{s,t}\right)
\right\vert ^{2}\right) \leq C\omega \left( \left[ s,t\right] ^{2}\right)
^{n/\rho }.
\end{equation*}
\end{corollary}

\begin{proof}
For $n=1,2$ this is an immediate consequence of (i),(ii) of the preceding
proposition. From Appendix III, $\pi _{3}\left( \ln \mathbf{X}_{s,t}\right) $
expands with respect to the basis elements $\left[ e_{i},\left[ e_{j},e_{k}%
\right] \right] $ with coefficients of only four possible types 
\begin{equation*}
\mathbf{X}_{s,t}^{i,j,k},\,\mathbf{X}_{s,t}^{i,i,j},\,\left\vert \mathbf{X}%
_{s,t}^{i}\right\vert ^{2}\mathbf{X}_{s,t}^{j},\,\mathbf{X}_{s,t}^{i}\mathbf{%
X}_{s,t}^{i,j}\text{ \ \ (}i,j,k\text{ distinct).}
\end{equation*}%
The first two are directly handled with (iii.1) and (iii.2). For the last
two we use the estimate (\ref{L2splitup}) together with (i),(ii).
\end{proof}

An application of proposition \ref{PropWienerItoAndGRR} with 1D control $%
\left( s,t\right) \mapsto \omega \left( \lbrack s,t]^{2}\right) $ leads to:

\begin{corollary}
\label{FerniquePwlinear}With $\mathbf{X},\rho ,\omega $ as in the last
proposition\footnote{%
The optimal choice for $\omega $ is the $\rho $-variation of $R$ raised to
power $\rho $.}, there exists $C=C\left( \rho ,d\right) $ such that for all $%
q\in \lbrack 1,\infty )$%
\begin{equation*}
\left\vert d\left( \mathbf{X}_{s},\mathbf{X}_{t}\right) \right\vert
_{L^{q}\left( \mathbb{P}\right) }\leq C\sqrt{q}\omega \left( \left[ s,t%
\right] ^{2}\right) ^{\frac{1}{2\rho }}.
\end{equation*}%
If $p>2\rho $ then there exists $\eta =\eta \left( p,\rho ,K\right) >0$,
with $\omega \left( \left[ 0,1\right] ^{2}\right) \leq K,$ such that 
\begin{equation}
\mathbb{E}\left( \exp \left( \eta \left\Vert \mathbf{X}\right\Vert _{p\text{$%
-var$;}\left[ 0,1\right] }^{2}\right) \right) <\infty .  \label{FerniquePvar}
\end{equation}%
If $\omega \left( s,t\right) \leq K\left\vert t-s\right\vert ,$ then $%
\left\Vert \mathbf{X}\right\Vert _{p\text{$-var$}}$ above may be replaced by 
$\left\Vert \mathbf{X}\right\Vert _{1/p-H\ddot{o}l}$.
\end{corollary}

\begin{remark}
Consider an $\mathbb{R}^{d}$-valued centered, Gaussian process $X$ with
independent components, all having covariance of finite $\rho $-variation
for some $\rho \in \lbrack 1,2)$ dominated by (for simplicity) a H\"{o}lder
dominated control. The above Fernique type estimate is plenty to see that
the family $\left\{ S_{3}\left( X^{n}\right) :n\right\} $ is tight in $C^{1/p%
\text{-H\"{o}l}}\left( \left[ 0,1\right] ,G^{3}\left( \mathbb{R}^{d}\right)
\right) $. We shall see that there is a unique limit point and, in fact,
show $L^{q}$-convergence in all $q\in \lbrack 1,\infty )$.
\end{remark}

\subsection{Continuity Estimates For Lifts of Piecewise Linear Gaussian
Processes}

\begin{proposition}
\label{Level123estimatesContEstPwLinGauss}\bigskip Let $\left( X,Y\right)
=\left( X^{1},Y^{1},\ldots ,X^{d},Y^{d}\right) $ be a centered continuous
Gaussian process with piecewise linear sample paths such that $\left(
X^{i},Y^{i}\right) $ is independent of $\left( X^{j},Y^{j}\right) $ when $%
i\neq j$. Let $\rho \in \lbrack 1,2)$ and $\omega $ a 2D control that
dominates the $\rho $-variation of the covariance of $\left( X,Y\right) $.
Assume $\rho ^{\prime }\in \left( \rho ,2\right) $ and $\omega \left( \left[
0,1\right] ^{2}\right) \leq K$. Then there exists $C_{\ref%
{Level123estimatesContEstPwLinGauss}}=C_{\ref%
{Level123estimatesContEstPwLinGauss}}\left( \rho ,\rho ^{\prime },K\right) $
such that for all $s<t$ in $\left[ 0,1\right] $ and indices $i,j,k\in
\left\{ 1,...d\right\} $%
\begin{eqnarray*}
\text{(i)}\,\ \mathbb{E}\left( \left\vert \mathbf{X}_{s,t}^{i}-\mathbf{Y}%
_{s,t}^{i}\right\vert ^{2}\right) &\leq &\left\vert R_{X-Y}\right\vert
_{\infty }^{1-\frac{\rho }{\rho ^{\prime }}}\omega \left( \left[ s,t\right]
^{2}\right) ^{1/\rho ^{\prime }}\text{ for all }i\text{;} \\
\text{(ii) }\mathbb{E}\left( \left\vert \mathbf{X}_{s,t}^{i,j}-\mathbf{Y}%
_{s,t}^{i,j}\right\vert ^{2}\right) &\leq &C_{\ref%
{Level123estimatesContEstPwLinGauss}}\left\vert R_{X-Y}\right\vert _{\infty
}^{1-\frac{\rho }{\rho ^{\prime }}}\omega \left( \left[ s,t\right]
^{2}\right) ^{2/\rho ^{\prime }}\text{ for }i,j\text{ distinct;} \\
\text{(iii.1) }\mathbb{E}\left( \left\vert \mathbf{X}_{s,t}^{i,i,j}-\mathbf{Y%
}_{s,t}^{i,i,j}\right\vert ^{2}\right) &\leq &C_{\ref%
{Level123estimatesContEstPwLinGauss}}\left\vert R_{X-Y}\right\vert _{\infty
}^{1-\frac{\rho }{\rho ^{\prime }}}\omega \left( \left[ s,t\right]
^{2}\right) ^{3/\rho ^{\prime }}\text{ for }i,j\text{ distinct;} \\
\text{(iii.2) }\mathbb{E}\left( \left\vert \mathbf{X}_{s,t}^{i,j,k}-\mathbf{Y%
}_{s,t}^{i,j,k}\right\vert ^{2}\right) &\leq &C_{\ref%
{Level123estimatesContEstPwLinGauss}}\left\vert R_{X-Y}\right\vert _{\infty
}^{1-\frac{\rho }{\rho ^{\prime }}}\omega \left( \left[ s,t\right]
^{2}\right) ^{3/\rho ^{\prime }}\text{ for }i,j,k\text{ distinct.}
\end{eqnarray*}
\end{proposition}

\begin{proof}
We first remark that interpolation inequalities work for 2D variation just
as for 1D variation, more specifically,%
\begin{eqnarray*}
\left\vert R_{X-Y}\right\vert _{\rho ^{\prime }\text{$-var$};\left[ s,t%
\right] ^{2}} &\leq &\left\vert R_{X-Y}\right\vert _{\infty }^{1-\rho /\rho
^{\prime }}\left\vert R_{X-Y}\right\vert _{\rho \text{$-var$};\left[ s,t%
\right] ^{2}}^{\rho /\rho ^{\prime }} \\
&\leq &\left\vert R_{X-Y}\right\vert _{\infty }^{1-\rho /\rho ^{\prime
}}\omega \left( \left[ s,t\right] ^{2}\right) ^{1/\rho ^{\prime }},
\end{eqnarray*}%
and that the $\rho ^{\prime }$-variation of the covariance of $\left(
X,Y\right) $ is also controlled by $\omega $. The level (i) estimate is then
simply%
\begin{equation*}
\mathbb{E}\left( \left\vert \mathbf{X}_{s,t}^{i}-\mathbf{Y}%
_{s,t}^{i}\right\vert ^{2}\right) \leq \left\vert R_{X_{i}-Y_{i}}\right\vert
_{\rho ^{\prime }\text{$-var$};\left[ s,t\right] ^{2}}\leq \left\vert
R_{X-Y}\right\vert _{\infty }^{1-\rho /\rho ^{\prime }}\omega \left( \left[
s,t\right] ^{2}\right) ^{1/\rho ^{\prime }}\text{.}
\end{equation*}%
For the level (ii) estimate fix $i\neq j$. By the triangle inequality,%
\begin{eqnarray*}
\left\vert \mathbf{X}_{s,t}^{i,j}-\mathbf{Y}_{s,t}^{i,j}\right\vert _{L^{2}}
&\leq &\left\vert \mathbf{X}_{s,t}^{i,j}-\int_{s}^{t}X_{s,u}^{i}dY_{u}^{j}%
\right\vert _{L^{2}}+\left\vert \int_{s}^{t}X_{s,u}^{i}dY_{u}^{j}-\mathbf{Y}%
_{s,t}^{i,j}\right\vert _{L^{2}} \\
&\leq &\left\vert \int_{s}^{t}X_{s,u}^{i}d\left( X_{u}^{j}-Y_{u}^{j}\right)
\right\vert _{L^{2}}+\left\vert \int_{s}^{t}\left(
X_{s,u}^{i}-Y_{s,u}^{i}\right) dY_{u}^{j}\right\vert _{L^{2}},
\end{eqnarray*}%
Using independence of $\sigma \left( X^{i},Y^{i}\right) $ and $\sigma \left(
X^{j},Y^{j}\right) $, the variances of the Riemann-Stieltjes integrals which
appear in the line above, are expressed as 2D Young integrals involving the
respective covariances. Using 2D\ Young estimates with $1/\rho ^{\prime
}+1/\rho ^{\prime }>1$, we see that, with changing constants $c$,%
\begin{eqnarray*}
\left\vert \mathbf{X}_{s,t}^{i,j}-\mathbf{Y}_{s,t}^{i,j}\right\vert
_{L^{2}}^{2} &\leq &c\left\vert R_{X-Y}\right\vert _{\rho ^{\prime }\text{$%
-var$;}\left[ s,t\right] ^{2}}\omega \left( \left[ s,t\right] ^{2}\right)
^{1/\rho ^{\prime }} \\
&\leq &c\left\vert R_{X-Y}\right\vert _{\infty }^{1-\rho /\rho ^{\prime
}}\omega \left( \left[ s,t\right] ^{2}\right) ^{2/\rho ^{\prime }}.
\end{eqnarray*}%
We now turn to level (iii) estimates and keep $i\neq j$ fixed throughout. We
have%
\begin{eqnarray*}
\left\vert \mathbf{X}_{s,t}^{i,i,j}-\mathbf{Y}_{s,t}^{i,i,j}\right\vert
_{L^{2}}^{2} &\leq &2\left\vert \int_{s}^{t}\left( X_{s,u}^{i}\right)
^{2}d\left( X_{u}^{j}-Y_{u}^{j}\right) \right\vert _{L^{2}}^{2} \\
&&+2\left\vert \int_{s}^{t}\left\{ \left( X_{s,u}^{i}\right) ^{2}-\left(
Y_{s,u}^{i}\right) ^{2}\right\} dY_{u}^{j}\right\vert _{L^{2}}^{2}.
\end{eqnarray*}%
The variance of $\int_{s}^{t}\left( X_{s,u}^{i}\right) ^{2}d\left(
X_{u}^{j}-Y_{u}^{j}\right) $ can be written as 2D Young integral and by
Proposition \ref{p-variationE(X2)} and 2D Young estimates we obtain the bound%
\begin{equation*}
\left\vert \int_{s}^{t}\left( X_{s,u}^{i}\right) ^{2}d\left(
X_{u}^{j}-Y_{u}^{j}\right) \right\vert _{L^{2}}^{2}\leq c\left\vert
R_{X-Y}\right\vert _{\infty }^{1-\rho /\rho ^{\prime }}\omega \left(
s,t\right) ^{3/\rho ^{\prime }}.
\end{equation*}%
To deal with the other term, we first note that, from Proposition \ref%
{p-variationE(X2)}, the $\rho $-variation of%
\begin{equation*}
\left( u,v\right) \mapsto g\left( u,v\right) \equiv \mathbb{E}\left[ \left\{
\left( X_{s,u}^{i}\right) ^{2}-\left( Y_{s,u}^{i}\right) ^{2}\right\}
\left\{ \left( X_{s,v}^{i}\right) ^{2}-\left( Y_{s,v}^{i}\right)
^{2}\right\} \right]
\end{equation*}%
over $\left[ s,t\right] ^{2}$ is controlled by a constant times $\omega
\left( \left[ s,t\right] ^{2}\right) ^{2}$ while its supremum norm on $\left[
s,t\right] ^{2}$ is bounded by a constant times%
\begin{equation*}
\left\vert R_{X-Y}\right\vert _{\infty }\omega \left( \left[ s,t\right]
^{2}\right) ^{1/\rho }.
\end{equation*}%
To see the latter, it suffices to write $g\left( u,v\right) $ as expectation
of the product of the four factors $X_{s,u}^{i}\pm
Y_{s,u}^{i},X_{s,v}^{i}\pm Y_{s,v}^{i}$, bounded by the product of the
respective $L^{4}$-norms which are (everything is Gaussian) equivalent to
the respective $L^{2}$-norms. This leads to%
\begin{eqnarray*}
&&\left\vert \int_{s}^{t}\left\{ \left( X_{s,u}^{i}\right) ^{2}-\left(
Y_{s,u}^{i}\right) ^{2}\right\} dY_{u}^{j}\right\vert _{L^{2}}^{2} \\
&=&\int_{\left[ s,t\right] ^{2}}g\left( u,v\right) dR_{Y^{j}}\left(
u,v\right) \\
&\leq &c\left\vert g\right\vert _{\rho ^{\prime }\text{$-var$;}\left[ s,t%
\right] }\left\vert R_{Y^{j}}\right\vert _{\rho ^{\prime }\text{$-var$;}%
\left[ s,t\right] } \\
&\leq &c\left\vert g\right\vert _{\infty }^{1-\rho /\rho ^{\prime
}}\left\vert g\right\vert _{\rho \text{$-var$;}\left[ s,t\right] }^{\rho
/\rho ^{\prime }}\omega \left( \left[ s,t\right] ^{2}\right) ^{1/\rho
^{\prime }} \\
&\leq &c\left( \left\vert R_{X-Y}\right\vert _{\infty }\omega \left( \left[
s,t\right] ^{2}\right) ^{1/\rho }\right) ^{1-\rho /\rho ^{\prime }}\left(
\omega \left( \left[ s,t\right] ^{2}\right) ^{2}\right) ^{1/\rho ^{\prime
}}\omega \left( \left[ s,t\right] ^{2}\right) ^{1/\rho ^{\prime }} \\
&=&c\left\vert R_{X-Y}\right\vert _{\infty }^{1-\rho /\rho ^{\prime }}\omega
\left( \left[ s,t\right] ^{2}\right) ^{1/\rho +2/\rho ^{\prime }}
\end{eqnarray*}%
and it follows that%
\begin{equation*}
\left\vert \mathbf{X}_{s,t}^{i,i,j}-\mathbf{Y}_{s,t}^{i,i,j}\right\vert
_{L^{2}}^{2}\leq c\left\vert R_{X-Y}\right\vert _{\infty }^{1-\rho /\rho
^{\prime }}\omega \left( \left[ s,t\right] ^{2}\right) ^{1/\rho +2/\rho
^{\prime }}.
\end{equation*}%
It remains to prove (iii.2) and we fix distinct indices $i,j,k$. To see that%
\begin{equation*}
\mathbb{E}\left( \left\vert \mathbf{X}_{s,t}^{i,j,k}-\mathbf{Y}%
_{s,t}^{i,j,k}\right\vert ^{2}\right) \leq c\left\vert R_{X-Y}\right\vert
_{\infty }^{\left( \rho ^{\prime }-\rho \right) /\rho ^{\prime }}\omega
\left( \left[ s,t\right] ^{2}\right) ^{3/\rho ^{\prime }}
\end{equation*}%
we proceed as in the proof of (ii) and start by subtract/adding%
\begin{equation*}
\int_{\lbrack s,t]}\mathbf{X}_{s,\cdot }^{i,j}dY^{k}.
\end{equation*}%
After using the triangle inequality we are left with two terms. The first is
the variance of the 2D Young integral $\int $ $\mathbf{X}_{s,\cdot
}^{i,j}d\left( X-Y\right) ^{k}$ which is handled via Proposition \ref%
{p-variationE(X2)} and 2D Young estimates, exactly as earlier. The second
term is of form $\int \left( \mathbf{X}_{s,u}^{i,j}-\mathbf{Y}%
_{s,u}^{i,j}\right) dY_{u}^{k}$ and is handled by the split-up,%
\begin{equation*}
\mathbf{X}_{s,u}^{i,j}-\mathbf{Y}_{s,u}^{i,j}=\int_{s}^{u}X_{s,\cdot
}^{i}d\left( X^{j}-Y^{j}\right) +\int_{s}^{u}\left( X_{s,\cdot
}^{i}-Y_{s,\cdot }^{i}\right) dY^{j}.
\end{equation*}%
We leave the remaining details to the reader.
\end{proof}

\begin{corollary}
\label{MainCorollaryContinuityPwLinearGauss}With $\mathbf{X},\mathbf{Y},\rho
,\rho ^{\prime },\omega ,K$ as in the last proposition there exists $C_{\ref%
{MainCorollaryContinuityPwLinearGauss}}=C_{\ref%
{MainCorollaryContinuityPwLinearGauss}}\left( \rho ,\rho ^{\prime },K\right) 
$ and $\theta =\theta \left( \rho ,\rho ^{\prime }\right) >0$ such that for
all $s<t$ in $\left[ 0,1\right] $ and $n=1,2,3,$%
\begin{equation*}
\mathbb{E}\left( \left\vert \pi _{n}\left( \ln \left( \mathbf{X}%
_{s,t}^{-1}\otimes \mathbf{Y}_{s,t}\right) \right) \right\vert ^{2}\right)
\leq C_{\ref{MainCorollaryContinuityPwLinearGauss}}\left\vert
R_{X-Y}\right\vert _{\infty }^{\theta }\omega \left( \left[ s,t\right]
^{2}\right) ^{n/\rho ^{\prime }}.
\end{equation*}
\end{corollary}

\begin{proof}
For $n=1$ this is a trivial consequence of (i) of the preceding proposition.
From Appendix III,%
\begin{equation*}
\left\vert \pi _{2}\left( \ln \left( \mathbf{X}_{s,t}^{-1}\otimes \mathbf{Y}%
_{s,t}\right) \right) \right\vert \leq \left\vert \pi _{2}\left( \ln \mathbf{%
Y}_{s,t}\right) -\pi _{2}\left( \ln \mathbf{X}_{s,t}\right) \right\vert +%
\frac{1}{2}\left\vert Y_{s,t}-X_{s,t}\right\vert .\left\vert
Y_{s,t}\right\vert
\end{equation*}%
which is readily handled by (i) and (ii) of the preceding proposition,
noting that thanks to Wiener-It\^{o} chaos integrability we can split up the 
$L^{2}$-norm of products, cf equation (\ref{L2splitup}),%
\begin{equation*}
\mathbb{E}\left[ \left\vert Y_{s,t}-X_{s,t}\right\vert ^{2}\left\vert
Y_{s,t}\right\vert ^{2}\right] \leq C\mathbb{E}\left( \left\vert
Y_{s,t}-X_{s,t}\right\vert ^{2}\right) \mathbb{E}\left( \left\vert
Y_{s,t}\right\vert ^{2}\right) .
\end{equation*}%
From Proposition \ref{BCHlevel3normbound} (appendix III)%
\begin{eqnarray*}
\left\vert \pi _{3}\left( \ln \left( \mathbf{X}_{s,t}^{-1}\otimes \mathbf{Y}%
_{s,t}\right) \right) \right\vert &\leq &\left\vert \pi _{3}\left( \ln 
\mathbf{Y}_{s,t}\right) -\pi _{3}\left( \ln \mathbf{X}_{s,t}\right)
\right\vert \\
&&+\frac{1}{2}\left\vert \pi _{2}\left( \ln \mathbf{Y}_{s,t}\right) -\pi
_{2}\left( \ln \mathbf{X}_{s,t}\right) \right\vert \left\vert
Y_{s,t}\right\vert \\
&&+\frac{1}{12}\left\vert Y_{s,t}-X_{s,t}\right\vert \left( \left\vert
X_{s,t}\right\vert ^{2}+\left\vert Y_{s,t}\right\vert ^{2}+6\left\vert \pi
_{2}\left( \ln \mathbf{X}_{s,t}\right) \right\vert \right) .
\end{eqnarray*}%
The terms which appear in the last two lines are handled by split up of $%
L^{2}$-norm as above, the term%
\begin{equation*}
\left\vert \pi _{3}\left( \ln \mathbf{Y}_{s,t}\right) -\pi _{3}\left( \ln 
\mathbf{X}_{s,t}\right) \right\vert
\end{equation*}%
expands with respect to the basis elements $\left[ e_{i},\left[ e_{j},e_{k}%
\right] \right] $ with coefficients of only four possible types, 
\begin{equation*}
\mathbf{Y}_{s,t}^{i,j,k}-\mathbf{X}_{s,t}^{i,j,k},\,\mathbf{Y}_{s,t}^{i,i,j}-%
\mathbf{X}_{s,t}^{i,i,j},\,\left\vert Y_{s,t}^{i}\right\vert
^{2}Y_{s,t}^{j}-\left\vert X_{s,t}^{i}\right\vert
^{2}X_{s,t}^{j},\,Y_{s,t}^{i}\mathbf{Y}_{s,t}^{i,j}-X_{s,t}^{i}\mathbf{X}%
_{s,t}^{i,j}\text{ \ }
\end{equation*}%
with $i,j,k$ distinct. The first two difference terms are handled precisely
with (iii.1) and (iii.2), the remaining terms are estimated by the split up
of $L^{2}$-norms combined with the elementary estimates of type%
\begin{equation*}
\left\vert bb^{\prime }-aa^{\prime }\right\vert \leq \left\vert b(b^{\prime
}-a^{\prime })+\left( b-a\right) a^{\prime }\right\vert \leq \left\vert
b\right\vert \left\vert b^{\prime }-a^{\prime }\right\vert +\left\vert
a^{\prime }\right\vert \left\vert b-a\right\vert
\end{equation*}%
and the estimates (i), (ii).
\end{proof}

The above estimates and Proposition \ref{PropWienerItoAndGRR} lead to the
following important corollary. (Note that $\omega $ can be taken as $%
\left\vert R_{\left( X,Y\right) }\right\vert _{\rho \text{$-var$;}\left[
\cdot ,\cdot \right] ,\left[ \cdot ,\cdot \right] }^{\rho }.$)

\begin{corollary}
\label{contOnDenseSpace}Under the above hypothesis, $\left\vert
R_{X-Y}\right\vert _{\infty }\leq 1,$ $p>2\rho $ and $\omega \left( \left[
0,1\right] ^{2}\right) $ bounded by $K$, there exists positive constants $%
\theta =\theta \left( p,\rho \right) >0$ and $C_{\ref{contOnDenseSpace}}=C_{%
\ref{contOnDenseSpace}}\left( p,\rho ,K\right) $ such that 
\begin{equation*}
\left\vert d_{p\text{$-var$}}\left( S_{3}\left( X\right) ,S_{3}\left(
Y\right) \right) \right\vert _{L^{q}}\leq C_{\ref{contOnDenseSpace}%
}\left\vert R_{X-Y}\right\vert _{\infty }^{\theta }\sqrt{q}.
\end{equation*}%
If $\omega \left( s,t\right) \leq K\left\vert t-s\right\vert ,$ then $d_{p%
\text{$-var$}}$ above may be replaced by $d_{1/p\text{$-H\ddot{o}l$}}.$
\end{corollary}

\subsection{Natural Lift of a Gaussian Process\label{GaussianRoughPaths}}

We are now able to prove the main theorems of this chapter.

\begin{theorem}[Construction of Lifted Gaussian Processes]
\label{ExistenceLiftGaussianProcess}Assume $X=\left( X^{1},\ldots
,X^{d}\right) $ is a centered continuous Gaussian process with independent
components. Let $\rho \in \lbrack 1,2)$ and assume the covariance of $X$ if
of finite $\rho $-variation dominated by a 2D control $\omega $.\newline
(i): (Existence) There exists a continuous $G^{3}\left( \mathbb{R}%
^{d}\right) $-valued process $\mathbf{X}$, such that a.e. realization is in $%
C_{0}^{0,p-var}\left( \left[ 0,1\right] ,G^{3}\left( \mathbb{R}^{d}\right)
\right) $ for $p>2\rho ,$ and hence a geometric $p$-rough path for $p\in
\left( 2\rho ,4\right) $, and which lifts the Gaussian process $X$ in the
sense $\pi _{1}\left( \mathbf{X}_{t}\right) =X_{t}-X_{0}$. If $\omega $ is H%
\"{o}lder dominated a.e. realization is in $C_{0}^{0,1/p-H\ddot{o}%
lder}\left( \left[ 0,1\right] ,G^{3}\left( \mathbb{R}^{d}\right) \right) $.
Finally, there exists $C_{\ref{ExistenceLiftGaussianProcess}}=C_{\ref%
{ExistenceLiftGaussianProcess}}\left( \rho \right) $ such that for all $s<t$
in $\left[ 0,1\right] $ and $q\in \lbrack 1,\infty ),$%
\begin{equation}
\left\vert d\left( \mathbf{X}_{s},\mathbf{X}_{t}\right) \right\vert
_{L^{q}}\leq C_{\ref{ExistenceLiftGaussianProcess}}\sqrt{q}\omega \left( %
\left[ s,t\right] ^{2}\right) ^{\frac{1}{2\rho }};
\label{LqEstimateThmLiftedGaussProcess}
\end{equation}%
and the random variables $\pi _{n}\left( \mathbf{X}_{s,t}\right) ,\pi
_{n}\left( \ln \mathbf{X}_{s,t}\right) ,\,n=1,2,3,$ are in the $n^{th}$\
(not necessarily homogenous) Wiener-It\^{o} chaos.\newline
(ii): (Fernique-estimates) Let $p>2\rho $ and $\omega \left( \left[ 0,1%
\right] ^{2}\right) \leq K$. Then there exists $\eta =\eta \left( p,\rho
,K\right) >0$, such that%
\begin{equation*}
\mathbb{E}\left( \exp \left( \eta \left\Vert \mathbf{X}\right\Vert _{p\text{$%
-var$,}\left[ 0,1\right] }^{2}\right) \right) <\infty .
\end{equation*}%
If $\omega \left( \left[ s,t\right] ^{2}\right) \leq K\left\vert
t-s\right\vert $ for all $s<t$ in $\left[ 0,1\right] $, then we can replace $%
\left\Vert \mathbf{X}\right\Vert _{p\text{$-var$,}\left[ 0,1\right] }$ by $%
\left\Vert \mathbf{X}\right\Vert _{1/p-H\ddot{o}l\text{;}\left[ 0,1\right] }$
above.\newline
(iii): (Uniqueness) The lift $\mathbf{X}$ is unique in the sense\footnote{%
We shall see in a later section that $\mathbf{X}$ is also the limit of
(lifted)\ Karhunen-Loeve type approximations of the Gaussian process $X$.}
that it is the $d_{p\text{$-var$}}$-limit in $L^{q}\left( \mathbb{P}\right) $%
, for any $q\in \lbrack 1,\infty )$, of any sequence $S_{3}\left(
X^{D}\right) $ with $\left\vert D\right\vert \rightarrow 0$. (As usual, $%
X^{D}$ denotes the piecewise linear approximation of $X$ based on a
dissection $D$ of $\left[ 0,1\right] \,$.)\newline
(iv): (Consistency) If $X$ has a.s. sample paths of finite $[1,2)$%
-variation, $\mathbf{X}$ coincides with the canonical lift obtained by
iterated Young-integration of $X$. If $\mathbf{\tilde{X}}=\left( 1,\pi
_{1}\left( \mathbf{X}\right) ,\pi _{2}\left( \mathbf{X}\right) \right) \in
C_{0}^{0,p-var}\left( \left[ 0,1\right] ,G^{2}\left( \mathbb{R}^{d}\right)
\right) $ a.s. for $p<3$ then $\mathbf{\tilde{X}}$ is a geometric $p$-rough
path and and $\mathbf{X}$ coincides with the Young-Lyons lift of $\mathbf{%
\tilde{X}}$.
\end{theorem}

\begin{definition}
We call $\mathbf{X}$ natural lift (of the Gaussian process) $X$. A typical
realizations of $\mathbf{X}$ is called a Gaussian rough path.
\end{definition}

\begin{proof}
(Existence, Uniqueness)\ Let $\left( D_{n}\right) $ be a sequence of
dissections with mesh $\left\vert D_{n}\right\vert \rightarrow 0$. Clearly, $%
\left\vert R_{X^{D_{n}}-X^{D_{m}}}\right\vert _{\infty }\rightarrow 0$ and
from Corollary \ref{contOnDenseSpace} for every $p>2\rho $,%
\begin{equation*}
\left\vert d_{p\text{$-var$}}\left( S_{3}\left( X^{D_{n}}\right)
,S_{3}\left( X^{D_{m}}\right) \right) \right\vert _{L^{q}}\rightarrow 0\text{%
.}
\end{equation*}%
In particular, we see that\ $\left( S_{3}\left( X^{D_{n}}\right) \right) $
is Cauchy in probability as sequence of $C_{0}^{0,p\text{$-var$}}$-valued
random variables\footnote{%
A\ Cauchy criterion for convergence in probability of r.v.s with values in a
Polish space is an immediate generalization of the corresponding real-valued
case.} and so there exists $\mathbf{X}\in C_{0}^{0,p\text{$-var$}}\left( %
\left[ 0,1\right] ,G^{3}\left( \mathbb{R}^{d}\right) \right) $ so that $d_{p%
\text{$-var$}}\left( S_{3}\left( X^{D_{n}}\right) ,\mathbf{X}\right)
\rightarrow 0$ in probability and from the uniform estimates from Corollary %
\ref{FerniquePwlinear} also in $L^{q}$ for all $q\in \lbrack 1,\infty )$. \
If $\left( \tilde{D}_{n}\right) $ is another sequence of dissections with
mesh tending to zero, the same construction yields a limit, say $\mathbf{%
\tilde{X}}$. But%
\begin{eqnarray*}
d_{p\text{$-var$}}\left( \mathbf{X},\mathbf{\tilde{X}}\right)  &\leq &d_{p%
\text{$-var$}}\left( \mathbf{X},S_{3}\left( X^{D_{n}}\right) \right) +d_{p%
\text{$-var$}}\left( S_{3}\left( X^{D_{n}}\right) ,S_{3}\left( X^{\tilde{D}%
_{n}}\right) \right)  \\
&&+d_{p\text{$-var$}}\left( S_{3}\left( X^{\tilde{D}_{n}}\right) ,\mathbf{%
\tilde{X}}\right) 
\end{eqnarray*}%
and the right hand side converges to zero (in probability, say) as $%
n\rightarrow \infty $ which shows $\mathbf{X}=\mathbf{\tilde{X}}$ a.s. We
now show the estimate (\ref{LqEstimateThmLiftedGaussProcess}). To this end,
let $\omega ^{n}$ denote the 2D control given by $\left\vert
R_{X^{D_{n}}}\right\vert _{\rho ^{\prime }\text{$-var$;}\left[ \left[ \cdot
,\cdot \right] ,\left[ \cdot ,\cdot \right] \right] }^{\rho ^{\prime }}$ for 
$\rho ^{\prime }\in \left( \rho ,2\right) $. From Corollary \ref%
{FerniquePwlinear}%
\begin{equation*}
\left\vert d\left( S_{3}\left( X^{D_{n}}\right) _{s},S_{3}\left(
X^{D_{n}}\right) _{t}\right) \right\vert _{L^{q}}\leq C\sqrt{q}\omega
^{n}\left( \left[ s,t\right] ^{2}\right) ^{\frac{1}{2\rho ^{\prime }}}
\end{equation*}%
and after sending $n\rightarrow \infty $, followed by $\rho ^{\prime
}\downarrow \rho $ using Lemma \ref{LemmaRhoPrimeToRho}, we find%
\begin{equation*}
\left\vert d\left( \mathbf{X}_{s},\mathbf{X}_{t}\right) \right\vert
_{L^{q}}\leq C\sqrt{q}\omega \left( \left[ s,t\right] ^{2}\right) ^{\frac{1}{%
2\rho }}
\end{equation*}%
and (\ref{LqEstimateThmLiftedGaussProcess}) is proved. The statements on $%
\pi _{n}\left( \mathbf{X}_{s,t}\right) ,\pi _{n}\left( \ln \mathbf{X}%
_{s,t}\right) \in $ $n^{th\text{ }}$Wiener-It\^{o} chaos are are immediate
from Proposition \ref{PropGroupIncrInWIC} and closeness of the $n^{th}$
Wiener-It\^{o} chaos under convergence in $L^{q}$. We then see that one can
switch to equivalent estimates in terms of $\pi _{n}\left( \mathbf{X}%
_{s,t}\right) ,\pi _{n}\left( \ln \mathbf{X}_{s,t}\right) $ thanks to Recall
that Corollary \ref{ChaosAndGroup}), in particular for $n=1,2,3$ and all $s<t
$ in $\left[ 0,1\right] $,%
\begin{equation*}
\left\vert \pi _{n}\left( \ln \mathbf{X}_{s,t}\right) \right\vert
_{L^{2}}\leq c\omega \left( \left[ s,t\right] ^{2}\right) ^{\frac{n}{2\rho }%
}.
\end{equation*}%
(Regularity, Fernique)\ An immediate consequence of Proposition \ref%
{PropWienerItoAndGRR} applied with 1D control $\left( s,t\right) \mapsto
\omega \left( \left[ s,t\right] ^{2}\right) $.\newline
(Consistency) An immediate consequence of our construction and basic
continuity statements of the Young- resp. Young-Lyons lift, \cite{lyons-98,
lyons-qian-02}.
\end{proof}

\begin{theorem}
\label{ContLiftGaussianProcess}Let $\left( X,Y\right) =\left(
X^{1},Y^{1},\ldots ,X^{d},Y^{d}\right) $ be a centered continuous Gaussian
process such that $\left( X^{i},Y^{i}\right) $ is independent of $\left(
X^{j},Y^{j}\right) $ when $i\neq j$. Let $\rho \in \lbrack 1,2)$ and assume
the covariance of $\left( X,Y\right) $ is of finite $\rho $-variation
dominated by a 2D control $\omega $. Then, for every $p>2\rho $, there exist
positive constants $\theta =\theta \left( p,\rho \right) $ and $C_{\ref%
{ExistenceLiftGaussianProcess}}=C\left( p,\rho ,K\right) ,$ with$\ \omega
\left( \left[ 0,1\right] ^{2}\right) \leq K,$ such that for all $q\in
\lbrack 1,\infty ),$%
\begin{equation*}
\left\vert d_{p\text{$-var$}}\left( \mathbf{X},\mathbf{Y}\right) \right\vert
_{L^{q}}\leq C_{\ref{ExistenceLiftGaussianProcess}}\sqrt{q}\left\vert
R_{X-Y}\right\vert _{\infty }^{\theta }.
\end{equation*}%
If $\omega \left( \left[ s,t\right] ^{2}\right) \leq K\left\vert
t-s\right\vert $ for all $s<t$ in~$\left[ 0,1\right] $ we can replace $d_{p%
\text{$-var$}}$ by $d_{1/p-H\ddot{o}l}$ in the preceding line.
\end{theorem}

\begin{proof}
Pick $\rho ^{\prime }$ such that $p>2\rho ^{\prime }>2\rho $ and, similarly
to the last proof, pass to the limit in Proposition \ref%
{MainCorollaryContinuityPwLinearGauss}. Conclude with Proposition \ref%
{MainCorollaryContinuityPwLinearGauss}.
\end{proof}

Related ideas lead to the following proposition which will be a useful tool
in the section on Karhunen-Loeve approximations. (It would also allow to
discuss lifts of non-centred Gaussian processes for which $t\mapsto \mathbb{E%
}\left( X_{t}\right) \in C^{\rho \text{$-var$}}\left( \left[ 0,1\right] ,%
\mathbb{R}^{d}\right) $, these cannot be handled by the rough path
translation operator. We shall not pursue this here.)

\begin{proposition}[Young-Wiener integral]
\label{fDX}Assume $X$ has covariance $R$ with finite $\rho $-variation. Let $%
f\in C^{q\text{$-var$}}\left( \left[ 0,1\right] ,\mathbb{R}\right) ,$ with $%
q^{-1}+\rho ^{-1}>1$. If $X^{n}$ is a sequence of Gaussian processes whose
covariances are uniformly of finite $p$-variation and such that $\left\vert
R_{X^{n}-X}\right\vert _{\infty }$ converges to $0,$ then in the supremum
topology, $t\rightarrow \int_{0}^{t}f_{u}dX_{u}^{n}$ converges in $L^{2}$.
We define this limit to be the integral%
\begin{equation*}
t\mapsto \int_{0}^{t}f_{u}dX_{u}.
\end{equation*}%
For all $s<t$ in $\left[ 0,1\right] ,$ we have the Young-Wiener isometry,%
\begin{equation}
\mathbb{E}\left( \left\vert \int_{s}^{t}f_{u}dX_{u}\right\vert ^{2}\right)
=\int_{\left[ s,t\right] ^{2}}f_{u}f_{v}dR\left( u,v\right) ,  \notag
\end{equation}%
and if $f\left( s\right) =0$ we have the Young-Wiener estimate%
\begin{equation}
\mathbb{E}\left( \left\vert \int_{s}^{t}f_{u}dX_{u}\right\vert ^{2}\right)
\leq C_{\rho ,q}\left\vert f\right\vert _{q\text{$-var$;}\left[ s,t\right]
}^{2}\left\vert R\right\vert _{\rho \text{$-var$;}\left[ s,t\right] }.
\label{fXL2estimate}
\end{equation}
\end{proposition}

\begin{proof}
Proving (\ref{fXL2estimate}) for $X$ piecewise linear and applying the same
methodology developed in this chapter is enough. But for $X$ piecewise
linear, it is obvious that%
\begin{equation*}
\mathbb{E}\left( \left\vert \int_{s}^{t}f_{u}dX_{u}\right\vert ^{2}\right)
=\int_{\left[ s,t\right] ^{2}}f_{u}f_{v}dR\left( u,v\right) .
\end{equation*}%
Now, the $q$-variation of $\left( u,v\right) \rightarrow f_{u}f_{v}$ is of
course bounded by $\left\vert f\right\vert _{q-var}^{2},$ so applying Young
2D estimates, we are done.
\end{proof}

\begin{remark}
When $X$ is Brownian motion, $dR=\delta _{\left\{ s=t\right\} }$ and we
recover the usual It\^{o} isometry.
\end{remark}

\subsection{Almost Sure Convergence}

\begin{proposition}
\label{clearPieceWiseLinearConvLq}Let $X=\left( X_{1},\ldots ,X_{d}\right) $
be a centered continuous Gaussian process with independent components, and
assume that the covariance of $X$ is of finite $\rho $-variation dominated
by a 2D control $\omega ,$ for some $\rho <2.$ Then, if $D=\left(
t_{i}\right) $ is a subdivision of $\left[ 0,1\right] ,$ and $X^{D}$ be the
piecewise linear approximation of $X$. Then, if $p>2\rho ,$ there exist
positive constants $\theta =\theta \left( \rho ,p\right) $ and $C=C\left(
\rho ,p,K\right) $, with $\omega \left( \left[ 0,1\right] ^{2}\right) \leq K$%
, such that for all $q\in \lbrack 1,\infty )$,%
\begin{equation*}
\left\vert d_{p\text{$-var$}}\left( \mathbf{X},S_{3}\left( X^{D}\right)
\right) \right\vert _{L^{q}}\leq C\sqrt{q}\max_{i}\omega \left( \left[
t_{i},t_{i+1}\right] ^{2}\right) ^{\theta }.
\end{equation*}%
If $\omega \left( \left[ s,t\right] ^{2}\right) \leq K\left\vert
t-s\right\vert $ for all $s<t$ in $\left[ 0,1\right] $ we have 
\begin{equation*}
\left\vert d_{1/p-H\ddot{o}l}\left( \mathbf{X},S_{3}\left( X^{D}\right)
\right) \right\vert _{L^{q}}\leq C\sqrt{q}\max_{i}\left\vert
t_{i+1}-t_{i}\right\vert ^{\theta }.
\end{equation*}
\end{proposition}

\begin{proof}
A simple corollary of Theorems \ref{ExistenceLiftGaussianProcess}, \ref%
{ContLiftGaussianProcess} combined with $\left\vert R_{X-X^{D}}\right\vert
_{\infty }\leq \max_{i}\omega \left( \left[ t_{i},t_{i+1}\right] ^{2}\right)
^{1/\rho }.$
\end{proof}

As a corollary, we obtain a.s. convergence of dyadic approximations in a H%
\"{o}lder situation. In view of Lemma \ref{QCboundedbyfBM} we have arrived
at a substantial generalization of the results in \cite{coutin-qian-02}.

\begin{corollary}
\label{almostSureBorelCantelli}Let $X,\omega ,\rho <2,p>2\rho $ as above and
assume $\omega \left( \left[ s,t\right] ^{2}\right) \leq K\left\vert
t-s\right\vert $ for all $s<t$ in $\left[ 0,1\right] $. If $X^{D_{n}}$
denote the dyadic piecewise linear approximation of $X$ based on $%
D_{n}=\left\{ \frac{k}{2^{n}},0\leq k\leq 2^{n}\right\} $ then there exist
positive constants $\theta =\theta \left( \rho ,p\right) $ and $C=C\left(
\rho ,p,K\right) $ so that for all $q\in \lbrack 1,\infty )$%
\begin{equation*}
\left\vert d_{1/p-H\ddot{o}l}\left( \mathbf{X},S_{3}\left( X^{D_{n}}\right)
\right) \right\vert _{L^{q}}\leq C\sqrt{q}2^{-\frac{n\theta }{\rho }}
\end{equation*}%
and as $n$ tends to infinity, $d_{1/p-H\ddot{o}l}\left( \mathbf{X}%
,S_{3}\left( X^{D_{n}}\right) \right) \rightarrow 0$ a.s. and in $L^{q}$.
\end{corollary}

\begin{proof}
Only the a.s. convergence statement remains to be seen. But this is a
standard Borell-Cantelli argument.
\end{proof}

\section{Weak Approximations}

\subsection{Tightness}

\begin{proposition}
\label{tightnessPvarHold}Let $\left( X_{n}\right) $ be a sequence of
centered, $d$-dimensional, continuous Gaussian process with independent
components, and assume that the covariances of $X_{n}$ with finite $\rho \in
\lbrack 1,2)$-variation dominated by a 2D control $\omega $, uniformly in $n$%
. Let $p>2\rho $ and let $\mathbf{X}_{n}$ denote the natural lift of $X_{n}$
with a.e. sample path in $C_{0}^{0,p\text{$-var$}}\left( \left[ 0,1\right]
,G^{3}\left( \mathbb{R}^{d}\right) \right) $. Then the family $\left( 
\mathbb{P}_{\ast }\mathbf{X}_{n}\right) $, i.e. the laws of $\mathbf{X}_{n}$
viewed as Borel measures on the Polish space $C_{0}^{0,p\text{$-var$}}\left( %
\left[ 0,1\right] ,G^{3}\left( \mathbb{R}^{d}\right) \right) $, are tight$.$
If $\omega $ is H\"{o}lder dominated, then tightness holds in $C_{0}^{0,1/p-H%
\ddot{o}l}\left( \left[ 0,1\right] ,G^{3}\left( \mathbb{R}^{d}\right)
\right) $.
\end{proposition}

\begin{proof}
Let us fix $p^{\prime }\in \left( 2\rho ,p\right) $. Define $K_{R}$ to be
the relatively compact set in $C_{0}^{0,p\text{$-var$}}\left( \left[ 0,1%
\right] ,G^{3}\left( \mathbb{R}^{d}\right) \right) $, 
\begin{equation*}
\left\{ \mathbf{x:}\text{ for all }s<t\text{ in }\left[ 0,1\right]
:\left\Vert \mathbf{x}_{s,t}\right\Vert ^{p^{\prime }}\leq R\left\vert
\omega \left( \left[ 0,t\right] ^{2}\right) -\omega \left( \left[ 0,s\right]
^{2}\right) \right\vert \right\} .
\end{equation*}%
From the results of appendix II, there exists real random variables $M_{n}$
such that (i) for some $\mu $ small enough, $\sup_{n}E\left( \exp \left( \mu
M_{n}^{2}\right) \right) <\infty ,$ (ii) for all $n\geq 1,$ for all $s,t\in %
\left[ 0,1\right] $,%
\begin{equation*}
\left\Vert \mathbf{X}_{n}\left( s,t\right) \right\Vert ^{p^{\prime }}\leq
M_{n}\left\vert \omega \left( \left[ 0,t\right] ^{2}\right) -\omega \left( %
\left[ 0,s\right] ^{2}\right) \right\vert .
\end{equation*}%
Hence, there exists $c=c\left( \mu \right) >0$ such that $\sup_{n}P\left( 
\mathbf{X}_{n}\in K_{R}\right) \leq \exp \left( -cR^{2}\right) $ which shows
tightness in $C_{0}^{0,p\text{$-var$}}\left( \left[ 0,1\right] ,G^{3}\left( 
\mathbb{R}^{d}\right) \right) $. Similarly, for H\"{o}lder dominated $\omega 
$ we obtain tightness in $C_{0}^{0,1/p\text{-H\"{o}l}}\left( \left[ 0,1%
\right] ,G^{3}\left( \mathbb{R}^{d}\right) \right) $ from the relative
compactness of%
\begin{equation*}
\left\{ \mathbf{x:}\text{ for all }s<t\text{ in }\left[ 0,1\right]
:\left\Vert \mathbf{x}_{s,t}\right\Vert ^{p^{\prime }}\leq R\left\vert
t-s\right\vert \right\} \text{.}
\end{equation*}
\end{proof}

\subsection{Convergence}

\begin{theorem}
Let $\rho \in \lbrack 1,2)$. Let $X_{n},X_{\infty }$ be continuous Gaussian
process with covariance $R^{n},R^{\infty }$ of finite $\rho \in \lbrack 1,2)$%
-variation dominated uniformly in $n$ by a 2D control $\omega $, such that%
\begin{equation*}
R^{n}\rightarrow R^{\infty }\text{ \ \ \ pointwise on }\left[ 0,1\right] ^{2}%
\text{.}
\end{equation*}%
Let $\mathbf{X}_{n},\mathbf{X}_{\infty }$ denote the associated natural $%
G^{3}\left( \mathbb{R}^{d}\right) $-valued lifted processes. Then, for any $%
p>2\rho $, the processes $\mathbf{X}_{n}$ converge in distribution to $%
\mathbf{X}_{\infty }$ with respect to $p$-variation topology. If $\omega $
is H\"{o}lder dominated, then convergence holds with respect to $1/p$-H\"{o}%
lder topology.
\end{theorem}

\begin{proof}
By\ Prohorov's theorem \cite{billingsley-1968}, tightness already implies
existence of weak limits as measures on%
\begin{equation*}
C^{0,p\text{$-var$}}\left( \left[ 0,1\right] ,G^{3}\left( \mathbb{R}%
^{d}\right) \right) \text{\thinspace\ resp. }C^{0,1/p-H\ddot{o}l}\left( %
\left[ 0,1\right] ,G^{3}\left( \mathbb{R}^{d}\right) \right)
\end{equation*}%
and it will suffice to establish weak convergence on the space $E:=C\left( %
\left[ 0,1\right] ,G^{3}\left( \mathbb{R}^{d}\right) \right) $ with $%
d_{\infty }$-metric. By the Portmanteau theorem \cite{billingsley-1968}, it
suffices to show that for every $f:E\rightarrow \mathbb{R},\,\,$bounded and
uniformly continuous, 
\begin{equation}
\mathbb{E}f\left( \mathbf{X}_{n}\right) \rightarrow \mathbb{E}f\left( 
\mathbf{X}_{\infty }\right) .  \label{WeakConvAlongUnifContTestFcts}
\end{equation}%
To see this, fix $\varepsilon >0,$ and $\delta =\delta \left( \varepsilon
\right) >0$ such that $d_{\infty }\left( \mathbf{x},\mathbf{y}\right)
<\delta $ implies $\left\vert f\left( \mathbf{x}\right) -f\left( \mathbf{y}%
\right) \right\vert <\varepsilon .$ The estimates of Proposition \ref%
{clearPieceWiseLinearConvLq}) are more than enough to see that there exists
a dissection $D$, with small enough mesh, such that%
\begin{equation*}
\sup_{0\leq n\leq \infty }\mathbb{P}\left( d_{\infty }\left( \mathbf{X}%
_{n},S_{3}\left( X_{n}^{D}\right) \right) \geq \delta \right) <\varepsilon .
\end{equation*}%
Hence, 
\begin{multline*}
\sup_{0\leq n\leq \infty }\left\vert \mathbb{E}f\left( \mathbf{X}_{n}\right)
-\mathbb{E}f\left( S_{3}\left( X_{n}^{D}\right) \right) \right\vert \\
\left. 
\begin{array}{l}
\leq \sup_{0\leq n\leq \infty }\left\vert \mathbb{E}\left[ \left\vert
f\left( \mathbf{X}_{n}\right) -f\left( S_{3}\left( X_{n}^{D}\right) \right)
\right\vert ;\,d_{\infty }\left( \mathbf{X}_{n},S_{3}\left( X_{n}^{D}\right)
\right) \geq \delta \right] \right\vert \\ 
\text{ \ }+\sup_{0\leq n\leq \infty }\left\vert \mathbb{E}\left[ \left\vert
f\left( \mathbf{X}_{n}\right) -f\left( S_{3}\left( X_{n}^{D}\right) \right)
\right\vert ;\,d_{\infty }\left( \mathbf{X}_{n},S_{3}\left( X_{n}^{D}\right)
\right) <\delta \right] \right\vert \\ 
\leq 2\left\vert f\right\vert _{\infty }\sup_{0\leq n\leq \infty }\mathbb{P}%
\left( d_{\infty }\left( \mathbf{X}_{n},S_{3}\left( X_{n}^{D}\right) \right)
\geq \delta \right) +\varepsilon \\ 
\leq \left( 2\left\vert f\right\vert _{\infty }+1\right) \varepsilon .%
\end{array}%
\right.
\end{multline*}%
On the other hand, $R^{n}\rightarrow R$ pointwise gives convergence of the
finite-dimensional distributions and hence weak convergence of $\left(
X_{n}^{D}\left( t\right) \right) _{t\in D}$ to $\left( X_{\infty }^{D}\left(
t\right) \right) _{t\in D}$. The map $\left( X_{n}^{D}\left( t\right)
\right) _{t\in D}\mapsto f\left( S_{3}\left( X_{n}^{D}\right) \right) $ is
easily seen to be continuous and so, for $n\geq n_{0}\left( \varepsilon
\right) $ large enough, 
\begin{equation*}
\left\vert \mathbb{E}f\left( S_{3}\left( X_{n}^{D}\right) \right) -\mathbb{E}%
f\left( S_{3}\left( X_{\infty }^{D}\right) \right) \right\vert \leq
\varepsilon .
\end{equation*}%
The proof is then finished with the triangle inequality,%
\begin{eqnarray*}
\left\vert \mathbb{E}f\left( \mathbf{X}_{n}\right) -\mathbb{E}f\left( 
\mathbf{X}_{\infty }\right) \right\vert &\leq &\left\vert \mathbb{E}f\left( 
\mathbf{X}_{n}\right) -\mathbb{E}f\left( S_{3}\left( X_{n}^{D}\right)
\right) \right\vert \\
&&+\left\vert \mathbb{E}f\left( S_{3}\left( X_{\infty }^{D}\right) \right) -%
\mathbb{E}f\left( \mathbf{X}_{\infty }\right) \right\vert \\
&&+\left\vert \mathbb{E}f\left( S_{3}\left( X_{n}^{D}\right) \right) -%
\mathbb{E}f\left( S_{3}\left( X^{D}\right) \right) \right\vert \\
&\leq &\left( 2\left\vert f\right\vert _{\infty }+1\right) 2\varepsilon
+\varepsilon .
\end{eqnarray*}
\end{proof}

\begin{example}
Set $R\left( s,t\right) =\min \left( s,t\right) $. The covariance of
fractional Brownian Motion is given by%
\begin{equation*}
R^{H}\left( s,t\right) =\frac{1}{2}\left( s^{2H}+t^{2H}-\left\vert
t-s\right\vert ^{2H}\right) .
\end{equation*}%
Take a sequence $H_{n}\uparrow 1/2$. It is easy to see that $%
R^{H_{n}}\rightarrow R$ pointwise and from our discussion of fractional
Brownian Motion, for any $\rho >1$,%
\begin{equation*}
\lim \sup_{n\rightarrow \infty }\left\vert R^{H_{n}}\right\vert _{\rho \text{%
$-var$}}<\infty .
\end{equation*}%
It follows that RDE\ solutions driven by (multidimensional) fractional
Brownian Motion with Hurst parameter $H_{n}$ tend weakly to the usual
Stratonovich solution. More elementary, for $H_{n}\downarrow 1/2$ we see
that Young ODE solutions driven by $B^{H_{n}}$ tend weakly to a Stratonovich
solution.
\end{example}

\section{Appendix I: $L^{q}$-convergence for Rough Paths}

The following lemma is an elementary consequence of the definition of $%
\otimes $ and equivalence of homogenous norms.

\begin{lemma}
\label{ghg}Let $g,h\in G^{N}\left( \mathbb{R}^{d}\right) $. Then\bigskip\
there exists $C=C\left( N,d\right) $ such that%
\begin{equation*}
\left\Vert g^{-1}\otimes h\otimes g\right\Vert \leq C\max \left\{ \left\Vert
h\right\Vert ,\left\Vert h\right\Vert ^{1/N}\left\Vert g\right\Vert
^{1-1/N}\right\} 
\end{equation*}
\end{lemma}

Recall the notions of $d_{0}$ and $d_{\infty }$ as defined in section \ref%
{secNotations}.

\begin{proposition}[$d_{0}/d_{\infty }$ estimate]
On the path-space $C_{0}\left( \left[ 0,1\right] ,G^{N}\left( \mathbb{R}%
^{d}\right) \right) $ the distances $d_{\infty }$ and $d_{0}\equiv d_{0-H%
\ddot{o}l}$ are locally $1/N$-H\"{o}lder equivalent. More precisely, there
exists $C=C\left( N,d\right) $ such that%
\begin{equation*}
d_{\infty }\left( \mathbf{x},\mathbf{y}\right) \leq d_{0}\left( \mathbf{x},%
\mathbf{y}\right) \leq C\max \left\{ d_{\infty }\left( \mathbf{x},\mathbf{y}%
\right) ,d_{\infty }\left( \mathbf{x},\mathbf{y}\right) ^{1/N}\left(
\left\Vert \mathbf{x}\right\Vert _{\infty }+\left\Vert \mathbf{y}\right\Vert
_{\infty }\right) ^{1-1/N}\right\} 
\end{equation*}
\end{proposition}

\begin{proof}
Only the second inequality requires a proof. We write $gh$ instead of $%
g\otimes h$. For any $s<t$ in $\left[ 0,1\right] $,%
\begin{equation*}
\mathbf{x}_{st}^{-1}\mathbf{y}_{s,t}=\mathbf{x}_{st}^{-1}\mathbf{y}_{s}^{-1}%
\mathbf{x}_{s}\mathbf{x}_{st}\mathbf{x}_{t}^{-1}\mathbf{y}_{t}\mathbf{x}%
_{t}^{-1}\mathbf{y}_{t}\mathbf{y}_{t}^{-1}\mathbf{x}_{t}.
\end{equation*}%
By sub-additivity,%
\begin{eqnarray*}
\left\Vert \mathbf{x}_{st}^{-1}\mathbf{y}_{s,t}\right\Vert  &\leq
&\left\Vert \mathbf{x}_{st}^{-1}\mathbf{y}_{s}^{-1}\mathbf{x}_{s}\mathbf{x}%
_{st}\right\Vert +\left\Vert \mathbf{x}_{t}^{-1}\mathbf{y}_{t}\mathbf{x}%
_{t}^{-1}\mathbf{y}_{t}\mathbf{y}_{t}^{-1}\mathbf{x}_{t}\right\Vert  \\
&=&\left\Vert v^{-1}\mathbf{y}_{s}^{-1}\mathbf{x}_{s}v\right\Vert
+\left\Vert w^{-1}\mathbf{x}_{t}^{-1}\mathbf{y}_{t}w\right\Vert 
\end{eqnarray*}%
with $v=\mathbf{x}_{st}$ and $w=\mathbf{y}_{t}^{-1}\mathbf{x}_{t}$. Note that%
\begin{equation*}
\left\Vert \mathbf{y}_{t}^{-1}\mathbf{x}_{t}\right\Vert =\left\Vert \mathbf{x%
}_{t}^{-1}\mathbf{y}_{t}\right\Vert =d\left( \mathbf{x}_{t},\mathbf{y}%
_{t}\right) 
\end{equation*}%
and $\left\Vert v\right\Vert ,\left\Vert w\right\Vert \leq \left\Vert 
\mathbf{x}\right\Vert _{\infty }+\left\Vert \mathbf{y}\right\Vert _{\infty }$%
. The conclusion now follows from Lemma \ref{ghg}.
\end{proof}

We recall the following simple interpolation result \cite%
{friz-victoir-04-Note}.

\begin{lemma}
\label{interpolation}For $0\,\leq \alpha ^{\prime }<\alpha \leq 1$ there
exists a constant $C=C\left( \alpha ,\alpha ^{\prime }\right) $ such that%
\begin{equation*}
d_{\alpha ^{\prime }-H\ddot{o}l}\left( \mathbf{x},\mathbf{y}\right) \leq
C\left( \left\Vert \mathbf{x}\right\Vert _{\alpha -H\ddot{o}l}\vee
\left\Vert \mathbf{y}\right\Vert _{\alpha -H\ddot{o}l}\right) ^{\alpha
^{\prime }/\alpha }d_{0}\left( \mathbf{x},\mathbf{y}\right) ^{1-\alpha
^{\prime }/\alpha }.
\end{equation*}%
Similarly, for $1\leq p<p^{\prime }<\infty $ there exists $C=C\left(
p,p^{\prime }\right) $ such that%
\begin{equation*}
d_{p^{\prime }\text{$-var$}}\left( \mathbf{x},\mathbf{y}\right) \leq C\left(
\left\Vert \mathbf{x}\right\Vert _{p\text{$-var$}}\vee \left\Vert \mathbf{y}%
\right\Vert _{p\text{$-var$}}\right) ^{p/p^{\prime }}d_{0}\left( \mathbf{x},%
\mathbf{y}\right) ^{1-p/p^{\prime }}.
\end{equation*}
\end{lemma}

\begin{corollary}[$L^{q}$-convergence in rough path metrics]
Let $\mathbf{X}^{n},\mathbf{X}^{\infty }$ be continuous $G^{N}\left( \mathbb{%
R}^{d}\right) $-valued process defined on $\left[ 0,1\right] $. Let $q\in
\lbrack 1,\infty )$ and assume that for some $\alpha \in (0,1],$ (resp. $%
p\geq 1$),%
\begin{equation}
\sup_{1\leq n\leq \infty }\mathbb{E}\left( \left\Vert \mathbf{X}%
^{n}\right\Vert _{\alpha -H\ddot{o}l}^{q}\right) <\infty \text{ (resp. \ }%
\sup_{1\leq n\leq \infty }\mathbb{E}\left( \left\Vert \mathbf{X}%
^{n}\right\Vert _{p\text{$-var$}}^{q}\right) <\infty \text{ )}
\label{UniformHoelderPvarBounds}
\end{equation}%
and that we have uniform convergence in $L^{q}\left( \mathbb{P}\right) $
i.e. 
\begin{equation}
d_{\infty }\left( \mathbf{X}^{n},\mathbf{X}^{\infty }\right) \rightarrow 0%
\text{ in }L^{q}\left( \mathbb{P}\right) .  \label{dInfinityLqConvAss}
\end{equation}%
Then $d_{\alpha ^{\prime }-H\ddot{o}l}\left( \mathbf{X}^{n},\mathbf{X}%
^{\infty }\right) $ for $\alpha ^{\prime }<\alpha $, (resp. $d_{p^{\prime }%
\text{$-var$}}\left( \mathbf{X}^{n},\mathbf{X}^{\infty }\right) $ and $%
p^{\prime }>p$), converges to zero in $L^{q}\left( \mathbb{P}\right) .$
\end{corollary}

\begin{proof}
From the $d_{0}/d_{\infty }$ estimate there exists $c_{1}>0$ such that%
\begin{equation*}
\frac{1}{c_{1}}d_{0}\left( \mathbf{X}^{n},\mathbf{X}^{\infty }\right) \leq
d_{\infty }\left( \mathbf{X}^{n},\mathbf{X}^{\infty }\right) +d_{\infty
}\left( \mathbf{X}^{n},\mathbf{X}^{\infty }\right) ^{1/N}\left( \left\Vert 
\mathbf{X}^{n}\right\Vert _{\infty }+\left\Vert \mathbf{X}^{\infty
}\right\Vert _{\infty }\right) ^{1-1/N}
\end{equation*}%
and so%
\begin{eqnarray*}
\frac{1}{c_{1}}\mathbb{E}\left( d_{0}\left( \mathbf{X}^{n},\mathbf{X}%
^{\infty }\right) ^{q}\right) ^{1/q} &\leq &\mathbb{E}\left( d_{\infty
}\left( \mathbf{X}^{n},\mathbf{X}^{\infty }\right) ^{q}\right) ^{1/q} \\
&&+\mathbb{E}\left[ d_{\infty }\left( \mathbf{X}^{n},\mathbf{X}^{\infty
}\right) ^{q/N}\left( \left\Vert \mathbf{X}^{n}\right\Vert _{\infty
}+\left\Vert \mathbf{X}^{\infty }\right\Vert _{\infty }\right) ^{q\left(
1-1/N\right) }\right] ^{1/q}.
\end{eqnarray*}%
By H\"{o}lder's inequality, 
\begin{equation*}
\mathbb{E}\left[ d_{\infty }\left( \mathbf{X}^{n},\mathbf{X}\right)
^{q/N}\left\Vert \mathbf{X}^{n}\right\Vert _{\infty }{}^{q\left(
1-1/N\right) }\right] \leq \mathbb{E}\left[ d_{\infty }\left( \mathbf{X}^{n},%
\mathbf{X}\right) ^{q}\right] ^{\frac{1}{N}}\mathbb{E}\left[ \left\Vert 
\mathbf{X}^{n}\right\Vert _{\infty }{}^{q}\right] ^{\left( 1-\frac{1}{N}%
\right) }.
\end{equation*}%
Since $\left\Vert \cdot \right\Vert _{\infty }$ is dominated by H\"{o}lder-
and variation norms, assumption (\ref{UniformHoelderPvarBounds}) is
plentiful to bound $\mathbb{E}\left( \left\Vert \mathbf{X}^{n}\right\Vert
_{\infty }{}^{q}\right) $ uniformly in $n$. We thus obtain convergence of $%
d_{0}\left( \mathbf{X}^{n},\mathbf{X}^{\infty }\right) $ to $0$ in $L^{q}.$
An almost identical application of H\"{o}lder's inequality, now using Lemma %
\ref{interpolation} instead of the $d_{0}/d_{\infty }$ estimate, shows that $%
d_{\alpha ^{\prime }-H\ddot{o}l}\left( \mathbf{X}^{n},\mathbf{X}^{\infty
}\right) $, resp. $d_{p^{\prime }\text{$-var$}}\left( \mathbf{X}^{n},\mathbf{%
X}^{\infty }\right) $, converges to zero in $L^{q}\left( \mathbb{P}\right) .$
\end{proof}

The assumption (\ref{dInfinityLqConvAss}) can often be weakened to pointwise
convergence.

\begin{corollary}
\label{pointwiseHolderConvergenceLemma}Let $\mathbf{X}^{n},\mathbf{X}%
^{\infty }$ be continuous $G^{N}\left( \mathbb{R}^{d}\right) $-valued
process defined on $\left[ 0,1\right] $. Let $q\in \lbrack 1,\infty )$ and
assume that we have pointwise convergence in $L^{q}\left( \mathbb{P}\right) $
i.e. for all $t\in \left[ 0,1\right] $,%
\begin{equation}
d\left( \mathbf{X}_{t}^{n},\mathbf{X}_{t}^{\infty }\right) \rightarrow 0%
\text{ in }L^{q}\left( \mathbb{P}\right) \text{ as }n\rightarrow \infty ;
\label{ptwiseLq0Convergenc}
\end{equation}%
and uniform H\"{o}lder bounds, i.e. 
\begin{equation*}
\sup_{1\leq n\leq \infty }\mathbb{E}\left( \left\Vert \mathbf{X}%
^{n}\right\Vert _{\alpha -H\ddot{o}l}^{q}\right) <\infty
\end{equation*}%
then for $\alpha ^{\prime }<\alpha ,$ 
\begin{equation*}
d_{\alpha ^{\prime }-H\ddot{o}l}\left( \mathbf{X}^{n},\mathbf{X}^{\infty
}\right) \rightarrow 0\text{ in }L^{q}\left( \mathbb{P}\right) .
\end{equation*}
\end{corollary}

\begin{proof}
From the previous corollary, we only need to show $d_{\infty }$-convergence
in $L^{q}$. For any integer $m,$ 
\begin{eqnarray*}
2^{1-q}\mathbb{E}\left[ d_{\infty }\left( \mathbf{X}^{n},\mathbf{X}^{\infty
}\right) ^{q}\right] &\leq &\mathbb{E}\left[ \sup_{i=1,...,m}d\left( \mathbf{%
X}_{\frac{i}{m}}^{n},\mathbf{X}_{\frac{i}{m}}^{\infty }\right) ^{q}\right] +%
\mathbb{E}\left[ \sup_{\left\vert t-s\right\vert <\frac{1}{m}}\left(
\left\Vert \mathbf{X}_{s,t}^{n}\right\Vert ^{q}+\left\Vert \mathbf{X}%
_{s,t}^{\infty }\right\Vert ^{q}\right) \right] \\
&\leq &\sum_{i=1}^{m}\mathbb{E}\left( d\left( \mathbf{X}_{\frac{i}{m}}^{n},%
\mathbf{X}_{\frac{i}{m}}^{\infty }\right) ^{q}\right) +\left( \frac{1}{m}%
\right) ^{\alpha q}\times 2\sup_{1\leq n\leq \infty }\mathbb{E}\left[
\left\Vert \mathbf{X}^{n}\right\Vert _{\alpha -H\ddot{o}l}^{q}\right] .
\end{eqnarray*}%
By first choosing $m$ large enough, followed by choosing $n$ large enough we
see that $d_{\infty }\left( \mathbf{X}^{n},\mathbf{X}^{\infty }\right)
\rightarrow 0$ in $L^{q}$ as required.
\end{proof}

\begin{corollary}
\label{pointwisePvarConvergenceLemma}Let $\mathbf{X}^{n},\mathbf{X}^{\infty
} $ be continuous $G^{N}\left( \mathbb{R}^{d}\right) $-valued process
defined on $\left[ 0,1\right] $. Let $q\in \lbrack 1,\infty )$ and assume
that we have pointwise convergence in $L^{q}\left( \mathbb{P}\right) $ i.e.
for all $t\in \left[ 0,1\right] $,%
\begin{equation*}
d\left( \mathbf{X}_{t}^{n},\mathbf{X}_{t}^{\infty }\right) \rightarrow 0%
\text{ in }L^{q}\left( \mathbb{P}\right) \text{ as }n\rightarrow \infty ;
\end{equation*}%
uniform $p$-variation bounds,%
\begin{equation}
\sup_{1\leq n\leq \infty }\mathbb{E}\left( \left\Vert \mathbf{X}%
^{n}\right\Vert _{p\text{$-var$}}^{q}\right) <\infty
\label{pointwisePvarConvergenceLemma_Ass2}
\end{equation}%
\newline
\underline{and} a tightness condition%
\begin{equation}
\lim_{\varepsilon \rightarrow 0}\sup_{n}\mathbb{E}\left( \left\vert
osc\left( \mathbf{X}^{n},\varepsilon \right) \right\vert ^{q}\right) =0.
\label{pointwisePvarConvergenceLemma_Ass3}
\end{equation}%
where $osc\left( \mathbf{X},\varepsilon \right) \equiv \sup_{\left\vert
t-s\right\vert <\varepsilon }\left\Vert \mathbf{X}_{s,t}\right\Vert $, then 
\begin{equation*}
d_{p\text{$-var$}}\left( \mathbf{X}^{n},\mathbf{X}^{\infty }\right)
\rightarrow 0\text{ in }L^{q}\left( \mathbb{P}\right) .
\end{equation*}%
Conditions (\ref{pointwisePvarConvergenceLemma_Ass2}) and (\ref%
{pointwisePvarConvergenceLemma_Ass3}) are implied by a Kolmogorov type
tightness criterion: there exists a 1D control function $\omega $ and a real
number $\theta \geq \frac{1}{2q}+\frac{1}{p}$ such that for all $s<t$ in $%
\left[ 0,1\right] ,$ 
\begin{equation}
\sup_{1\leq n\leq \infty }\mathbb{E}\left( \left\vert d\left( \mathbf{X}%
_{s}^{n},\mathbf{X}_{t}^{n}\right) \right\vert ^{q}\right) ^{1/q}\leq \omega
\left( s,t\right) ^{\theta }.
\label{pointwisePvarConvergenceLemma_AssKolmTight}
\end{equation}
\end{corollary}

\begin{proof}
From our criterion for $L^{q}$-convergence in rough path metrics, we only
need to show $d_{\infty }$-convergence in $L^{q},$ which is an obvious
consequence of the inequality%
\begin{equation*}
d_{\infty }\left( \mathbf{X}^{n},\mathbf{X}^{\infty }\right) \leq osc\left( 
\mathbf{X}^{n},1/m\right) +osc\left( \mathbf{X},1/m\right)
+\sup_{i=1,...,m}d\left( \mathbf{X}_{\frac{i}{m}}^{n},\mathbf{X}_{\frac{i}{m}%
}^{\infty }\right) .
\end{equation*}%
Finally, the assumption (\ref{pointwisePvarConvergenceLemma_AssKolmTight})
implies (\ref{pointwisePvarConvergenceLemma_Ass2}) and (\ref%
{pointwisePvarConvergenceLemma_Ass3}) as an application of Corollary \ref%
{GRRBesovEmbLq}. (The bound on $\theta $ comes from $q_{0}=\left(
1/r-1/p\right) ^{-1}/2.$)
\end{proof}

\begin{remark}
One cannot get rid of the tightness condition. Consider $f_{n}\left(
t\right) =0$ on $\left[ 0,1/n\right] $ and a triangle peak of height on $%
\left[ 1/n,1\right] $. Clearly, $f_{n}\left( t\right) \rightarrow 0$ a.s.
(and hence in measure) and the $1$-variation of $\left\{ f_{n}\right\} $ is
uniformly bounded. Yet, $f_{n}\nrightarrow 0$ in any variation topology
which is stronger than the uniform topology.
\end{remark}

\section{Appendix II: Garsia Rodemich Rumsey}

Similarly to the last appendix, but more quantitatively, we aim for
conditions under which $G^{N}\left( \mathbb{R}^{d}\right) $-valued processes
are close in H\"{o}lder- (resp. variation-)/$L^{q}\left( \mathbb{P}\right) $
sense. When possible, we formulate regularity results in the more general
setting of paths (or processes) with values in a Polish space $\left(
E,d\right) $.

\begin{theorem}[Garsia Rodemich Rumsey]
\label{GRR}Let $\Psi $ and $p$ be continuous strictly increasing functions
on $[0,\infty )$ with $p(0)=\Psi (0)=0$ and $\Psi (x)\rightarrow \infty $ as 
$x\rightarrow \infty $. Given $f\in C\left( [0,1],E\right) $, if 
\begin{equation}
\int_{0}^{1}\int_{0}^{1}\Psi \left( \frac{d\left( f_{s},f_{t}\right) }{%
p(\left\vert t-s\right\vert )}\right) dsdt\leq F,  \label{GRR_inequ}
\end{equation}%
then for $0\leq s<t\leq 1,$%
\begin{equation*}
d\left( f_{s},f_{t}\right) \leq 8\int_{0}^{t-s}\Psi ^{-1}\left( \frac{4F}{%
u^{2}}\right) dp(u).
\end{equation*}%
In particular, if $osc\left( f,\delta \right) \equiv \sup_{\left\vert
t-s\right\vert \leq \delta }d\left( f_{s},f_{t}\right) $ denotes the modulus
of continuity of $f$, we have%
\begin{equation*}
osc\left( f,\delta \right) \leq 8\int_{0}^{\delta }\Psi ^{-1}\left( \frac{4F%
}{u^{2}}\right) dp(u).
\end{equation*}
\end{theorem}

\begin{corollary}
\label{GRRBesovEmb}Let $r\geq 1$ and $\alpha \in \lbrack 0,1/r)$. Then, for
any fixed $q\geq q_{0}\left( r,\alpha \right) $, 
\begin{equation*}
\text{ }\int_{0}^{1}\int_{0}^{1}\frac{d\left( f_{s},f_{t}\right) ^{q}}{%
\left\vert t-s\right\vert ^{q/r}}dsdt\leq M^{q},
\end{equation*}%
implies the existence of $C=C\left( r,\alpha \right) $ such that $osc\left(
f,\delta \right) \leq C\delta ^{\alpha }M$ and%
\begin{equation*}
\left\Vert f\right\Vert _{\alpha -H\ddot{o}l\text{;}\left[ 0,1\right] }\leq
CM.
\end{equation*}
\end{corollary}

\begin{proof}
From Garsia-Rodemich-Rumsey with $\Psi \left( x\right) =x^{q},$ $p\left(
u\right) =u^{1/r}$ \ and $F=M^{q}$ it follows that%
\begin{eqnarray*}
d\left( f_{s},f_{t}\right) &\leq &8\left( 4F\right)
^{1/q}\int_{0}^{t-s}u^{-2/q+1/r-1}du \\
&=&\frac{8\left( 4F\right) ^{1/q}}{1/r-2/q}\left\vert t-s\right\vert
^{1/r-2/q} \\
&\leq &\frac{32M}{1/2r}\left\vert t-s\right\vert ^{\alpha }
\end{eqnarray*}%
provided $q$ is large enough so that $0\leq \alpha <1/r-2/q$ and $%
1/r-2/q>1/(2r)$. Both statements follow. (One can take $q_{0}=\left(
1/r-\alpha \right) ^{-1}/2\vee 4r$ and $C=64/r$. Alternatively, at least if $%
\alpha >0$, one can take $q_{0}=\left( 1/r-\alpha \right) ^{-1}/2$ and $%
C=32/\alpha $.)
\end{proof}

\begin{corollary}
\label{GRRBesovEmbLq}Let $\omega $ be 1D control function and $X$ a
continuous $\left( E,d\right) $-valued stochastic process defined on $\left[
0,1\right] $. Assume $r\geq 1$ and $1/p\in \lbrack 0,1/r)$. Then, for any
fixed $q\geq q_{0}\left( r,p\right) $ (one can take $q_{0}=\left(
1/r-1/p\right) ^{-1}/2$) 
\begin{equation*}
\left\vert d\left( X_{s},X_{t}\right) \right\vert _{L^{q}\left( \mathbb{P}%
\right) }\leq M\omega \left( s,t\right) ^{1/r}\text{ for all }s,t\in \left[
0,1\right]
\end{equation*}%
implies $osc\left( X,\delta \right) \rightarrow 0$ in $L^{q}\left( \mathbb{P}%
\right) $ as $\delta \rightarrow 0$ and there exists $C=C\left( r,p\right) $
such that 
\begin{equation*}
\left\vert \left\Vert X\right\Vert _{p-var\text{;}\left[ 0,1\right]
}\right\vert _{L^{q}\left( \mathbb{P}\right) }\leq CM.\text{ }
\end{equation*}%
If $\omega \left( s,t\right) \leq t-s$ for all $s,t\in \left[ 0,1\right] $
then $\left\Vert X\right\Vert _{p-var\text{;}\left[ 0,1\right] }$ above can
be replaced by $\left\Vert X\right\Vert _{1/p-H\ddot{o}l\text{;}\left[ 0,1%
\right] }.$
\end{corollary}

\begin{proof}
We first consider the case of H\"{o}lder dominated control $\omega \left(
s,t\right) \leq t-s$. From the preceding corollary%
\begin{equation*}
\left\Vert X\right\Vert _{\alpha -H\ddot{o}l\text{;}\left[ 0,1\right]
}^{q}\leq C^{q}\int_{0}^{1}\int_{0}^{1}\frac{d\left( X_{s},X_{t}\right) ^{q}%
}{\left\vert t-s\right\vert ^{q/r}}dsdt
\end{equation*}%
and taking expectations gives%
\begin{equation*}
\mathbb{E}\left( \left\Vert X\right\Vert _{\alpha -H\ddot{o}l\text{;}\left[
0,1\right] }^{q}\right) \leq C^{q}\int_{0}^{1}\int_{0}^{1}\frac{\mathbb{E}%
\left( d\left( X_{s},X_{t}\right) ^{q}\right) }{\left\vert t-s\right\vert
^{q/r}}dsdt\leq \left( CM\right) ^{q}
\end{equation*}%
which shows $\left\vert \left\Vert X\right\Vert _{p-H\ddot{o}l\text{;}\left[
0,1\right] }\right\vert _{L^{q}\left( \mathbb{P}\right) }\leq CM$. The
statement on $osc\left( X,\delta \right) $ obvious. We now discuss a general
control $\omega $. At the price of replacing $M$ by $M\omega \left(
0,1\right) ^{1/r},$ we assume $\omega \left( 0,1\right) =1.$ The function $%
\omega \left( t\right) :=\omega \left( 0,t\right) $ maps $\left[ 0,1\right] $
continuously and increasingly onto $\left[ 0,1\right] $ and there exists a
continuous process $Y$ such that $Y_{\omega \left( t\right) }=X_{t}$ for all 
$t\in \left[ 0,1\right] $. We then have, for all $s,t\in \left[ 0,1\right] ,$
\begin{equation*}
\mathbb{E}\left( d\left( Y_{s},Y_{t}\right) ^{q}\right) ^{1/q}\leq
M\left\vert t-s\right\vert ^{1/r}.
\end{equation*}%
By the H\"{o}lder case just discussed, $osc\left( Y,\delta \right)
\rightarrow 0$ (in $L^{q}$) and so $osc\left( X,\delta \right) \rightarrow 0$
in $L^{q}$ by (uniform) continuity of $\omega $. The H\"{o}lder case also
takes care of the $L^{q}$-bound of $\left\Vert X\right\Vert _{p-var\text{;}%
\left[ 0,1\right] }$, it suffices to note 
\begin{equation*}
\left\Vert X\right\Vert _{p-var\text{;}\left[ 0,1\right] }=\left\Vert
Y\right\Vert _{p-var\text{;}\left[ 0,1\right] }\leq \left\Vert Y\right\Vert
_{\frac{1}{p}\text{$-H\ddot{o}l$;}\left[ 0,1\right] }.
\end{equation*}
\end{proof}

We now consider paths with values in $G^{N}\left( \mathbb{R}^{d}\right) $
for which we can of increments, $x_{s,t}\equiv x_{s}^{-1}\otimes x_{t}$, and
thus of H\"{o}lder- and variation distance.

\begin{corollary}
\label{BesovDistance}Let $r\geq 1$ and $\alpha \in \lbrack 0,1/r)$. Then,
for any $q\geq q_{0}\left( r,\alpha \right) $ and $M>0,\delta \in \left(
0,1\right) ,$ 
\begin{eqnarray*}
\int_{0}^{1}\int_{0}^{1}\frac{d\left( x_{s},x_{t}\right) ^{q}}{\left\vert
t-s\right\vert ^{q/r}}dsdt, &\leq &M^{q}, \\
\int_{0}^{1}\int_{0}^{1}\frac{d\left( y_{s},y_{t}\right) ^{q}}{\left\vert
t-s\right\vert ^{q/r}}dsdt &\leq &M^{q}, \\
\int_{0}^{1}\int_{0}^{1}\frac{d\left( x_{s,t},y_{s,t}\right) ^{q}}{%
\left\vert t-s\right\vert ^{q/r}}dsdt &\leq &\left( \delta M\right) ^{q},
\end{eqnarray*}%
implies the existence of $C=C\left( r;N,d\right) ,\theta =\theta \left(
r,\alpha ;N\right) >0$ such that 
\begin{equation*}
d_{\alpha \text{$-H\ddot{o}l$;}\left[ 0,1\right] }\left( x,y\right) \leq
C\delta ^{\theta }M.
\end{equation*}
\end{corollary}

\begin{proof}
We first note that with $\alpha ^{\prime }=\left( \alpha +1/r\right) /2$ and
assuming $q\geq q_{0}\left( r,\alpha \right) $ large enough, Corollary \ref%
{GRRBesovEmb} implies%
\begin{equation}
\left\Vert x\right\Vert _{0;\left[ 0,1\right] }\leq \left\Vert x\right\Vert
_{\alpha -H\ddot{o}l;\left[ 0,1\right] }\leq \left\Vert x\right\Vert
_{\alpha ^{\prime }-H\ddot{o}l;\left[ 0,1\right] }\leq c_{0}M
\label{BesovDistProofEquOne}
\end{equation}%
and the same estimate holds for $y$. Let us define $z_{t}=y_{t}\otimes
x_{t}^{-1}$. Since $z_{s,t}\equiv z_{s}^{-1}\otimes z_{t}=x_{t}\otimes
\left( x_{s,t}^{-1}\otimes y_{s,t}\right) \otimes x_{t}^{-1}$ Lemma \ref{ghg}
gives%
\begin{equation*}
\frac{1}{c_{1}}\left\Vert z_{s,t}\right\Vert \leq d\left(
x_{s,t},y_{s,t}\right) \vee d\left( x_{s,t},y_{s,t}\right) ^{\frac{1}{N}%
}\left\Vert x_{t}\right\Vert ^{1-\frac{1}{N}}.
\end{equation*}%
Dividing by $\left\vert t-s\right\vert ^{1/rN}$ and raising everything to
power $q$ yields%
\begin{equation*}
\left( \frac{1}{c_{1}}\frac{\left\Vert z_{s,t}\right\Vert }{\left\vert
t-s\right\vert ^{\frac{1}{rN}}}\right) ^{q}\leq \left( \frac{d\left(
x_{s,t},y_{s,t}\right) }{\left\vert t-s\right\vert ^{\frac{1}{r}}}\right)
^{q}\vee \left( \frac{d\left( x_{s,t},y_{s,t}\right) ^{q}}{\left\vert
t-s\right\vert ^{\frac{q}{r}}}\right) ^{\frac{1}{N}}\left\Vert
x_{t}\right\Vert ^{q\left( 1-\frac{1}{N}\right) }
\end{equation*}%
and after integration over $\left( \,s,t\right) \in \left[ 0,1\right] ^{2}$,
using H\"{o}lder's inequality on the last term, we arrive at%
\begin{eqnarray*}
&&\left( \frac{1}{c_{1}}\right) ^{q}\int_{0}^{1}\int_{0}^{1}\left( \frac{%
\left\Vert z_{s,t}\right\Vert }{\left\vert t-s\right\vert ^{\frac{1}{rN}}}%
\right) ^{q}dsdt \\
&\leq &\int_{0}^{1}\int_{0}^{1}\left( \frac{d\left( x_{s,t},y_{s,t}\right) }{%
\left\vert t-s\right\vert ^{\frac{1}{r}}}\right) ^{q}dsdt \\
&&+\left( \int_{0}^{1}\int_{0}^{1}\left( \frac{d\left(
x_{s,t},y_{s,t}\right) ^{q}}{\left\vert t-s\right\vert ^{\frac{q}{r}}}%
\right) dsdt\right) ^{\frac{1}{N}}\left( \int_{0}^{1}\int_{0}^{1}\left\Vert
x_{t}\right\Vert ^{q}dsdt\right) ^{1-\frac{1}{N}} \\
&\leq &\left( \delta M\right) ^{q}+\left( \delta M\right) ^{q/N}\left(
\left\Vert x\right\Vert _{0-H\ddot{o}l;\left[ 0,1\right] }\right) ^{q\left(
1-\frac{1}{N}\right) } \\
&\leq &\left( \delta M\right) ^{q}+\left( \delta M\right) ^{q/N}\left(
c_{0}M\right) ^{q\left( 1-\frac{1}{N}\right) }\text{ \ \ \ by (\ref%
{BesovDistProofEquOne})} \\
&\leq &(c_{2}\delta ^{\frac{1}{N}}M)^{q}.
\end{eqnarray*}%
We can then apply Corollary \ref{GRRBesovEmb} to $z$ (with $M$ replaced by $%
c_{1}c_{2}\delta ^{\frac{1}{N}}M$) to see that $\left\Vert z\right\Vert _{0;%
\left[ 0,1\right] }\leq c_{3}\delta ^{\frac{1}{N}}M$. On the other hand, $%
d\left( x_{s,t},y_{s,t}\right) =\left\Vert x_{t}^{-1}\otimes z_{s,t}\otimes
x_{t}\right\Vert $ and Lemma \ref{ghg} implies, again using (\ref%
{BesovDistProofEquOne}), 
\begin{equation*}
d_{0}\left( x,y\right) \leq c_{1}\max \left\{ \left\Vert z\right\Vert _{0-H%
\ddot{o}l},\left\Vert z\right\Vert _{0-H\ddot{o}l}^{\frac{1}{N}}.\left\Vert
x\right\Vert _{0-H\ddot{o}l}^{1-\frac{1}{N}}\right\} \leq c_{4}\delta ^{%
\frac{1}{N^{2}}}M.
\end{equation*}%
We now use interpolation, Lemma \ref{interpolation}, with H\"{o}lder
exponents \thinspace $\alpha <\alpha ^{\prime }$. For $c_{5}=c_{5}\left(
\alpha ,r\right) $ and again using (\ref{BesovDistProofEquOne}) we have 
\begin{eqnarray*}
d_{a-H\ddot{o}l}\left( x,y\right) &\leq &c_{5}\left( \left\Vert x\right\Vert
_{\alpha ^{\prime }-H\ddot{o}l}\vee \left\Vert y\right\Vert _{\alpha
^{\prime }-H\ddot{o}l}\right) ^{\frac{\alpha }{\alpha ^{\prime }}%
}d_{0}\left( x,y\right) ^{1-\frac{\alpha }{\alpha ^{\prime }}} \\
&\leq &c_{5}\left( c_{0}M\right) ^{\frac{\alpha }{\alpha ^{\prime }}}\left(
c_{4}\delta ^{\frac{1}{N^{2}}}M\right) ^{1-\frac{\alpha }{\alpha ^{\prime }}}
\\
&=&c_{6}M\delta ^{\theta }\text{ \ with }\theta =\theta \left( \alpha
,r,N\right) :=\frac{\alpha ^{\prime }-\alpha }{\alpha ^{\prime }N^{2}}.
\end{eqnarray*}%
The proof is finished.
\end{proof}

\begin{corollary}
\label{BesovDistanceLq}Assume that $X,Y$ are continuous $G^{N}\left( \mathbb{%
R}^{d}\right) $-valued processes defined on $\left[ 0,1\right] $. Assume $%
r\geq 1$ and $1/p\in \lbrack 0,1/r)$. Then, for any fixed $q\geq q_{0}\left(
r,p\right) \,$and $M,\delta \in \left( 0,1\right) ,$ 
\begin{eqnarray*}
\mathbb{E}\left( d\left( X_{s},X_{t}\right) ^{q}\right) &\leq &\left(
M\omega \left( s,t\right) ^{1/r}\right) ^{q} \\
\mathbb{E}\left( d\left( Y_{s},Y_{t}\right) ^{q}\right) &\leq &\left(
M\omega \left( s,t\right) ^{1/r}\right) ^{q} \\
\mathbb{E}\left( d\left( X_{s,t},Y_{s,t}\right) ^{q}\right) &\leq &\left(
\delta M\omega \left( s,t\right) ^{1/r}\right) ^{q}
\end{eqnarray*}%
implies the existence of $C=C\left( r;N,d\right) ,\theta =\theta \left(
r,p;N\right) >0$ such that%
\begin{equation*}
\left\vert d_{p-var\text{;}\left[ 0,1\right] }\left( X,Y\right) \right\vert
_{L^{q}\left( \mathbb{P}\right) }\leq C\delta ^{\theta }M.
\end{equation*}%
If $\omega \left( s,t\right) \leq t-s$ for all $s,t\in \left[ 0,1\right] $
then $d_{p-var\text{;}\left[ 0,1\right] }$ above can be replaced by $d_{1/p%
\text{$-H\ddot{o}l$;}\left[ 0,1\right] }.$
\end{corollary}

\begin{proof}
By a (deterministic) time-change argument, exactly as in the proof of
Corollary \ref{GRRBesovEmbLq}, we may assume $\omega \left( s,t\right) =t-s$%
. From Corollary \ref{BesovDistance} there exists $c_{1}=c_{1}\left(
r;N,d\right) ,\theta =\theta \left( r,p;N\right) $ such that for $q\geq
q_{0}\left( r,p\right) $ large enough 
\begin{eqnarray*}
\left( \frac{1}{c_{1}\delta ^{\theta }}d_{1/p\text{$-H\ddot{o}l$;}\left[ 0,1%
\right] }\left( X,Y\right) \right) ^{q} &\leq &\int_{0}^{1}\int_{0}^{1}\frac{%
d\left( x_{s},x_{t}\right) ^{q}}{\left\vert t-s\right\vert ^{q/r}}%
dsdt+\int_{0}^{1}\int_{0}^{1}\frac{d\left( y_{s},y_{t}\right) ^{q}}{%
\left\vert t-s\right\vert ^{q/r}}dsdt \\
&&+\left( \frac{1}{\delta ^{q}}\int_{0}^{1}\int_{0}^{1}\frac{d\left(
x_{s,t},y_{s,t}\right) ^{q}}{\left\vert t-s\right\vert ^{q/r}}dsdt\right) .
\end{eqnarray*}%
After taking expectations we see that $\left( c_{1}\delta ^{\theta }\right)
^{-q}\mathbb{E}\left( d_{1/p\text{$-H\ddot{o}l$;}\left[ 0,1\right] }\left(
X,Y\right) ^{q}\right) \leq 3M^{q}$ and the proof is easily finished. (One
can take $C=3c_{1}$).
\end{proof}

\section{Appendix III: Step-3 Lie Algebra}

As usual, $e_{1},...,e_{d}$ denotes the standard basis in $\mathbb{R}^{d}$.
A vector space basis of the Lie algebra $g_{2}\left( \mathbb{R}^{d}\right) $
is given by%
\begin{equation*}
\left\{ \left( e_{i}\right) ,\left( \left[ e_{i},e_{j}\right] \right)
_{i<j}\right\}
\end{equation*}%
and if $x:\left[ s,t\right] \rightarrow \mathbb{R}^{d}$ be a smooth path
with signature $S\left( x\right) _{s,t}=$ $\mathbf{x}_{s,t}$ then its
log-signature satisfies, trivially,%
\begin{equation*}
\pi _{1}\left( \ln \mathbf{x}_{s,t}\right) =\sum_{i}x_{s,t}^{i}e_{i}\in 
\mathbb{R}^{d}
\end{equation*}%
and%
\begin{equation*}
\pi _{2}\left( \ln \mathbf{x}_{s,t}\right) =\frac{1}{2}\sum_{i<j}\left( 
\mathbf{x}_{s,t}^{i,j}-\mathbf{x}_{s,t}^{j,i}\right) \left[ e_{i},e_{j}%
\right] \in so\left( d\right) .
\end{equation*}%
We aim for a similar understanding of $g_{3}\left( \mathbb{R}^{d}\right) $.
We leave the following simple technical lemma to the reader:

\begin{lemma}[Step-3\ Hall expansion]
A vector space basis of the Lie algebra $g_{3}\left( \mathbb{R}^{d}\right) $
is given by%
\begin{equation*}
\left\{ \left( e_{i}\right) ,\left( \left[ e_{i},e_{j}\right] \right)
_{i<j},\left( \left[ e_{i},\left[ e_{j},e_{k}\right] \right] \right) _{j\leq
i<k\text{ or }j<k\leq i}\right\}
\end{equation*}%
for $i,j,k\in \left\{ 1,...,d\right\} $, known as Philip-Hall Lie basis. For
any $3$-tensor $\alpha $, the following identity holds: 
\begin{eqnarray*}
\sum_{i,j,k}\alpha _{i,j,k}\left[ e_{i},\left[ e_{j},e_{k}\right] \right]
&=&\sum_{\substack{ j<i<k  \\ \text{or}  \\ j<k<i}}\left( \alpha
_{i,j,k}-\alpha _{i,k,j}+\alpha _{j,i,k}-\alpha _{j,k,i}\right) \left[ e_{i},%
\left[ e_{j},e_{k}\right] \right] \\
&&+\sum_{i\neq j}\left( \alpha _{i,i,j}-\alpha _{i,j,i}\right) \left[ e_{i},%
\left[ e_{i},e_{j}\right] \right] .
\end{eqnarray*}
\end{lemma}

\begin{proposition}
\label{logSignatureStep3}Let $x:\left[ s,t\right] \rightarrow \mathbb{R}^{d}$
be a smooth path with lift $S\left( x\right) =$ $\mathbf{x}$. Then its
log-signature projected to the third level, $\pi _{3}\left( \ln \mathbf{x}%
_{s,t}\right) $, expands to in the Hall-basis as follows. 
\begin{eqnarray*}
\pi _{3}\left( \ln \mathbf{x}_{s,t}\right) &=&\frac{1}{6}\sum_{\substack{ %
j<i<k  \\ \text{or}  \\ j<k<i}}\left( \mathbf{x}_{s,t}^{i,j,k}+\mathbf{x}%
_{s,t}^{j,i,k}-2\mathbf{x}_{s,t}^{i,k,j}+\mathbf{x}_{s,t}^{k,i,j}-2\mathbf{x}%
_{s,t}^{j,k,i}+\mathbf{x}_{s,t}^{k,j,i}\right) \left[ e_{i},\left[
e_{j},e_{k}\right] \right] \\
&&+\sum_{i\neq j}\left\{ \mathbf{x}_{s,t}^{i,i,j}+\frac{1}{12}\left\vert
x_{s,t}^{i}\right\vert ^{2}x_{s,t}^{j}-\frac{1}{2}x_{s,t}^{i}\mathbf{x}%
_{s,t}^{i,j}\right\} \left[ e_{i},\left[ e_{i},e_{j}\right] \right]
\end{eqnarray*}%
\bigskip This identity remains valid for (weak) geometric rough paths.
\end{proposition}

\begin{proof}
Without loss of generalities $x:\left[ 0,1\right] \rightarrow \mathbb{R}^{d}$
and $x\left( 0\right) =0$. The signature of concatenated paths is given by
the group product in the free group. Specializing to the step-$3$ group
(viewed as subset of the enveloping tensor algebra), the smooth path $%
x=\left. x\right\vert _{\left[ 0,t+dt\right] }$ is the concatenation of $%
\left. x\right\vert _{\left[ 0,t\right] }$ and $\left. x\right\vert _{\left[
t,t+dt\right] }$. We have $S\left( x_{t+dt}\right) =S\left( x_{t}\right)
\otimes \exp \left( dx_{t}\right) $ and by sending $dt\rightarrow 0$ it is
easy to (re-)derive the usual control ODE for lifted paths%
\begin{equation*}
d\mathbf{x}_{t}=U_{i}\left( \mathbf{x}\right) dx^{i}
\end{equation*}%
where $\mathbf{x}_{t}=S\left( x\right) _{t}$ and $U_{i}\left( \mathbf{x}%
\right) =\mathbf{x}\otimes e_{i}$. To understand the evolution in the step-$%
3 $ Lie algebra we write%
\begin{equation*}
\mathbf{z}_{i}\left( t\right) =\pi _{i}\left( \ln S\left( x\right)
_{t}\right) \text{, }i=1,2,3
\end{equation*}%
and using the Baker Campbell Hausdorff formula, we obtain \footnote{%
A recursion formula for $\mathbf{z}_{n}$ appears in Chen's seminal work \cite%
{chen-57}.},%
\begin{eqnarray*}
d\mathbf{z}_{1}\left( t\right) &=&dx_{t} \\
d\mathbf{z}_{2}\left( t\right) &=&\frac{1}{2}\left[ \mathbf{z}_{1}\left(
t\right) ,dx_{t}\right] , \\
d\mathbf{z}_{3}\left( t\right) &=&\frac{1}{2}\left[ \mathbf{z}_{2}\left(
t\right) ,dx_{t}\right] +\frac{1}{12}\left[ \mathbf{z}_{1}\left( t\right) ,%
\left[ \mathbf{z}_{1}\left( t\right) ,dx_{t}\right] \right] ,
\end{eqnarray*}%
which integrates iteratively to 
\begin{eqnarray*}
\mathbf{z}_{1}\left( t\right) &=&x_{t} \\
\mathbf{z}_{2}\left( t\right) &=&\frac{1}{2}\int_{0<u<v<t}\left[
dx_{u},dx_{v}\right] \\
\mathbf{z}_{3}\left( t\right) &=&\frac{1}{4}\int_{0<u<v<w<t}\left[ \left[
dx_{u},dx_{v}\right] ,dx_{w}\right] +\frac{1}{12}\int_{0<u<t}\left[ x_{u},%
\left[ x_{u},dx_{u}\right] \right] .
\end{eqnarray*}%
In particular, the log-signature of $x$ projected to the third level is
precisely $\mathbf{z}_{3}\left( 1\right) $ and given by%
\begin{eqnarray*}
&&\frac{1}{4}\int_{0<u<v<w<t}\left[ \left[ dx_{u},dx_{v}\right] ,dx_{w}%
\right] +\frac{1}{12}\int_{0<u<t}\left[ x_{u},\left[ x_{u},dx_{u}\right] %
\right] \\
&=&\frac{1}{4}\sum_{i,j,k}\mathbf{x}^{i,j,k}\left[ \left[ e_{i},e_{j}\right]
,e_{k}\right] +\frac{1}{12}\sum_{i,j,k}\left( \mathbf{x}^{i,j,k}+\mathbf{x}%
^{j,i,k}\right) \left[ e_{i},\left[ e_{j},e_{k}\right] \right] \\
&=&\frac{1}{12}\sum_{i,j,k}\left( -3\mathbf{x}^{j,k,i}+\mathbf{x}^{i,j,k}+%
\mathbf{x}^{j,i,k}\right) \left[ e_{i},\left[ e_{j},e_{k}\right] \right] .
\end{eqnarray*}%
Using the\bigskip\ step-$3$ Hall expansion lemma, a few lines of
computations give%
\begin{eqnarray*}
6\mathbf{z}_{3}\left( 1\right) &=&\sum_{\substack{ j<i<k  \\ or  \\ j<k<i}}%
\left( \mathbf{x}_{t}^{i,j,k}+\mathbf{x}_{t}^{j,i,k}-2\mathbf{x}_{t}^{i,k,j}+%
\mathbf{x}_{t}^{k,i,j}-2\mathbf{x}_{t}^{j,k,i}+\mathbf{x}_{t}^{k,j,i}\right) %
\left[ e_{i},\left[ e_{j},e_{k}\right] \right] \\
&&+\sum_{i\neq j}\left( -2\mathbf{x}_{t}^{i,j,i}+\mathbf{x}_{t}^{i,i,j}+%
\mathbf{x}_{t}^{j,i,i}\right) \left[ e_{i},\left[ e_{i},e_{j}\right] \right]
,
\end{eqnarray*}%
When $x$ is defined on $\left[ s,t\right] $ the last expression is, of
course,%
\begin{equation*}
\frac{1}{6}\left( \mathbf{x}_{s,t}^{i,i,j}-2\mathbf{x}_{s,t}^{i,j,i}+\mathbf{%
x}_{s,t}^{j,i,i}\right)
\end{equation*}%
and can be simplified with some calculus. We have 
\begin{eqnarray*}
\mathbf{x}_{s,t}^{i,i,j} &=&\frac{1}{2}\int_{s}^{t}\left\vert
x_{s,u}^{i}\right\vert ^{2}dx_{u}^{j} \\
\mathbf{x}_{s,t}^{i,j,i}+\mathbf{x}_{s,t}^{j,i,i}
&=&\int_{s<u<t}x_{s,u}^{i}x_{s,u}^{j}dx_{u}^{i} \\
&=&\frac{1}{2}\left\vert x_{s,t}^{i}\right\vert ^{2}x_{s,t}^{j}-\frac{1}{2}%
\int_{s<u<t}\left\vert x_{s,u}^{i}\right\vert ^{2}dx_{s,u}^{j}\text{ (by
integration by part),} \\
\mathbf{x}_{s,t}^{j,i,i}
&=&\int_{s<u_{1}<u_{2}<u_{3}<t}dx_{u_{1}}^{j}dx_{u_{2}}^{i}dx_{u_{3}}^{i} \\
&=&\frac{1}{2}\int_{s<u<t}\left\vert x_{u,t}^{i}\right\vert ^{2}dx_{u}^{j} \\
&=&\frac{1}{2}\left\vert x_{s,t}^{i}\right\vert
^{2}x_{s,t}^{j}-x_{s,t}^{i}\int_{s<u<t}x_{s,u}^{i}dx_{s,u}^{j}+\frac{1}{2}%
\int_{s}^{t}\left\vert x_{s,u}^{i}\right\vert ^{2}dx_{u}^{j}.
\end{eqnarray*}%
and therefore 
\begin{eqnarray*}
\mathbf{x}_{s,t}^{i,i,j}-2\mathbf{x}_{s,t}^{i,j,i}+\mathbf{x}_{s,t}^{j,i,i}
&=&\mathbf{x}_{s,t}^{i,i,j}-2\left( \mathbf{x}_{s,t}^{i,j,i}+\mathbf{x}%
_{s,t}^{j,i,i}\right) +3\mathbf{x}_{s,t}^{j,i,i} \\
&=&3\int_{s}^{t}\left\vert x_{s,u}^{i}\right\vert ^{2}dx_{u}^{j}+\frac{1}{2}%
\left\vert x_{s,t}^{i}\right\vert ^{2}x_{s,t}^{j}-3x_{s,t}^{i}\mathbf{x}%
_{s,t}^{i,j}.
\end{eqnarray*}%
For the final statement, it suffices to remark that a (weak) geometric rough
path is, in particular, a pointwise limit of smooth paths.
\end{proof}

An obvious application of the Baker Campbell Hausdorff formula gives

\begin{lemma}
\label{BCHlevel3normbound}Let $a,b$ be two elements of the Lie algebra $%
g_{3}\left( \mathbb{R}^{d}\right) $ and write $a^{i}=\pi _{i}\left( a\right)
,$ $b_{i}=\pi _{i}\left( b\right) $. Then there exists $C=C\left( d\right) $
such that%
\begin{eqnarray*}
\left\vert \pi _{2}\left( \ln \left( e^{-a}\otimes e^{b}\right) \right)
\right\vert &\leq &\left\vert b^{2}-a^{2}\right\vert +\frac{1}{2}\left\vert
b^{1}-a^{1}\right\vert \left\vert b^{1}\right\vert \\
\left\vert \pi _{3}\left( \ln \left( e^{-a}\otimes e^{b}\right) \right)
\right\vert &\leq &\left\vert b^{3}-a^{3}\right\vert +\frac{1}{2}\left\vert
b^{2}-a^{2}\right\vert \left\vert b^{1}\right\vert +\left\vert
b^{1}-a^{1}\right\vert \left( \frac{1}{2}\left\vert b^{2}\right\vert +\frac{1%
}{12}\left\vert a^{1}\right\vert ^{2}+\frac{1}{12}\left\vert
b^{1}\right\vert ^{2}\right) \text{.}
\end{eqnarray*}
\end{lemma}

\bibliographystyle{plain}
\bibliography{roughpaths}

\bigskip

\end{document}